\documentclass[11pt]{article}
\usepackage[a4paper,margin=1in]{geometry}
\usepackage{amsmath,amssymb,amsthm,mathtools}
\usepackage{mathrsfs}
\usepackage{enumitem}
\usepackage{microtype}
\usepackage[hidelinks]{hyperref}
\usepackage{aliascnt}
\usepackage[nameinlink,noabbrev,capitalise]{cleveref}
\usepackage{array}
\newtheorem{theorem}{Theorem}[section]
\newaliascnt{proposition}{theorem}
\newtheorem{proposition}[proposition]{Proposition}
\aliascntresetthe{proposition}
\newaliascnt{lemma}{theorem}
\newtheorem{lemma}[lemma]{Lemma}
\aliascntresetthe{lemma}
\newaliascnt{corollary}{theorem}
\newtheorem{corollary}[corollary]{Corollary}
\aliascntresetthe{corollary}
\theoremstyle{definition}
\newaliascnt{definition}{theorem}
\newtheorem{definition}[definition]{Definition}
\aliascntresetthe{definition}
\newaliascnt{example}{theorem}

\aliascntresetthe{example}
\theoremstyle{remark}
\newaliascnt{remark}{theorem}
\newtheorem{remark}[remark]{Remark}
\aliascntresetthe{remark}

\newcommand{\kk}{k}
\newcommand{\F}{\mathbb F}
\newcommand{\Ext}{\operatorname{Ext}}

\newcommand{\Hom}{\operatorname{Hom}}

\newcommand{\Der}{\operatorname{Der}}

\newcommand{\Jac}{\operatorname{Jac}}
\newcommand{\rad}{\operatorname{rad}}

\newcommand{\im}{\operatorname{im}}
\newcommand{\coker}{\operatorname{coker}}

\newcommand{\uenv}{\mathrm u}
\newcommand{\Witt}{\mathrm W}
\newcommand{\Spec}{\mathrm S}
\newcommand{\Ham}{\mathrm H}
\newcommand{\Cont}{\mathrm K}
\newcommand{\Res}{\operatorname{Res}}
\newcommand{\Ind}{\operatorname{Ind}}

\newcommand{\soc}{\operatorname{soc}}

\title{Extensions of Restricted Simple Modules for Restricted Lie Algebras of Cartan Type}
\author{Chao Ma\thanks{Independent Researcher, London, UK; email: \texttt{raymond.ma@me.com}; ORCID: 0009-0004-2456-9098.}}
\date{\emph{Dedicated to the memory of Georgia Benkart}}

\hypersetup{
  pdftitle={Extensions of Restricted Simple Modules for Restricted Lie Algebras of Cartan Type},
  pdfauthor={Chao Ma},
  pdfcreator={Chao Ma},
  pdfsubject={Extensions between restricted simple modules for the restricted simple Lie algebras of Cartan type, and the resulting Ext-quivers},
  pdfkeywords={extension, Ext quiver, restricted simple module, restricted Lie algebra, Cartan type, Witt-Jacobson algebra, Special Lie algebra, Hamiltonian Lie algebra, contact Lie algebra, de Rham complex, Cartier operator, parabolically induced module, exceptional weight, modular representation theory}
}

\begin{document}
\maketitle

\begin{abstract}
We determine the first extension groups between simple restricted modules
for the restricted simple Lie algebras of Cartan type over an algebraically
closed field of characteristic $p\ge5$, which answers the restricted case
of a problem of Benkart and Feldvoss.  Extensions from parabolically induced
modules are computed by finite restricted-cohomology complexes of the
grading parabolic, and for the exceptional modules we obtain closed formulas
in all four families.  Neither the trivial module nor any nontrivial
exceptional simple has self-extensions.  In characteristic seven we
determine the quiver of $S(3;1)^{(1)}$, which has $49$ vertices and $566$
arrows.
\end{abstract}

\section{Introduction and main results}

Let
\[
\Witt(n;\mathbf1),\qquad
\Spec(n;\mathbf1)^{(1)},\qquad
\Ham(2m;\mathbf1)^{(2)},\qquad
\Cont(2m+1;\mathbf1)^{(1)}
\]
denote the restricted simple Lie algebras of Cartan type.  Adopt
the low-rank convention
\[
 \Spec(2;\mathbf1)^{(2)}=\Ham(2;\mathbf1)^{(2)}.
\tag{1.0}\label{eq:low-rank-SH-identification}
\]
The defining divergence and Hamiltonian conditions in two variables give
this equality, as well as equality of the corresponding derived series.
We treat this algebra as Special and assume $m\ge2$ in the Hamiltonian
family.  For a
restricted Lie algebra $\mathfrak g$, write
\[
\uenv(\mathfrak g)=\uenv_0(\mathfrak g)=
U(\mathfrak g)/
\langle x^p-x^{[p]}:x\in\mathfrak g\rangle.
\]

Benkart and Feldvoss \cite[Problem~4]{BenkartFeldvoss} asked for the
$\Ext^1$ groups and $\Ext^1$ quivers of restricted simple Cartan-type Lie
algebras.  The present paper treats the restricted case $\chi=0$ in
characteristic $p\ge5$.  Each extension group between two simple restricted modules is
realized as the first cohomology of one of the finite complexes constructed
below.
Throughout,
``restricted module'' means a module for
$\uenv(\mathfrak g)=\uenv_0(\mathfrak g)$.

Boe, Nakano, and Wiesner determined the restricted $W(1;1)$ quiver
\cite{BoeNakanoWiesner}.  Lin and Nakano developed a
Hopf-algebra method for self-extensions in the Cartan-type setting
\cite{LinNakano}.  Let $L_r$ be the last nonzero term of the filtration
used to define the height $h(\chi)$ of a $p$-character $\chi$
\cite[(2.4)]{BenkartFeldvoss}.  For
arbitrary $p$-characters of $W(1;1)$, Boe, Nakano, and Wiesner also
determined the extensions when $h(\chi)\in\{-1,0,1\}$ and when
$\chi(L_r)\ne0$, and for the remaining characters gave a lower bound for
the dimension of the self-extensions of the unique simple module
\cite[\S4]{BoeNakanoWiesner}.  Rian later obtained the first two cases by
different methods \cite{RianThesis,RianApplications}.  In general, simple modules are either
parabolically induced or exceptional.
Restricted Shapiro identifies extensions from an induced module with
restricted cohomology for the grading parabolic.  The exceptional modules
are treated by the de Rham, Holmes, and Rumin-type \cite{Rumin} form
resolutions for the respective families.

The argument uses Shen's graded module constructions
\cite{ShenI,ShenII,ShenIII}, Nakano's projective, block-theoretic and
Witt--Jacobson cohomology results
\cite{NakanoProjective,NakanoWittCohomology}, the Witt--Jacobson theory of
Feldvoss--Nakano \cite{FeldvossNakanoWitt}, and Jantzen's first-cohomology
results for classical Levi algebras \cite{JantzenH1}.  The Contact case uses
Hu's graded-module and constituent theory
\cite{HuContactGraded,HuContactConstituents}, together with Skryabin's
derivation-theoretic results \cite{SkryabinDerivations}.  The
Hamiltonian contraction in \cref{sec:H-zero} adapts a Cartier-decomposition
technique from the author's study of divergence-free polynomial vector
fields \cite{MachaoCartierKoszul}.

\paragraph{Main results.}
Let $p\ge5$ and let $\mathfrak g$ be a restricted simple Lie algebra of
Cartan type $\Witt$, $\Spec$, $\Ham$, or $\Cont$, with $m\ge2$ in the
Hamiltonian case.

\begin{enumerate}[label=\textup{(\roman*)}]
\item For Witt--Jacobson algebras, the groups with parabolically induced source are computed by
the restricted parabolic complex of \cref{thm:SW-pair}.  For exceptional
vertices, \cref{thm:SW-exceptional-closed} shows that every ordered pair of
nontrivial vertices has zero $\Ext^1$; the two orientations incident with the
trivial vertex are given by binomial coefficients indexed by the de Rham degree.

\item For Special algebras, the same parabolic and de Rham constructions
apply.  The exceptional subquiver on nontrivial vertices vanishes, and the
extensions to and from the trivial module are the binomial formulas of
\cref{thm:SW-exceptional-closed}.  In characteristic seven,
\cref{thm:S3-p7-quiver} determines the $49$-vertex quiver of
$\Spec(3;\mathbf1)^{(1)}$.

\item For $\Ham(2m;\mathbf1)^{(2)}$, the positive Hamiltonian algebra is
generated by its cubic layer.  For a parabolically induced source, the parabolic complex
is concentrated in degrees $0$, $1$, and $2$, with Levi and positive Spencer
summands as in \eqref{eq:H-standard-graded-split} and
\cref{thm:H-zero-pair}.  The exceptional subquiver is the Adamovich
direct-cover graph of \cref{thm:H-exceptional-quiver}, together with the
symmetric formulas of \cref{thm:H-trivial-exceptional} for extensions to
and from the trivial module.

\item For the Contact family, \cref{thm:contact-vertex-parametrization}
parametrizes the simple modules by Holmes' parabolic heads.  Extensions with
nonexceptional source are the cohomology of the Contact parabolic complex,
and duality treats the reverse mixed pairs.  For exceptional pairs,
\cref{thm:contact-exceptional-closed} gives the matrix for all $p\ge5$:
nontrivial arrows are multiplicity-free Adamovich covers subject to the
grading-centre congruence.  Higher multiplicities occur only among
extensions incident with the trivial module.
\end{enumerate}

These constructions determine
$\Ext^1_{\uenv_0(\mathfrak g)}(L,L')$ for simple restricted modules
$L,L'$.  In particular, \cref{cor:self-extensions} records the resulting
self-extension vanishing for the trivial and exceptional simples, and
\cref{prop:small-rank-exceptional} writes out the exceptional subquiver of
every restricted simple Cartan-type algebra of rank one or two.  For
instance, the two nontrivial exceptional simples $M_1,M_2$ of
$\Ham(4;\mathbf1)^{(2)}$ are not joined by an arrow, while
$\dim\Ext^1(\kk,M_1)=\dim\Ext^1(M_1,\kk)=6$ and
$\dim\Ext^1(\kk,M_2)=\dim\Ext^1(M_2,\kk)=10$.  For
nonexceptional labels, the degree-zero term is the classical Levi extension group,
and the positive-degree contribution is computed by the family complexes
below.

\begin{remark}[Degree-zero Levi summands]
\label{rem:Levi-summands}
For every nonexceptional pair, the internal degree-zero strand
contains the classical direct summand
\[
 \Ext^1_{u_0(\mathfrak g_0)}(U_\lambda,U_\mu),
 \qquad
 \mathfrak g_0=\mathfrak{gl}_n,\ \mathfrak{sl}_n,\
 \mathfrak{sp}_{2m},\ \text{or }\mathfrak{sp}_{2m}\oplus\kk z,
\]
because a degree-zero nilradical cochain would have to land in negative PBW
depth.  These type-$A$/type-$C$ Frobenius-kernel extensions are represented
by the classical root matrices of \cref{lem:classical-Levi-root-complex}.
\end{remark}

\section{Restricted derivations and parabolic induction}
\label{sec:finite-algebra}

\subsection{The restricted derivation complex}

Let $\mathfrak l$ be a finite-dimensional restricted Lie algebra and $E$ a
finite-dimensional restricted $\mathfrak l$-module.  A restricted derivation
$f:\mathfrak l\to E$ satisfies
\begin{align}
 f([x,y])&=x\cdot f(y)-y\cdot f(x),\label{eq:der-Lie}\\
 f(x^{[p]})&=x^{p-1}\cdot f(x).\label{eq:der-p}
\end{align}
\begin{lemma}[$p$-semilinearity of the restricted defect]\label{lem:p-semilinear}
Let $f:\mathfrak l\to E$ be $\kk$-linear and satisfy the ordinary Leibniz
identity \eqref{eq:der-Lie}.  Define
$\varphi_f(x):=f(x^{[p]})-x^{p-1}f(x)$.
Then $\varphi_f$ is Frobenius-semilinear:
\[
 \varphi_f(x+y)=\varphi_f(x)+\varphi_f(y),\qquad
 \varphi_f(\lambda x)=\lambda^p\varphi_f(x).
\]
It follows that \eqref{eq:der-p} may be checked on a basis of
$\mathfrak l$.
\end{lemma}

\begin{proof}
Form the restricted semidirect product
$\widetilde{\mathfrak l}=\mathfrak l\ltimes E$, where $E$ is an abelian
ideal and
\[
(x+e)^{[p]}=x^{[p]}+x^{p-1}\cdot e
\qquad(x\in\mathfrak l, e\in E).
\]
The map $\sigma:\mathfrak l\to\widetilde{\mathfrak l}$,
$\sigma(x)=x+f(x)$, is a Lie algebra homomorphism because $f$
satisfies \eqref{eq:der-Lie}.  Its restricted defect is
\[
\sigma(x^{[p]})-\sigma(x)^{[p]}=\varphi_f(x)\in E.
\]
Apply Jacobson's formula to $x+y$ in $\mathfrak l$ and to
$\sigma(x)+\sigma(y)$ in $\widetilde{\mathfrak l}$.  Since every Jacobson
term $s_i$ is an iterated bracket and $\sigma$ preserves brackets, one has
$\sigma(s_i(x,y))=s_i(\sigma(x),\sigma(y))$; these terms cancel pairwise.
The remaining terms give
$\varphi_f(x+y)=\varphi_f(x)+\varphi_f(y)$.  Applying the same comparison to
$\lambda x$ gives
$\varphi_f(\lambda x)=\lambda^p\varphi_f(x)$.
\end{proof}

After choosing bases, \eqref{eq:der-Lie}--\eqref{eq:der-p} define a finite
matrix $\mathcal R_{\mathfrak l,E}$.  By \cref{lem:p-semilinear}, its
kernel is the space of restricted derivations.  Inner derivations form the
image of
\[
\delta_0:E\longrightarrow\Hom_\kk(\mathfrak l,E),
\qquad
\delta_0(e)(x)=x\cdot e.
\]

Fix a basis $x_1,\ldots,x_n$ of $\mathfrak l$ and let
\[
\mathcal C^{(2,p)}_{\mathfrak l,E}
:=\Hom_\kk(\textstyle\bigwedge^2\mathfrak l,E)
  \oplus\bigoplus_{i=1}^n E_{(p),i}
\]
be the target of $\mathcal R_{\mathfrak l,E}$, with the first summand
recording the Leibniz defect $f([x,y])-x\cdot f(y)+y\cdot f(x)$ of
\eqref{eq:der-Lie} and the $i$th copy $E_{(p),i}$ recording the basis value
\[
\varphi_f(x_i)=f(x_i^{[p]})-x_i^{p-1}f(x_i).
\]
The map $\mathcal R_{\mathfrak l,E}$ is $\kk$-linear.  On the kernel of
its ordinary Leibniz rows,
\cref{lem:p-semilinear} shows that the second block is equivalently the
Frobenius-semilinear restricted defect on all of $\mathfrak l$.

\begin{proposition}[Restricted derivation complex]\label{prop:root-p}
The complex
\[
E\xrightarrow{\delta_0}
\Hom_\kk(\mathfrak l,E)
\xrightarrow{\mathcal R_{\mathfrak l,E}}
\mathcal C^{(2,p)}_{\mathfrak l,E}
\]
has first cohomology
\[
H_*^1(\mathfrak l,E)
=\ker\mathcal R_{\mathfrak l,E}/\im\delta_0.
\]
Every matrix entry is a polynomial in the restricted structure constants
and module action matrices.
In particular,
\[
\dim H_*^1(\mathfrak l,E)
=\operatorname{nullity}\mathcal R_{\mathfrak l,E}
-\dim E+\dim E^{\mathfrak l}.
\tag{2.0}
\]
A row-reduced complement to $\operatorname{im}\delta_0$ inside
$\ker\mathcal R_{\mathfrak l,E}$ yields explicit cocycle representatives
when they are required.

Moreover, let $L,L'$ be finite-dimensional restricted
$\uenv(\mathfrak l)$-modules, with action matrices
$\rho,\rho'$.  On
$E=\Hom_\kk(L,L')$ put
\[
x\cdot\varphi=\rho'(x)\varphi-\varphi\rho(x).
\]
$E$ is a restricted $\mathfrak l$-module, and there is a canonical isomorphism
\[
\Ext^1_{\uenv(\mathfrak l)}(L,L')
\cong H_*^1(\mathfrak l,E).
\tag{2.0a}\label{eq:restricted-rootp}
\]
In matrices the $p$-row for $f:\mathfrak l\to E$ is
\[
f(x^{[p]})-
\sum_{i=0}^{p-1}\rho'(x)^{p-1-i}f(x)\rho(x)^i=0.
\tag{2.0b}\label{eq:restricted-prow}
\]
\end{proposition}

\begin{proof}
The ordinary rows are \eqref{eq:der-Lie}, the restricted rows are
\eqref{eq:der-p}, and the Jacobson identity shows that every inner
derivation satisfies both.  Conversely, the common kernel is the
space of restricted derivations.  Quotienting by the inner image gives
restricted cohomology in degree one; see \cite{Hochschild}.

For \eqref{eq:restricted-rootp}, write an extension action in block form
\[
\widetilde\rho(x)=
\begin{pmatrix}
\rho'(x)&f(x)\\0&\rho(x)
\end{pmatrix}.
\]
The Lie relation gives the ordinary derivation row for the conjugation
module.  The upper-right block of $\widetilde\rho(x)^p$ is
\[
\sum_{i=0}^{p-1}\rho'(x)^{p-1-i}f(x)\rho(x)^i,
\]
whereas the restricted relation
$\widetilde\rho(x)^p=\widetilde\rho(x^{[p]})$
is \eqref{eq:restricted-prow}.  Since left and right multiplication
commute and
$(-1)^i\binom{p-1}{i}=1$ in characteristic $p$, this is also
$f(x^{[p]})=x^{p-1}\cdot f(x)$.  Changing a splitting of the extension adds
an inner derivation.  This proves \eqref{eq:restricted-rootp}.
\end{proof}

\begin{lemma}[Socle filtrations of quotient modules]
\label{lem:socle-filtered-receiver}
Let \(A\) be a finite-dimensional algebra, let \(S\) be a finite-length
\(A\)-module, and let \(C\subseteq S\) be a submodule.  Put
\(M=S/C\) and \(K=C^\perp\subseteq S^\vee\).  Then, for every \(a\ge0\),
\[
 \operatorname{Soc}^{a}K
 =K\cap\operatorname{Soc}^{a}S^\vee.
\tag{2.0c}\label{eq:socle-receiver-intersection}
\]
If \(f:S^\vee\to C^\vee\) is restriction and
\(F_a=f(\operatorname{Soc}^{a}S^\vee)\), then \(f\) is strict for these
filtrations and there are canonical exact sequences
\[
 0\longrightarrow
 \frac{\operatorname{Soc}^{a}K}{\operatorname{Soc}^{a-1}K}
 \longrightarrow
 \frac{\operatorname{Soc}^{a}S^\vee}
      {\operatorname{Soc}^{a-1}S^\vee}
 \longrightarrow
 \frac{F_a}{F_{a-1}}
 \longrightarrow0.
\tag{2.0d}\label{eq:socle-receiver-layers}
\]
In particular,
\[
 (\operatorname{hd}M)^\vee
 \cong\operatorname{soc}K
 =\ker\!\left(
 \operatorname{soc}S^\vee\longrightarrow C^\vee
 \right).
\tag{2.0e}\label{eq:socle-receiver-head}
\]
After decomposing a socle layer as
\(\bigoplus_\lambda L(\lambda)\otimes M_{\lambda,a}\), the induced map is
\(\bigoplus_\lambda 1_{L(\lambda)}\otimes A_{\lambda,a}\).  Thus every
layer of \(K\), and in particular the head of \(M\), is obtained from the
kernels of the maps \(A_{\lambda,a}\) on simple multiplicity spaces.
\end{lemma}

\begin{proof}
For a submodule \(K\subseteq S^\vee\), the identity
\(\operatorname{soc}K=K\cap\operatorname{soc}S^\vee\) and induction on
\(a\) give \eqref{eq:socle-receiver-intersection}.  Indeed, after
quotienting by the common \((a-1)\)-st term, the assertion is again the
socle identity for a submodule.  If \(x\in\operatorname{Soc}^{a}S^\vee\)
has image in \(F_{a-1}\), choose
\(y\in\operatorname{Soc}^{a-1}S^\vee\) with \(f(x)=f(y)\).  Then
\(x-y\in K\cap\operatorname{Soc}^{a}S^\vee\), and
\eqref{eq:socle-receiver-layers} follows.  Since \(M^\vee=C^\perp=K\)
and duality interchanges head and socle,
\eqref{eq:socle-receiver-head} is the case \(a=1\).  The final assertion
is Schur's lemma applied to the semisimple layers.
\end{proof}

\subsection{Parabolic induction}

\begin{proposition}[Restricted PBW--Shapiro isomorphism]\label{prop:Shapiro}
Let $\mathfrak p\subseteq\mathfrak g$ be a restricted subalgebra compatible
with the restricted structure, put
$B=\uenv(\mathfrak p)$ and $A=\uenv(\mathfrak g)$, and let $U$ be
a finite $B$-module.  Restricted PBW makes $A$ free as a right $B$-module,
so for every $A$-module $M$,
\[
\Ext_A^n(A\otimes_B U,M)
\cong
\Ext_B^n(U,\Res_B^A M).
\tag{2.6}
\]
For $M=A\otimes_BV$, the right side is computed by the restricted pair
complex with coefficient module
\[
E_{U,V}=\Hom_\kk\bigl(U,\Res_B^A(A\otimes_BV)\bigr).
\tag{2.7}
\]
\end{proposition}

\begin{proof}
The PBW basis makes induction exact and sends projectives to projectives.
Apply the usual induction--restriction adjunction to a projective
resolution of $U$.  The tensor-Hom identification converts the resulting
cochains into restricted derivations with coefficients (2.7).
\end{proof}

\begin{lemma}[Invariant top of a simple graded parabolically induced module]\label{lem:invariant-top-standard}
Let
\[
\mathfrak g=\mathfrak g_-\oplus\mathfrak g_0\oplus N,
\qquad P=\mathfrak g_0\ltimes N,
\]
be a compatible restricted grading, let $U$ be a restricted
$\mathfrak g_0$-module inflated to $P$, and suppose the parabolically induced
module
\[
Z(U)=u_0(\mathfrak g)\otimes_{u_0(P)}U
\]
is simple.  Give $Z(U)$ its weighted negative-PBW depth grading, with
$Z(U)_0=U$.  Then
\[
Z(U)^N=U.
\]
\end{lemma}

\begin{proof}
Because every positive homogeneous element of $N$ lowers PBW depth, the
subspace $Z(U)^N$ is graded.  Hence, if it contains a nonzero vector outside
$U$, it contains a nonzero homogeneous invariant $v$ of some depth $d>0$.
Restricted PBW gives
\[
u_0(\mathfrak g)=\operatorname{span}
\bigl(u_0(\mathfrak g_-)u_0(\mathfrak g_0)u_0(N)\bigr).
\]
The augmentation ideal of $u_0(N)$ annihilates $v$.  Thus $u_0(\mathfrak g)v$
is generated from $v$ using only $\mathfrak g_-$ and $\mathfrak g_0$.
The zero part preserves PBW depth and every negative homogeneous generator
strictly increases it, so
\[
u_0(\mathfrak g)v\subseteq\bigoplus_{e\ge d}Z(U)_e.
\]
This is a nonzero submodule which cannot meet the inducing top $Z(U)_0=U$,
hence is proper, contradicting the simplicity of $Z(U)$.  Therefore no
positive-depth $N$-invariant exists.
\end{proof}

By \cref{prop:Shapiro}, the parabolic cocycle complex computes extensions
from an induced module.  The exceptional modules are treated by the form
resolutions associated with the individual families.

\begin{lemma}[Toral normalization of the classical Levi term]
\label{lem:classical-Levi-root-complex}
Let $\mathfrak l$ be one of the split classical Levi algebras
$\mathfrak{sl}_n$, $\mathfrak{gl}_n$, or $\mathfrak{sp}_{2m}$ occurring
below, possibly with a split central torus.  In the $\mathfrak{gl}_n$ case
take the full diagonal torus, also when $p\mid n$.  Fix a split maximal
torus $\mathfrak t$, Chevalley root vectors
$x_\alpha$ $(\alpha\in\Phi)$, and restricted simples
$S=L(\lambda)$, $T=L(\mu)$.  Put
\[
 E=\Hom_\kk(S,T),\qquad E_0=\text{the zero torus-weight space of }E,
\]
and, writing $E_{\bar\alpha}$ for the restricted torus-weight space of weight
$\alpha$, set
\[
 C^1_{\lambda\mu}:=\bigoplus_{\alpha\in\Phi}E_{\bar\alpha}.
\tag{2.7a}\label{eq:classical-C1}
\]
Every nonzero torus-weight cocycle is inner, so restricted cohomology is
concentrated in the zero torus-weight complex below.
For $a=(a_\alpha)_{\alpha\in\Phi}\in C^1_{\lambda\mu}$ let $D_a$ be the
unique torally normalized degree-one cochain with
$D_a(x_\alpha)=a_\alpha$.  Define the root-pair and restricted rows by
\[
 \mathcal R_{\alpha,\beta}(a)
 =x_\alpha a_\beta-x_\beta a_\alpha-D_a([x_\alpha,x_\beta]),
\tag{2.7b}\label{eq:classical-R}
\]
\[
 \mathcal P_\alpha(a)=x_\alpha^{p-1}a_\alpha,
\tag{2.7c}\label{eq:classical-P}
\]
and let
\[
 \delta_0:E_0\longrightarrow C^1_{\lambda\mu},\qquad
 \delta_0(A)=(x_\alpha A)_{\alpha\in\Phi}.
\tag{2.7d}\label{eq:classical-delta}
\]
Then
\[
 \Ext^1_{u_0(\mathfrak l)}(S,T)
 \cong
 \ker(\mathcal R\oplus\mathcal P)/\delta_0(E_0).
\tag{2.7e}\label{eq:classical-Levi-rootp}
\]
If a split toral center is present, the extension group vanishes for
unequal central characters; for equal central characters the same formula
applies with the full torus.
\end{lemma}

\begin{proof}
Let $f_\gamma$ be the component of a cocycle of nonzero torus weight
$\gamma$.  Choose $h\in\mathfrak t$ with $\gamma(h)\ne0$.  The cocycle
identity on $[h,x]$ gives
\[
 f_\gamma(x)=x\cdot\bigl(\gamma(h)^{-1}f_\gamma(h)\bigr),
\]
so $f_\gamma$ is inner.  Hence only the zero-weight component remains.
Toral normalization of a restricted cocycle removes all non-root coordinates.
Indeed, for $h\in\mathfrak t$ the weight-zero cocycle value $D(h)\in E_0$
is killed by $h$, so the restricted identity gives
$D(h^{[p]})=h^{p-1}D(h)=0$.  The $p$-map is surjective on the split torus,
so subtracting an inner derivation gives $D|_{\mathfrak t}=0$.  The full
diagonal torus gives the same conclusion for $\mathfrak{gl}_n$ when
$p\mid n$.  The value on $x_\alpha$ lies in $E_{\bar\alpha}$.  The
ordinary cocycle equations on pairs of root vectors are
\eqref{eq:classical-R}; the restricted condition on a root vector is
\eqref{eq:classical-P}.  Hence the restricted degree-one complex is
\[
 E_0\xrightarrow{\delta_0}C^1_{\lambda\mu}
 \xrightarrow{\mathcal R\oplus\mathcal P}C^2_{\mathrm{root},p},
\]
which gives \eqref{eq:classical-Levi-rootp}.  A central torus separates
unequal characters; for equal characters its zero-weight part is included
in $E_0$.
\end{proof}

\section{Witt--Jacobson and Special algebras}
\label{sec:SW}

The classification of simple restricted modules and their induced
realizations is due to \cite{HolmesZhang,HolmesWitt,ZhangGraded}.  It
separates the simple modules into parabolically induced and exceptional classes.

Concretely, let $\Lambda_L$ be the finite set of $p$-restricted dominant
highest weights of the semisimple part of $L=\mathfrak g_0$, together with
the permitted $\F_p$-characters of its toral center.  For every
$\lambda\in\Lambda_L$, the integral Weyl-to-dual-Weyl image construction
gives the restricted simple $L_0(\lambda)$, whose parabolically induced module
is obtained by PBW commutation.  At an exceptional weight, the
corresponding simple module is the kernel, image, or quotient specified by
the classification.

\subsection{Exceptional de Rham models for Witt--Jacobson and Special algebras}

The exceptional restricted modules have the following de Rham
realizations.  Let
$\Omega^j=A(n;\mathbf1)\otimes\Lambda^j(\kk^n)^*$,
$d_j:\Omega^j\to\Omega^{j+1}$ be the divided-power de Rham differential,
and $B^j=\im d_{j-1}$.  In the Witt--Jacobson family the exceptional vertices are
\[
M_n=\kk,\qquad M_a=B^{a+1}\quad(0\le a<n),
\tag{3.0a}
\]
where the de Rham map identifies $M_0=B^1\cong A/\kk$.
In the Special family let
$\int:A\to\kk$ be the coefficient of the top monomial and
$A^\circ=\ker\int$.  The exceptional vertices are
\[
L_0=\kk,\qquad L_1=A^\circ/\kk,\qquad L_j=B^j
\quad(2\le j\le n-1).
\tag{3.0b}
\]
The matrices are obtained by restricting the polynomial-vector-field
action to these kernels and images.  The classification theorems
in \cite{HolmesZhang,HolmesWitt,ZhangGraded} identify these modules with
the exceptional highest weights and prove that, together with the simple
parabolically induced modules, they are pairwise nonisomorphic and exhaustive.

Let $P=L\ltimes N$ be the grading parabolic and let $E$ be the
finite-dimensional restricted coefficient module.  A pair cocycle is
$(\alpha,\psi)$ with
\[
\alpha:L\to E,\qquad \psi:N\to E.
\]
Its ordinary rows are
\begin{align}
\alpha([x,y])&=x\alpha(y)-y\alpha(x),\label{eq:LL}\\
\psi([n,n'])&=n\psi(n')-n'\psi(n),\label{eq:NN}\\
\psi([x,n])&=x\psi(n)-n\alpha(x),\label{eq:LN}
\end{align}
and its restricted rows are
\begin{align}
\alpha(x^{[p]})&=x^{p-1}\alpha(x),\label{eq:Lp}\\
\psi(n^{[p]})&=n^{p-1}\psi(n).\label{eq:Np}
\end{align}
Denote this matrix by $\Gamma^{\mathrm{pair}}_{P,E}$.

\begin{lemma}[Split-torus concentration of restricted first cohomology]
\label{lem:SW-torus-zero}
Let \(\mathfrak a\) be a Witt--Jacobson or Special algebra, or its grading parabolic
\(P=L\ltimes N\), and let
\(\mathfrak t\subseteq L\) be the split diagonal torus.  For every finite
restricted \(\mathfrak a\)-module \(E\),
\[
 H_*^1(\mathfrak a,E)=H_*^1(\mathfrak a,E)_0
\]
with respect to the $\mathfrak t$-weight decomposition.  Equivalently every
restricted derivation of nonzero torus weight is inner.
\end{lemma}

\begin{proof}
Let $d:P\to E$ be a homogeneous restricted derivation of torus weight
$\chi\ne0$.  Choose $h\in\mathfrak t$ with $\chi(h)\ne0$.  For every
homogeneous $x\in P$, the derivation identity and the weight relation give
\[
 (h\cdot d)(x)=h\,d(x)-d([h,x])=x\,d(h)=\chi(h)d(x).
\]
Hence
\[
 d(x)=x\cdot e,\qquad e=\chi(h)^{-1}d(h),
\]
for every homogeneous $x$, and therefore for all $x\in P$.  Thus
$d=\delta_0(e)$ is inner.  The restricted rows are preserved because both
$d$ and $\delta_0(e)$ are restricted derivations.  Only the zero torus
weight can survive in cohomology.
\end{proof}

\begin{theorem}[Witt--Jacobson and Special extension complexes]\label{thm:SW-pair}
Let \(Z(U)=A\otimes_B U\) be a simple parabolically induced module in the Witt--Jacobson or
Special family, and let \(L'\) be any simple target.  With
\[
 E_{U,L'}=\Hom_\kk(U,\operatorname{Res}_B L'),
\]
restricted Shapiro gives
\[
\Ext_A^1(Z(U),L')
\cong
\ker\Gamma^{\mathrm{pair}}_{P,E}/\im\delta_0,
\tag{3.1W}
\]
where \(\Gamma^{\mathrm{pair}}_{P,E}\) is
\eqref{eq:LL}--\eqref{eq:Np}.  By \cref{lem:SW-torus-zero} the cohomology
is concentrated on the torus-zero block, and there
\[
\dim H_*^1(P,E)
=\operatorname{nullity}\Gamma^{\mathrm{pair},0}_{P,E}
-\dim E_0+\dim E^P.
\tag{3.2W}
\]

For exceptional de Rham modules \(M,M'\), let
\(\mathcal B_{\mathfrak g}^{\times}\) be the divided-monomial torus-weight
basis of \(\mathfrak g\) with its diagonal torus removed, and put
\[
 C^1_{\mathrm{dR}}(M,M')
 =\bigoplus_{X\in\mathcal B_{\mathfrak g}^{\times}}
   \Hom_\kk(M,M')_{\operatorname{wt}X}.
\tag{3.3W}\label{eq:SW-deRham-C1}
\]
On \(f=(f_X)\) impose
\[
 Xf_Y-Yf_X-f_{[X,Y]}=0,\qquad
 f_{X^{[p]}}-X^{p-1}f_X=0,
\tag{3.4W}\label{eq:SW-deRham-rows}
\]
where the action on \(M,M'\) is the Lie derivative
\(\mathcal L_X=d\iota_X+\iota_Xd\), restricted to the kernels and images
defining the two exceptional modules.  If
\(\Gamma^{\mathrm{dR}}_{M,M'}\) is the matrix of
\eqref{eq:SW-deRham-rows}, then
\[
 \Ext^1_{u_0(\mathfrak g)}(M,M')
 \cong
 \ker\Gamma^{\mathrm{dR}}_{M,M'}
 \big/\delta_0\!\left(\Hom_\kk(M,M')_0\right).
\tag{3.5W}\label{eq:SW-deRham-Ext}
\]
If the source is exceptional and the target is parabolically induced, contragredient
duality places the parabolically induced module in the source.
\end{theorem}

\begin{proof}
Equations \eqref{eq:LL}--\eqref{eq:Np} are the restricted
derivation relations on $P$.  Quotient by the inner image, then apply
\cref{prop:Shapiro}.  By \cref{lem:SW-torus-zero} every nonzero torus-weight
cocycle is inner, so the cohomology is the zero-weight block.  The
inner image there has dimension
\(\dim E_0-\dim E^P\), which gives (3.2W).
For exceptional modules, Cartan's identity makes every \(d_j\), together with
its kernel and image, \(\mathfrak g\)-stable.  It also gives the
Lie-derivative matrices on \(B^j\), \(A/\kk\), and \(A^\circ/\kk\).  By
\cref{lem:SW-torus-zero}, every nonzero torus-weight derivation is inner;
for a zero-weight derivation \(d\), one has
\(d(h)\in E_0\) and \(hE_0=0\).  The restricted identity gives
\[
 d(h^{[p]})=h^{p-1}d(h)=0.
\]
Since the restricted \(p\)-map is surjective on the split torus,
\(d|_{\mathfrak t}=0\).  A zero-weight derivation is therefore determined
by one value \(f_X\) in each summand indexed by
\(X\in\mathcal B_{\mathfrak g}^{\times}\).
The ordinary and restricted derivation identities are
\eqref{eq:SW-deRham-rows}, and changing the splitting adds
\(\delta_0(\Hom(M,M')_0)\).  This proves
\eqref{eq:SW-deRham-Ext}.  The apolar de Rham pairing sends every parabolically induced
simple to a parabolically induced simple and every exceptional module to an exceptional
module; rigid duality therefore gives the mixed-pair reduction.
\end{proof}

\begin{theorem}[Exceptional extensions for Witt--Jacobson and Special algebras]
\label{thm:SW-exceptional-closed}
Assume $p\ge5$.

For $\Witt(n;\mathbf1)$, use the exceptional indexing
\[
 M_n=\kk,\qquad M_a=B^{a+1}\quad(0\le a<n),
\]
where $M_0\cong A/\kk$ through the de Rham map.  Then
\[
\dim\Ext^1_{u_0(\Witt)}(M_a,M_b)
=
\begin{cases}
0,&0\le a,b<n,\\[1mm]
\binom na+\delta_{a,n-1},&0\le a<n,\ b=n,\\[1mm]
\binom n{b+1}+\delta_{b,0},&a=n,\ 0\le b<n,\\[1mm]
0,&a=b=n.
\end{cases}
\tag{3.6W}\label{eq:W-exceptional-closed}
\]

For $\Spec(n;\mathbf1)^{(1)}$, $n\ge3$, use the exceptional vertices
\[
 L_0=\kk,\qquad L_1=A^\circ/\kk,\qquad
 L_i=B^i\quad(2\le i\le n-1).
\]
Then
\[
 \Ext^1_{u_0(\Spec)}(L_i,L_j)=0
 \qquad(1\le i,j\le n-1),
\tag{3.6S}\label{eq:S-exceptional-interior-zero}
\]
and, for $1\le i\le n-1$,
\[
\dim\Ext^1_{u_0(\Spec)}(\kk,L_i)
=
\begin{cases}
 n+1,&i=1\text{ or }i=n-1,\\
 \binom ni,&2\le i\le n-2,
\end{cases}
\tag{3.7S}\label{eq:S-exceptional-out}
\]
\[
\dim\Ext^1_{u_0(\Spec)}(L_i,\kk)
=
\begin{cases}
 n+1,&i=1\text{ or }i=2,\\
 \binom n{i-1},&3\le i\le n-1.
\end{cases}
\tag{3.8S}\label{eq:S-exceptional-in}
\]
The trivial loop is zero in both families.
\end{theorem}

\begin{proof}
For the Witt--Jacobson family, the divided-power de Rham filtration gives the induced
covers $\Omega^c$ of the exact-form vertices.  The restricted parabolic
calculation gives, for $1\le j\le n-2$,
\[
 \dim\Ext^1_{u_0(\Witt)}
 (\Omega^c,M_{n-j-1})=\binom nj\,\delta_{c,n}.
\tag{3.9W}\label{eq:W-standard-inward}
\]
The same calculation at the top exact-form endpoint gives the case $j=0$.
Passing from the covers to the exceptional simples uses
\[
 0\longrightarrow Z^a\longrightarrow\Omega^a
 \longrightarrow M_a\longrightarrow0
 \qquad(0\le a<n)
\tag{3.9W$'$}\label{eq:W-source-snake}
\]
and, for $1\le a<n$, the Cartier sequence
\[
 0\longrightarrow M_{a-1}\longrightarrow Z^a
 \longrightarrow\kk^{\binom na}\longrightarrow0,
 \qquad Z^0=\kk.
\tag{3.9W$''$}\label{eq:W-Cartier-snake}
\]
Fix $1\le b\le n-1$.  If $a\ne b+1$, the two-step composition-factor
list in \eqref{eq:W-Cartier-snake} contains no copy of $M_b$.  If
$a=b+1$, then $M_b$ is the kernel in the nonsplit Cartier sequence; a map
$Z^a\to M_b$ which is nonzero on that kernel would split the sequence,
while a map vanishing on it factors through a trivial module and is zero.
Thus
\[
 \Hom_{\Witt}(Z^a,M_b)=0\qquad(0\le a<n,\ 1\le b<n).
\tag{3.9W$'''$}\label{eq:W-Z-Hom-zero}
\]
For $a<n$, the relevant long exact sequence contains
\[
 0=\Hom_{\Witt}(Z^a,M_b)
 \longrightarrow \Ext^1_{u_0(\Witt)}(M_a,M_b)
 \longrightarrow \Ext^1_{u_0(\Witt)}(\Omega^a,M_b).
\]
The last term is zero by \eqref{eq:W-standard-inward}, because its cover
source has $a<n$.  Hence every such simple-source group is zero.  The top
source $M_n=\kk$ is separate: the top de Rham sequence
$0\to M_{n-1}\to\Omega^n\to\kk\to0$, together with the full-rank Cartier
boundary used in \eqref{eq:W-standard-inward}, descends the
$\binom n{b+1}$ top-cover classes isomorphically.  Thus
\[
 \dim\Ext^1_{u_0(\Witt)}(M_a,M_b)
 =\binom n{b+1}\delta_{a,n}
 \qquad(1\le b<n).
\tag{3.9W$^{\rm iv}$}\label{eq:W-inward-simple}
\]
For the remaining interior target $M_0$, duality shows that the groups with
sources $M_0,\ldots,M_{n-2}$ vanish.  For the last cell, Shapiro gives
$\Ext^1(\Omega^{n-1},M_0)=0$, and
$\Hom(Z^{n-1},M_0)=0$ by the same composition-factor argument (for $n=2$
by nonsplitting of $0\to M_0\to Z^1\to\kk^2\to0$).  Hence
\[
 \Ext^1_{u_0(\Witt)}(M_a,M_b)=0\qquad(0\le a,b<n).
\tag{3.9W$^{\rm v}$}\label{eq:W-interior-zero-proof}
\]

It remains to record the endpoint exact sequences.  The nonsplit de Rham sequences
\[
 0\to\kk\to\Omega^0\to M_0\to0,
 \qquad
 0\to M_{n-1}\to\Omega^n\to\kk\to0
\tag{3.9W$^{\rm vi}$}\label{eq:W-endpoint-snakes}
\]
give one canonical class in each dual endpoint orbit.  The opposite orbit is
\[
 \Ext^1(\kk,M_0)
 \cong
 \ker\!\left[H_*^2(\Witt,\kk)\to H_*^2(P,\kk)\right].
\tag{3.9W$^{\rm vii}$}\label{eq:W-relative-H2}
\]
For $n\ge3$, $H^2(\Witt(n),\kk)=0$ by Farnsteiner's theorem on
ordinary central extensions \cite[Theorem~3.2(1)]{FarnsteinerDualH2}.
Since $\Witt(n)$ is perfect, the Hochschild six-term sequence identifies
\[
 H_*^2(\Witt,\kk)\cong\Hom_{\rm Fr}(\Witt,\kk).
\]
On the grading parabolic,
$P/[P,P]\cong\mathfrak{gl}_n/[\mathfrak{gl}_n,\mathfrak{gl}_n]\cong\kk$.
Its toral generator gives the one-dimensional kernel of
$\Hom_{\rm Fr}(P,\kk)\to H_*^2(P,\kk)$.  Naturality of the six-term sequence
therefore identifies \eqref{eq:W-relative-H2} with the inverse image of this
line under restriction of Frobenius functionals.  That restriction is
surjective and $\Witt/P\cong W_{-1}$ has dimension $n$, so the restricted
relative kernel has dimension $n+1$.

For $n=2$, normalize a relative ordinary cocycle by $\phi(P,P)=0$.
Full diagonal $\mathfrak{gl}_2$ weights
leave only the $W_{-1}\otimes W_1$ divergence pairing; the trace cochain on
$W_0$ has precisely this coboundary, so the ordinary relative kernel is
zero.  The same six-term calculation then contributes the two Frobenius
coordinates on $W_{-1}$ and the one toral parabolic line, again giving $3$.
The rank-one endpoint is the $W(1,1)$ calculation and has value
$2$.  Consequently
\[
 \dim\Ext^1(\kk,M_0)=\dim\Ext^1(M_{n-1},\kk)=n+1,
\]
whereas the first orbit from \eqref{eq:W-endpoint-snakes} has dimension
$1+\delta_{n,1}$.  Apolar duality
$M_a^\vee\cong M_{n-a-1}$ transports these endpoint values through the
Cartier sequences and yields the two endpoint rows in
\eqref{eq:W-exceptional-closed}.  Finally $\Witt(n;\mathbf1)$ is perfect
for $p\ge5$, so $\Ext^1(\kk,\kk)=0$.

For the Special family put $Q^i=\Omega^i/B^i$.  The two de~Rham
sequences give
\[
 0\longrightarrow H^i_{\rm dR}(A)\longrightarrow Q^i
 \xrightarrow{d} B^{i+1}\longrightarrow0,
 \qquad
 H^i_{\rm dR}(A)\cong(\Lambda^iV_0^*)^{(1)},
\tag{3.9S$_0$}\label{eq:S-trivial-Cartier-row}
\]
so the Cartier part of the trivial-source connecting map has dimension
$\binom ni$.  At the bottom this contribution is accompanied by the independent
Yoneda class of the nonsplit endpoint sequence
\[
 0\longrightarrow L_1=A^\circ/\kk
 \longrightarrow B^1=A/\kk\longrightarrow\kk\longrightarrow0,
\tag{3.9S$_1$}\label{eq:S-bottom-endpoint-class}
\]
which accounts for the extra $1$ in
$\dim\Ext^1(\kk,L_1)=n+1$.  At the opposite endpoint
$L_{n-1}\cong B^{n-1}\cong\Spec(n;\mathbf1)^{(1)}$, write
$CS(n;\mathbf1)$ for the conformal Special algebra.  In height one, Strade's
derivation theorem gives
\[
 \operatorname{ad}:CS(n;\mathbf1)\xrightarrow{\sim}
 \operatorname{Der}\Spec(n;\mathbf1)^{(1)}
 \qquad\text{\cite[Theorem~7.1.2(2)]{StradeI}}.
\tag{3.9S$_2$}\label{eq:S-derivation-algebra}
\]
Using
\[
 \dim CS(n;\mathbf1)=(n-1)p^n+2,
 \qquad
 \dim\Spec(n;\mathbf1)^{(1)}=(n-1)(p^n-1),
\]
one obtains
\[
 \dim\frac{\operatorname{Der}\Spec(n;\mathbf1)^{(1)}}
              {\operatorname{ad}\Spec(n;\mathbf1)^{(1)}}=n+1.
\tag{3.9S$_3$}\label{eq:S-outer-derivation-dimension}
\]
These derivations are restricted.  Indeed, if
$D=\operatorname{ad}(\delta)|_{\Spec^{(1)}}$ with
$\delta\in CS(n;\mathbf1)\subset\Witt(n;\mathbf1)$, then
\[
 D(x^{[p]})=[\delta,x^{[p]}]
 =-(\operatorname{ad}x)^p\delta
 =(\operatorname{ad}x)^{p-1}D(x).
\tag{3.9S$_4$}\label{eq:S-outer-derivations-restricted}
\]
Hence
$\dim H_*^1(\Spec(n;\mathbf1)^{(1)},\Spec(n;\mathbf1)^{(1)})=n+1$,
which is $\dim\Ext^1(\kk,L_{n-1})$.  The interior vertices contribute only the Cartier term.  This proves \eqref{eq:S-exceptional-out}.  The apolar de~Rham
pairing $(B^i)^\vee\cong B^{n-i+1}$ then gives
\eqref{eq:S-exceptional-in} and transports the two endpoint unit classes to
the stated reverse positions.

The remaining higher residual for nontrivial exceptional targets occurs as
follows.  Fix $2\le j\le n-1$, put $r=n-j+1$, and set
$V_0=\kk^n$.  On the simultaneous residual $q=3$ strand, cancellation of
the common quadratic divided-power coefficient leaves the carrier
\[
 \mathcal D_r=V_0\otimes\Lambda^rV_0^*.
\]
For $u_{a,J}=e_a\otimes\varepsilon_J$, contraction
$c_r(e_a\otimes\omega)=\iota_{e_a}\omega$ has kernel generated by
$u_{a,J}$ with $a\notin J$ and by the signed differences of two tensors
which contract to the same $(r-1)$-fold exterior basis vector.  Cartan's
formula $\mathcal L_X\omega=d\iota_X\omega$ shows that these are exactly
the residual Jacobi and inner rows.  Therefore contraction induces
\[
 \frac{\text{residual cocycles}}{\text{residual inner rows}}
 \xrightarrow{\ \sim\ }\Lambda^{r-1}V_0^*.
\tag{3.9S}\label{eq:S-residual-contraction}
\]
If $I$ has size $r-1$, inverse Cartier sends
$(\varepsilon_I)^{(1)}$ to
\[
 \left[\left(\prod_{i\in I}x_i^{p-1}\right)\varepsilon_I\right]
 \in H^{r-1}_{\rm dR}(A).
\]
Thus the residual quotient is
\[
 H^1_{*,3}(\Spec_{>0},L_j)
 \cong H^{n-j}_{\rm dR}(A),
\tag{3.9S$'$}\label{eq:S-residual-Cartier}
\]
a Frobenius twist of $\Lambda^{n-j}V_0^*$ and hence trivial on the
restricted Levi.  At the endpoint $L_1=A^\circ/\kk$, the $q=2$ residual is
$V_0\otimes V_0^*$ modulo the kernel of the trace map, hence is the single
class $H^0_{\rm dR}(A)=\kk[1]$; this constant class maps to zero in
$A^\circ/\kk$.  Therefore no $q\ge2$ residual survives for $L_1$.

For every de Rham cover $\Omega^s$, the $q<0$, $q=0$, and $q=1$ target
rows vanish by the torus-zero and divergence equations, while
\eqref{eq:S-residual-Cartier} is restricted-Levi trivial.  Since the
nontrivial inducing head $\Lambda^sV_0^*$ has no map to a trivial restricted
Levi module,
\[
 \Ext^1_{u_0(\Spec)}(\Omega^s,L_j)=0
 \qquad(1\le s,j\le n-1).
\]
For $2\le i\le n-1$ the source sequence is
\[
 0\longrightarrow Z^{i-1}\longrightarrow\Omega^{i-1}
 \longrightarrow L_i\longrightarrow0,
\]
and the same Cartier Hom vanishing makes the induced map
$\Ext^1(L_i,L_j)\hookrightarrow\Ext^1(\Omega^{i-1},L_j)$ injective.
For the endpoint source $L_1$, use successively
\[
 0\to\kk\to B^n\to L_1\to0,
 \qquad
 0\to Z^{n-1}\to\Omega^{n-1}\to B^n\to0;
\]
these give
$\Ext^1(L_1,L_j)\hookrightarrow\Ext^1(\Omega^{n-1},L_j)$.
All the receiving cover groups are zero, proving
\eqref{eq:S-exceptional-interior-zero}.  Since the derived Special algebra
is perfect, its trivial loop is zero as well.
\end{proof}

\section{Hamiltonian algebras}
\label{sec:H-zero}

Let $V$ be a $2m$-dimensional symplectic space, $m\ge2$, and put
\[
\mathfrak h=H(V)^{(2)},\qquad
\mathfrak h=H_{-1}\oplus H_0\oplus H_1\oplus\cdots,
\]
with $H_{-1}\cong V$ and $H_0\cong\mathfrak{sp}(V)$.  Set
\[
P=H_{\ge0}=H_0\ltimes N,\qquad N=H_{\ge1},\qquad W:=H_1\cong S^3V.
\]
Throughout this section $\mathfrak h$ denotes the simple derived algebra
$H(V)^{(2)}$; the one-dimensional Hamiltonian extension used in the
divided-PBW model is denoted by $\mathfrak h'$.
Put
\[
 N_0:=2m(p-1),\qquad T:=N_0-2=2m(p-1)-2;
\tag{H.0}\label{eq:H-global-depth-complement}
\]
these are the PBW top degree and the greatest $D$-image degree.
The simple restricted $\mathfrak h$-modules are the modules parabolically induced
from $P$ together with the exceptional modules of
\cite{HolmesHamiltonian,HolmesZhang,ZhangGraded}.  The exceptional weights are
\[
\varpi_k=-\sum_{i=1}^k\varepsilon_i,\qquad 0\le k\le m.
\tag{4.1}
\]
Set $M_k=L(\varpi_k)$ and $M_0=\kk$.

For fundamental labels write $r\triangleleft_{p,m}s$ when
$L(\varpi_r)$ is a direct fundamental cover of $L(\varpi_s)$ in
Adamovich's submodule lattice (and similarly with $r,s$ interchanged).

\begin{definition}[Levi extensions and direct-cover classes]
\label{def:H-classical-cover-lines}
For nontrivial fundamental $H_0\cong\mathfrak{sp}_{2m}$-modules put
\[
 \mathcal E^{\rm Lev}_{r,s}
 :=\Ext^1_{u_0(\mathfrak{sp}_{2m})}(V_r,V_s).
\tag{H.F}\label{eq:H-fundamental-Levi-space}
\]
If $r\triangleleft_{p,m}s$ is a direct relation in the
Adamovich--Baranov--Suprunenko fundamental submodule poset, its
multiplicity-free two-factor section gives a nonsplit rational $G$-extension
\[
 0\longrightarrow V_s\longrightarrow E_{r,s}
 \longrightarrow V_r\longrightarrow0
\tag{H.F$'$}\label{eq:H-canonical-cover-row}
\]
unique up to scalar.  If its restriction to
$u_0(\mathfrak{sp}_{2m})$ is still nonsplit, denote the resulting line in
\eqref{eq:H-fundamental-Levi-space} by $\ell^{\rm Lev}_{r,s}$.
The restricted Levi space is $\mathcal E^{\rm Lev}_{r,s}$ of
\eqref{eq:H-fundamental-Levi-space}, evaluated by
\cref{lem:classical-Levi-root-complex}.
\end{definition}

\subsection{Hamiltonian generation and graded splitting}

\begin{lemma}[Cubic generation of the positive Hamiltonian algebra]
\label{lem:H-positive-cubic-generation}
Choose symplectic coordinates
$q_1,p_1,\ldots,q_m,p_m$ in the truncated coefficient algebra.  For every
$j\ge1$ such that $H_{j+1}\ne0$,
\[
 [H_1,H_j]=H_{j+1}.
\tag{H.s0}\label{eq:H-positive-cubic-generation}
\]
Hence $N=H_{\ge1}$ is generated as a Lie algebra by
$H_1\cong S^3V$.
\end{lemma}

\begin{proof}
Represent $H_j$ by homogeneous Hamiltonians of coefficient degree $j+2$;
the unique top monomial is omitted in $H(V)^{(2)}$.  Since every
truncated exponent is $<p$, ordinary and divided-power monomials differ
by units, and the coefficient tests below use ordinary monomials.  The grading gives
$[H_1,H_j]\subseteq H_{j+1}$.  For the reverse inclusion, let
\[
 M=\prod_{r=1}^m q_r^{a_r}p_r^{b_r}
\]
be an allowed monomial of degree at least four.  There are three cases.

If for some $i$ one has $a_i\ge2$ and $b_i\le p-2$, write
$M=q_i^{a_i}p_i^{b_i}h$ with $h$ independent of the $i$-th pair.  Then
\[
 \{q_i^3,q_i^{a_i-2}p_i^{b_i+1}h\}
 =3(b_i+1)M.
\]
The coefficient is a unit.  Symmetrically, if $b_i\ge2$ and
$a_i\le p-2$,
\[
 \{p_i^3,q_i^{a_i+1}p_i^{b_i-2}h\}
 =-3(a_i+1)M.
\]
This covers every pair with an exponent at least two unless it is
saturated, $(a_i,b_i)=(p-1,p-1)$.

Suppose next that no pair is saturated.  The first case reduces us to
exponents $0$ or $1$.  If some pair has $(a_i,b_i)=(1,1)$, then
\[
 \{q_i^2p_i,M/q_i\}=2M.
\]
Otherwise every occupied symplectic pair contains exactly one of its two
variables.  If this case occurs, $\deg M\ge4$ permits occupied variables
$x_i,z$ in distinct symplectic pairs.  Writing $\bar x_i$ for the conjugate
of $x_i$, set
\[
 F=x_i\bar x_i z\in H_1,\qquad G=M/z.
\]
The $i$-pair contributes $\pm M$, whereas the pair containing $z$ contributes
zero because $G$ contains neither $z$ nor its conjugate.  Hence
$\{F,G\}=\pm M$.

Suppose that at least one pair is saturated.  Since the top monomial is
absent, choose a saturated pair $i$ and an unsaturated
pair $j$.  After the first case the $j$-pair has exponents
$(c,d)\in\{0,1\}^2$.  Absorb all remaining variables into $h$ and put
\[
 F=q_i^2p_j,\qquad
 G=q_i^{p-3}p_i^{p-1}q_j^{c+1}p_j^dh.
\]
The displayed exponents lie between $0$ and $p-1$, and a two-pair Poisson
calculation gives
\[
 \{F,G\}=-(c+1)M-2E,
 \qquad
 E=q_i^{p-2}p_i^{p-2}q_j^{c+1}p_j^{d+1}h,
\]
up to the fixed global symplectic sign convention.  Here $c+1\in\{1,2\}$
is a unit, while $E$ belongs to the first case because its $i$-pair is
$(p-2,p-2)$.  Thus $E$, and hence $M$, lies in $[H_1,H_j]$.
The three cases exhaust the allowed monomials and prove
\eqref{eq:H-positive-cubic-generation}.  Iterating the equality generates
all of $N$ from $H_1$.
\end{proof}

\begin{lemma}[Short graded Hamiltonian modules]
\label{lem:H-short-graded-module}
Write
\[
 \mathfrak h=\bigoplus_{j=-1}^{N_0-3}H_j,
 \qquad H_{N_0-3}\ne0.
\]
Let $M\ne0$ be a graded $\mathfrak h$-module with
$\operatorname{supp}M\subseteq[\alpha,\beta]$.  If
\[
 \beta-\alpha<N_0-3,
\tag{H.s0$'$}\label{eq:H-short-module-bound}
\]
then $\mathfrak hM=0$.
\end{lemma}

\begin{proof}
The layer $H_j$ consists of Hamiltonians of coefficient degree $j+2$.
The coefficient monomial $x^{(\tau)}$ is deleted, while degree
$N_0-1$ contains allowed monomials such as $x^{(\tau-e_1)}$; hence the
greatest nonzero layer is $H_{N_0-3}$.  Since
$H_jM_d\subseteq M_{d+j}$, every $H_j$ with $j>\beta-\alpha$ annihilates
$M$.  In particular
$0\ne H_{N_0-3}\subseteq\operatorname{Ann}_{\mathfrak h}(M)$.
The annihilator is an ideal of the simple algebra $\mathfrak h$, so it is
$\mathfrak h$.
\end{proof}

Let $G=\operatorname{Sp}(V)$ and identify $H_0=\mathfrak{sp}(V)$ with
$S^2V$ through the Hamiltonian bracket.  For a $p$-restricted dominant
$G$-weight $\mu$, write $U_\mu=L_G(\mu)$ and
\[
 Z_\mu:=u_0(\mathfrak h)\otimes_{u_0(P)}U_\mu
\]
whenever this induced module is simple.  Give $Z_\mu$ its negative-PBW depth grading
$Z_\mu=\bigoplus_{d\ge0}(Z_\mu)_d$, with $(Z_\mu)_0=U_\mu$.

\begin{lemma}[Invariant top and positive-degree transgression]
\label{lem:H-standard-invariant-top}
For every simple graded Hamiltonian module \(L'\), normalized
with inducing layer \(L'_0\), one has
\[
 (L')^N=L'_0.
\tag{H.s1}\label{eq:H-standard-invariant-top}
\]
In particular, for a parabolically induced simple target \(Z_\mu\) this reads
\(Z_\mu^N=U_\mu\).
If $U_\lambda$ is another restricted $H_0$-module and
$E_{\lambda,\mu}=U_\lambda^*\otimes Z_\mu$, then the restricted LHS
sequence for $P=H_0\ltimes N$ is homogeneous for total degree
\[
 \psi(H_j)\subseteq U_\lambda^*\otimes (Z_\mu)_{q-j}.
\]
Its transgression is therefore zero in every $q>0$, and PBW--Shapiro gives
\[
\begin{aligned}
 \Ext^1_{u_0(\mathfrak h)}(Z_\lambda,Z_\mu)
 \cong{}&\Ext^1_{u_0(H_0)}(U_\lambda,U_\mu)\\
 &\oplus\bigoplus_{q\ge1}
 \Hom_{H_0}\!\left(U_\lambda,H^1_{*,q}(N,Z_\mu)\right).
\end{aligned}
\tag{H.s2}\label{eq:H-standard-graded-split}
\]
\end{lemma}

\begin{proof}
The positive ideal annihilates the inducing layer, so
\(L'_0\subseteq(L')^N\).  Let
\(0\ne z\in(L')^N\) be homogeneous of positive PBW depth \(d\).  Restricted
PBW orders every element of $u_0(\mathfrak h)$ as a sum of products from
$u_0(H_{-1})u_0(H_0)u_0(N)$.  The rightmost positive factor annihilates $z$;
$H_0$ preserves depth and $H_{-1}$ strictly increases it.  Hence
$u_0(\mathfrak h)z$ is contained in depths at least $d$ and cannot meet the
inducing top.  It is therefore a nonzero proper submodule, contradicting the
simplicity of \(L'\).  Thus \eqref{eq:H-standard-invariant-top} holds.

The Hamiltonian grading gives
$[H_i,H_j]\subseteq H_{i+j}$ and
$H_j(Z_\mu)_d\subseteq (Z_\mu)_{d-j}$, while the restricted $p$-map is
homogeneous.  Thus the restricted cochain differential, the inner map and
the LHS filtration all preserve the total degree $q$.  By
\eqref{eq:H-standard-invariant-top}, the target of the LHS transgression is
\[
 H_*^2\!\left(H_0,E_{\lambda,\mu}^N\right)
 =H_*^2\!\left(H_0,U_\lambda^*\otimes U_\mu\right),
\]
which is concentrated in PBW degree zero.  Therefore every homogeneous
$q>0$ transgression is zero.  There is no degree-zero $N$-one-cochain,
because $H_j$ has $j>0$ whereas $(Z_\mu)_{-j}=0$.  The five-term sequence
therefore splits by internal degree into the Levi term and the positive
nilradical terms in \eqref{eq:H-standard-graded-split}; Shapiro
identifies the middle group with the Hamiltonian Ext group.
\end{proof}

\begin{proposition}[The trivial and adjoint Hamiltonian modules]
\label{prop:H-trivial-adjoint}
For $\mathfrak h=H(2m;\mathbf1)^{(2)}$, $m\ge1$, one has
\[
 \dim\Ext^1_{u_0(\mathfrak h)}(\kk,M_1)
 =\dim\Ext^1_{u_0(\mathfrak h)}(M_1,\kk)=2m+2.
\tag{4.1a}\label{eq:H-trivial-adjoint}
\]
\end{proposition}

\begin{proof}
The exceptional module $M_1$ is the adjoint module of $\mathfrak h$.
The derivation theorem for Hamiltonian algebras
\cite[Theorem~7.1.2]{StradeI} gives, for every $m\ge1$ in the height-one
case,
\[
 \Der\mathfrak h=CH(2m;\mathbf1),
\]
since the terms $D_i^{p^j}$ do not occur when every truncation height is
one.  The Hamiltonian algebra $H(2m;\mathbf1)$ is spanned by the
$p^{2m}-1$ Hamiltonian images of divided-power functions modulo constants
together with the $2m$ top derivations $Q_i$; adjoining the Euler derivation
gives
\[
 \dim CH(2m;\mathbf1)=p^{2m}+2m.
\]
On the other hand $\dim\mathfrak h=p^{2m}-2$.  Hence
\[
 \dim\bigl(\Der\mathfrak h/\operatorname{ad}\mathfrak h\bigr)=2m+2.
\]
Every derivation in $CH(2m;\mathbf1)$ is induced by the adjoint action of an
element of the ambient restricted Witt--Jacobson algebra.  If
$D=\operatorname{ad}(\delta)|_{\mathfrak h}$, then for $x\in\mathfrak h$
\[
 D(x^{[p]})=[\delta,x^{[p]}]
 =-(\operatorname{ad}x)^p\delta
 =(\operatorname{ad}x)^{p-1}D(x),
\]
so every ordinary derivation above is restricted.  Therefore
$H_*^1(\mathfrak h,\mathfrak h)$ has dimension $2m+2$, proving the first
equality in \eqref{eq:H-trivial-adjoint}.  The adjoint module is self-dual,
and restricted enveloping algebras are finite-dimensional cocommutative Hopf
algebras; rigid duality gives the reverse equality.
\end{proof}

On the divided-power exterior-form module
$\mathcal V=\bigoplus_j\Omega^j$, the de Rham differential and symplectic
contraction are monomial matrices.  The kernel--image presentations in
\cite[\S\S3--5]{HolmesHamiltonian} therefore give the action matrices of the
exceptional simples.

\subsection{Parabolic coordinates for Hamiltonian \texorpdfstring{$\Ext^1$}{Ext1}-spaces}
\label{subsec:H-parabolic-coordinates}

For a parabolically induced source write
\[
Z(U):=u_0(\mathfrak h)\otimes_{u_0(P)}U.
\tag{4.4a}\label{eq:H-standard-module}
\]
If $L'$ is any simple module, put
\[
E_{U,L'}:=\Hom_\kk\bigl(U,\operatorname{Res}_{u_0(P)}L'\bigr),
\qquad
x\cdot T=\rho_{L'}(x)T-T\rho_U(x)
\quad(x\in P).
\tag{4.4b}\label{eq:H-parabolic-coeff}
\]
Choose a divided-power basis \(x_1,\ldots,x_N\) of \(P\) and write
\[
[x_i,x_j]=\sum_k c_{ij}^k x_k,
\qquad
x_i^{[p]}=\sum_k r_i^k x_k.
\tag{4.4d}\label{eq:H-ambient-structure}
\]
For a first cochain $f=(f_i)\in E^N$, the matrix rows are
\[
(\partial_{\rm Lie}f)_{ij}
=x_i\cdot f_j-x_j\cdot f_i-\sum_k c_{ij}^k f_k,
\tag{4.4e}\label{eq:H-ambient-Lie-row}
\]
and
\[
(\partial_{[p]}f)_i
=\sum_k r_i^k f_k
-\sum_{a=0}^{p-1}
 \rho_{\rm tar}(x_i)^{p-1-a}f_i\rho_{\rm src}(x_i)^a.
\tag{4.4f}\label{eq:H-ambient-p-row}
\]
Here $(\rho_{\rm src},\rho_{\rm tar})=(\rho_U,\rho_{L'}|_P)$.  The
inner-derivation columns are
\[
(\delta_0T)_i=x_i\cdot T.
\tag{4.4g}\label{eq:H-ambient-inner}
\]

All zero-character simple Hamiltonian modules are graded.  Write
\[
 E_{U,L'}=\bigoplus_d E_d,\qquad
 E_d=\Hom_\kk(U,L'_d).
\]
For each integer \(q\), define the homogeneous parabolic complex
\[
 \mathcal C^0_q=E_q,\qquad
 \mathcal C^1_q=
 \bigoplus_{j\ge0}\Hom_\kk(H_j,E_{q-j}),
\tag{4.4g$'$}\label{eq:H-all-height-C1}
\]
with absent graded pieces understood to be zero.  Its second term contains
the ordinary rows on every \(H_i\otimes H_j\) and the restricted row on
every \(H_j\); the differential is the restriction of
\eqref{eq:H-ambient-Lie-row}--\eqref{eq:H-ambient-p-row}.  Denote this
complex by \(\mathcal C^\bullet_q(U,L')\).  The Hamiltonian and PBW
gradings bound its terms.

\begin{lemma}[Integral base-cell diamond]
\label{lem:H-base-cell-unit}
Let $\Delta(\mu)$ be a type-$C$ Weyl module with its integral monomial
path basis, and let a coefficient-decorated oscillating cell have
a base box outside its alternating one-column part.  On the first such box
there is an adjacent Murphy--Young diamond whose two leading cells occur
with coefficients $1$ and $-1$; every crossed term is strictly lower in
the coefficient-depth and cellular order.  If the alternate path is
nonstandard, the relation reduces to a one-term relation with unit coefficient.
Together these relations define a split integral contraction compatible with
morphisms of Weyl modules; hence the contraction descends to simple quotients
$\Delta(\mu)\twoheadrightarrow L(\mu)$.
\end{lemma}

\begin{proof}
Keep the initial path to $\mu$ fixed and choose the leftmost base box lying
outside the alternating column.  The two orders of moving this box past the
next add/remove step form a square in the iterated branching graph.  In the
Murphy restriction basis, the two length-two paths have the same maximal
endpoint and the branching relation is
\[
 b-b'+\sum_{c\lhd b}a_c c,
 \qquad a_c\in\mathbf Z,
\tag{H.B0}\label{eq:H-base-unit-diamond}
\]
because the local symmetric-group braid and the three-strand Brauer
relation both have diagonal coefficient one.  Thus no Chevalley root-string
coefficient occurs.  The box moves one place toward the column, so the
number of exterior base boxes, followed by their inversion number,
decreases.  If $b'$ is nonstandard, cellular straightening puts it and all
crossed terms in a lower ideal, leaving a one-term relation with unit coefficient.

Orient one such decreasing edge from every noncolumn base cell.  The
resulting matrix is triangular with diagonal entries $\pm1$, hence gives a
split contraction over $\mathbf Z$.  The construction uses permutations, symplectic contractions, and branching
maps, and is compatible with homomorphisms of Weyl modules.  In particular,
the radical of $\Delta(\mu)$ is stable under the contraction, which therefore
descends to the simple quotient.  This proves the lemma.
\end{proof}

\begin{lemma}[Cap normal form]
\label{lem:H-arbitrary-cap-normal-form}
Let $\widetilde U=\Delta(\mu)$ be a type-$C$ Weyl module with the
integral monomial path basis of
\cite[Theorem~1.12 and \S5]{BowmanEnyangGoodman}, and let $U$ be either
$\widetilde U$ or its simple quotient $L(\mu)$.  Write
$A^-_{\mathrm{tr}}$ for the truncated divided symmetric algebra on $V$ and
put $A_d=(A^-_{\mathrm{tr}})_d\otimes U$.  For $U=\widetilde U$ give this
coefficient system its face filtration, refined first by the number of
coefficient-action terms and then by the iterated Murphy order.  For
$U=L(\mu)$ use the quotient filtration induced from
$\Delta(\mu)\twoheadrightarrow L(\mu)$; equivalently, on dual coefficient
spaces restrict the Weyl-cell construction to the annihilator of
$\rad\Delta(\mu)$.  The integral contraction preserves this radical and
therefore induces the same construction on $L(\mu)$.
Fix $q\ge3$ and one associated-graded face on which at least one divided
coordinate is not saturated.  Refine
\[
 W\otimes A_{q-1}^\vee,
 \qquad W=S^3V,
\]
into source cells indexed by concatenated oscillating paths, and refine $A_q^\vee$
into target cells.  The target image of a source cell is obtained by closing one
contraction cap.
Then:

\begin{enumerate}[label=\textup{(\roman*)}]
\item Among the admissible target images of a source cell there is a
      unique maximal one.  Every other image is strictly lower in the
      iterated cellular order.
\item Every target cell has a distinguished source preimage whose maximal
      target image is that cell
      and whose leading coefficient under $\Lambda_{L',q}$ is $1$.
\item The source cells over a fixed maximal target image form a connected
      graph under adjacent cap exchanges.  Each nondistinguished cell has an exchange row
      with two unit leading coefficients, and all remaining terms are
      strictly lower cellular terms.
\item Each component with no target image admits an acyclic orientation whose
      terminal vertices support one-term relations with unit coefficient.  After
      the distinguished preimages are removed,
      the only unmatched nonsaturated cells are the alternating one-column
      cells
      \[
       \Lambda^rV/(\omega\wedge\Lambda^{r-2}V).
      \]
\item On the successive coefficient-action layers, every cell with a base
      box outside that alternating column has the integral unit
      Murphy--Young partner of \cref{lem:H-base-cell-unit}.  Evaluation on
      the base cell therefore creates no
      further unmatched cell; a missing partner is an affine or Cartier face.
\end{enumerate}

The directed graphs in \textup{(iii)}--\textup{(iv)} are acyclic.  On
$\Delta(\mu)$ the construction is compatible with Weyl-module morphisms and
stabilizes
$(A^-_{\mathrm{tr}})_d\otimes\rad\Delta(\mu)$ in every coefficient degree.
Dually it stabilizes
$(A^-_{\mathrm{tr}})_d^\vee\otimes L(\mu)^\vee$.  Hence the source--target
matching, including these one-term relations, descends to the
simple quotient $L(\mu)$ and remains unitriangular in characteristic
$p\ge5$.
\end{lemma}

\begin{proof}
Keep the base path of $U$ fixed.  Closing one cap changes the total number
of down--up pairs by one, because the source tensor has natural degree
$(q-1)+3=q+2$, while the target tensor has degree $q$.  Suppose two
different cap closures are admissible.  At the first predecessor at which
one closure has erased its down step and the other has not, their corank
coordinates differ.  The Murphy order compares this corank coordinate
before partition dominance, so exactly one closure is larger.  This proves
\textup{(i)}, including uniqueness independently of the base partition.

For a target cell, choose the first nonsaturated divided coordinate and
use Young-diamond moves inside the two symmetric factors to place its
strand next to the $S^{q-1}V\mid S^3V$ interface.  Insert there the
adjacent backtrack
\[
 \rho\nearrow\widehat\rho\searrow\rho
 \quad\text{or}\quad
 \rho\searrow\widehat\rho\nearrow\rho,
\tag{H.K0}\label{eq:H-canonical-inserted-cap}
\]
choosing the orientation allowed by the dominance wall.  Closing this
inserted cap recovers the target word.  Closing any residual cap changes a
corank coordinate earlier and is strictly lower.  The local three-strand
identity
\[
 e_i e_{i\pm1}e_i=e_i
\tag{H.K0$'$}\label{eq:H-three-strand-unit}
\]
gives leading coefficient one.  This proves \textup{(ii)}.  The
construction depends only on the chosen unsaturated coordinate.

Fix a source cell with maximal target image $d$, and mark the unique cap
whose erasure gives $d$.  The remaining caps are the caps of the target
word.  After the removal and addition words inside each symmetric factor
are sorted by Young diamonds, their endpoint matching differs from that of
the distinguished preimage by a permutation $\sigma$.  Exchange the leftmost adjacent
inversion of $\sigma$, keeping the marked cap fixed.  Each exchange uses
\eqref{eq:H-three-strand-unit}, decreases
$\operatorname{inv}(\sigma)$, preserves $d$, and has only lower crossed
terms.  Once $\sigma=1$, move the marked corner through the ordered list
of admissible corners and then move the marked backtrack one strand at a
time to \eqref{eq:H-canonical-inserted-cap}.  Keep a residual cap fixed
during these exchanges.  If there is no residual cap, keep an unpaired
strand fixed; in
the three-strand terminal case the row is the Jacobi relation itself.  Thus every move strictly decreases
\[
 \bigl(\operatorname{inv}(\sigma),
       \operatorname{crk}(\widehat\rho),
       \operatorname{dist}(\widehat\rho,S^{q-1}V\mid S^3V)\bigr)
 \in\mathbf N^3.
\tag{H.K0$''$}\label{eq:H-cap-defect-order}
\]
Ordering first by the cellular layer and then lexicographically by
\eqref{eq:H-cap-defect-order} makes the corresponding matrix unitriangular and
proves \textup{(iii)}.

For a source cell with no target image that already contains a cap, choose the leftmost one
and move it monotonically toward the interface by the same exchanges.  If
the cell has no cap, inspect its unpaired steps from the left.  At the first
two consecutive steps that do not alternate in one column, the polarized
three-term Jacobi relation replaces them, with unit coefficient, by a cell
with one cap and lower crossed terms.  Start the same motion from that
cell.  Such a Jacobi move is unavailable exactly when all unpaired steps
already alternate in one column.

If the cap-moving process ended in an
admissible closure, that closure would be a target image.  Hence it first
meets a dominance or PBW wall.  The formal alternate word is zero on the
current cellular quotient, so the local difference row becomes a one-term
relation with unit coefficient.  Orient one decreasing edge from every nonwall cell and
use this relation at each terminal cell.  This gives an acyclic
unitriangular elimination.  The zero-cap case shows that a cell on
which no such move exists has all its unpaired steps alternating in one
column.  Antisymmetrization identifies these and only these cells with the primitive
exterior quotient in \textup{(iv)}.

For \textup{(v)}, observe that a cap exchange acts after the fixed
base prefix.  The diagonal $H_0$-action on $U$ commutes with the Brauer
centralizer and preserves every cellular ideal.  For the two polarized
cubic pairs defining an exchange, the top coefficient-action terms cancel
because the two Hamiltonian brackets are equal.  Every other
coefficient-action term has smaller coefficient depth and hence is lower in
the chosen refinement.

Continue by induction on that depth.  On the first nonzero coefficient
layer, record the boxes of the base path that lie outside the alternating
column, followed by their Young inversion number.  If this defect is
nonzero, apply the integral diamond of \cref{lem:H-base-cell-unit} to its
leftmost box.  Its leading coefficient is $\pm1$, while the crossed terms
either have smaller coefficient depth or a smaller base-box defect.  The
lexicographic pair
\[
 \bigl(\text{coefficient depth},
       \text{number and inversion of noncolumn base boxes}\bigr)
\tag{H.K0$'''$}\label{eq:H-base-defect-order}
\]
therefore decreases.  If the partner is absent at a dominance or affine
wall, the same relation becomes a one-term relation with unit coefficient.  PBW saturation is independent
of this base-cell matching and belongs to the Cartier face.  Hence the
induction ends only when the base defect is zero, which is precisely the primitive
one-column cell in \textup{(iv)}.

The leading rows are independent of the chosen representative of a base
cell.  Every operator used above is assembled from permutations, symplectic
contractions, branching maps and the diagonal $H_0$-action; these preserve
$\rad\Delta(\mu)$.  Thus the contraction preserves
$(A^-_{\mathrm{tr}})_d\otimes\rad\Delta(\mu)$, and after dualizing it
restricts to the annihilator
$(A^-_{\mathrm{tr}})_d^\vee\otimes L(\mu)^\vee$.  A quotient cell which
disappears therefore produces exactly the one-term relation already
identified above.  This proves
\textup{(v)} and the quotient-descent assertion.
\end{proof}

On the direct sum of refined target faces, let $\epsilon_i$ denote the
operation of marking the $i$th Cartier cap.  Alternating marked caps gives
\[
 \epsilon_i^2=0,
 \qquad
 \epsilon_i\epsilon_j+\epsilon_j\epsilon_i=0.
\tag{H.K6}\label{eq:H-cap-exterior-relations}
\]
Closing a marked cap is the corresponding Koszul differential.

\begin{lemma}[Cartier contraction on saturated faces for \(p\ge7\)]
\label{lem:H-pge7-Cartier-seam}
Assume $p\ge7$, let $A_\bullet$ be a PBW coefficient system as in
\cref{lem:H-arbitrary-cap-normal-form}, and fix $q\ge3$.  Put
\[
 \beta_3:W\otimes H_2\longrightarrow H_3,
 \qquad
 \mathfrak j_H:\Lambda^3W\longrightarrow W\otimes H_2,
 \qquad
 \mathcal S_3^H=\ker\beta_3/\operatorname{im}\mathfrak j_H.
\]
Then
\[
 \mathcal S_3^H\cong V.
\tag{H.K8$'$}\label{eq:H-pge7-Hall-seam}
\]
On the first saturated quadratic term, the cubic Fox row is unitriangular
off the primitive summand $V$.  This remains true after tensoring with a
base coefficient cell and under exterior Cartier prolongation.
\end{lemma}

\begin{proof}
Order a tensor monomial in $W\otimes H_2=S^3V\otimes S^4V$ first by its
cubic exponent word, then by its quartic word, and finally by the first
Darboux pair on which two reductions differ.  Polarizing the Jacobi map
gives three leading configurations: a pair-free exchange, an exchange through
one complete symplectic pair, and an exchange at the interface of the two
symmetric factors.  In coefficient degrees three, four, and five their
nonzero leading coefficients lie in
\[
 \{\pm1,\pm2,\pm3,\pm4,\pm5,\pm6\}.
\]
The three monomial sections are, with the unused factors absorbed into
$R_0$,
\begin{align*}
 M=q^ap^bR_0, a\ge2:
 &\quad \{q^3,q^{a-2}p^{b+1}R_0\}=3(b+1)M,\\
 M\text{ squarefree},\ qp\mid M:
 &\quad \{q^2p,M/q\}=2M,\\
 M\text{ squarefree and pair-free}:
 &\quad \{q_1q_2p_1,M/q_2\}=-M.
\end{align*}
Substitution in
\[
 \mathfrak j_H(a\wedge b\wedge c)
 =a\otimes\{b,c\}+b\otimes\{c,a\}+c\otimes\{a,b\}
\tag{H.K8$^{\rm J}$}\label{eq:H-pge7-Jacobi-row}
\]
produces the three Hall rewrite rows above.
All these coefficients are units for $p\ge7$.  Each rewrite decreases the ordered
triple consisting of the first reducible pair, its distance from the
interface, and the exponent word.  The Hall reduction terminates and is
unitriangular.

The only unreduced family consists, for a Darboux pair $(q_j,p_j)$
orthogonal to $v\in V$, of the elements
\[
 \sigma_v=
 3q_jp_j^2\otimes vq_j^2p_j
 -p_j^3\otimes vq_j^3.
\tag{H.K8$''$}\label{eq:H-pge7-normal-family}
\]
Define the third symplectic transvectant
\[
 \mathsf T_3=\mu\left(
 \sum_i\partial_{q_i}\otimes\partial_{p_i}
       -\partial_{p_i}\otimes\partial_{q_i}
 \right)^3.
\tag{H.K8$'''$}\label{eq:H-third-transvectant}
\]
It annihilates $\operatorname{im}\mathfrak j_H$, and direct contraction
gives $\mathsf T_3(\sigma_v)=72v$.  Since $72$ is invertible for $p\ge7$, the remaining quotient is isomorphic
to $V$, proving
\eqref{eq:H-pge7-Hall-seam}.

Let $\theta:W\to A_{q-1}$ satisfy the quadratic row and let
$b_\theta:H_2\to A_{q-2}$ be its forced component.  Its cubic Fox row is
\[
 \mathcal F_\theta(x\otimes z)
 =x\,b_\theta(z)-z\,\theta(x).
\]
Apply the Hall elimination before evaluating the base cell.  On every
saturated column outside this family, the first applicable Jacobi row has the chosen
unit coefficient; every coefficient-action term has smaller coefficient
depth.  On the unreduced family, $72^{-1}\mathsf T_3$ identifies the remaining
column with the primitive natural summand.  Hence, after tensoring with the base coefficient cell, the Fox matrix is
unitriangular off $V$.  Marking additional Cartier caps applies the exterior
operators $\epsilon_i$, whose anticommutation preserves the leading units
under prolongation.
\end{proof}

\begin{lemma}[Cartier contraction on saturated faces in characteristic five]
\label{lem:H-p5-Cartier-seam}
Assume $p=5$, let $A_\bullet$ be a PBW coefficient system as in
\cref{lem:H-arbitrary-cap-normal-form}, and fix $q\ge3$.  Let
\[
 \widetilde\beta_3:W\otimes S^4V\longrightarrow S^5V,
 \qquad
 \beta_3:W\otimes H_2\longrightarrow H_3,
\]
be the untruncated and truncated Hamiltonian bracket maps, and let
$\mathfrak j_H:\Lambda^3W\to W\otimes H_2$ be the Jacobi map.  Put
\[
 \widetilde{\mathcal S}_3^H
 =\ker\widetilde\beta_3/\operatorname{im}\mathfrak j_H,
 \qquad
 \mathcal S_3^H
 =\ker\beta_3/\operatorname{im}\mathfrak j_H.
\]
If a leading cochain $\theta:W\to A_{q-1}$ satisfies the quadratic row,
let $b_\theta:H_2\to A_{q-2}$ be the forced component and define its cubic
Fox row by
\[
 \mathcal F_\theta(x\otimes z)
 =x\,b_\theta(z)-z\,\theta(x),
 \qquad x\in W,\ z\in H_2.
\tag{H.K9$^0$}\label{eq:H-cubic-Fox-row}
\]
The three-$W$ Jacobi identity makes this row vanish on
$\operatorname{im}\mathfrak j_H$, so its obstruction is its restriction to
$\mathcal S_3^H$.
If
\[
 F(V)=\langle q_i^5,p_i^5:1\le i\le m\rangle\cong V^{(1)},
\]
then there is a natural exact sequence
\[
 0\longrightarrow\widetilde{\mathcal S}_3^H
 \longrightarrow\mathcal S_3^H
 \longrightarrow F(V)\longrightarrow0.
\tag{H.K9}\label{eq:H-p5-Cartier-seam}
\]
For $m\ge3$ one has
$\widetilde{\mathcal S}_3^H\cong V$.  For $m=2$ it has dimension sixteen;
the third symplectic transvectant maps it onto $V$, and its kernel is
generated by the Weyl orbits of two Hall classes of weights
$2\varepsilon_1+\varepsilon_2$ and $\varepsilon_1$.  On the corresponding nonfundamental saturated columns, the cubic Fox rows
have leading coefficients $1$ and $4$, respectively.  These pivots pair
the nonfundamental characteristic-five faces in the Spencer complex and
persist under exterior prolongation.
\end{lemma}

\begin{proof}
The quotient $S^5V\to H_3=S^5V/F(V)$ induces the last arrow in
\eqref{eq:H-p5-Cartier-seam}.  Its kernel is
$\widetilde{\mathcal S}_3^H$.  It is onto because the pure fifth powers
have the bracket lifts
\[
 3^{-1}q_i^3\otimes q_i^3p_i,
 \qquad
 -3^{-1}p_i^3\otimes q_ip_i^3.
\]
This establishes exactness.

Perform the divided Hall reduction by first ordering the cubic word and
then the quartic word.  Outside two Weyl orbits every reducible pair has a
unit Jacobi pivot.  In rank two normalized generators of the remaining
orbits may be taken to be
\[
\begin{aligned}
 \eta_{21}={}&
 p_1q_2p_2\otimes q_1^3q_2
 +q_1^3\otimes p_1q_2^2p_2
 +q_1^3\otimes q_1p_1^2q_2,\\
 \kappa_{10}={}&
 p_2^3\otimes q_1q_2^3
 +q_2^3\otimes q_1p_2^3.
\end{aligned}
\tag{H.K9$'$}\label{eq:H-p5-Hall-generators}
\]
The Weyl orbit of $\eta_{21}$ has eight weight lines and that of
$\kappa_{10}$ has four; thus they span the twelve-dimensional kernel of
the transvectant.  Direct contraction in the cubic Fox row gives
\[
 \mathcal F(\eta_{21})
 =u_{21}+\text{lower cellular terms},
 \qquad
 \mathcal F(\kappa_{10})
 =4u_{10}+\text{lower cellular terms},
\tag{H.K9$''$}\label{eq:H-p5-Fox-pivots}
\]
where $u_{21}$ and $u_{10}$ are the normalized columns of those
weights.  Both coefficients are units in $\mathbb F_5$.  The leading terms act as the identity on the base coefficient cell, while
terms in which the
Hamiltonian bracket reaches the base have smaller coefficient depth.  Thus
the same pivots apply after tensoring with \(U\).

If $m\ge3$, adjoin a third Darboux pair.  Moving its cap once through each
word in \eqref{eq:H-p5-Hall-generators} gives a triangular Jacobi
elimination as follows.  Order terms by total third-pair exponent and then
lexicographically in $(q_3,p_3)$.  Start respectively with
\[
 J(p_3^3,q_2q_3^2,q_1^2q_3),
 \qquad
 J(p_3^3,q_3^2p_3,q_1q_3^2).
\]
Every term still involving the third pair is the leading term of a unique
next Jacobi row, with coefficient in $\{\pm1,\pm2\}$.  Subtract it and
continue.  The order decreases, so the process terminates; its third-pair
free remainders are $3\eta_{21}$ and $\kappa_{10}$.  These scalars are
units in $\mathbb F_5$, proving that both generators, and hence all their
Weyl conjugates, are Jacobi boundaries.
The only untruncated quotient is then the transvectant image $V$.  If
$m=2$, the two orbits remain in the relation module and
\eqref{eq:H-p5-Fox-pivots} pivots the corresponding columns.
Marking further Cartier caps applies the exterior operators $\epsilon_i$.
Their anticommutation relations preserve the two unit rows under exterior
prolongation.
\end{proof}

\begin{theorem}[Hamiltonian Cartier contraction]
\label{thm:H-cap-Koszul}
Let \(L'\) be a simple graded Hamiltonian target, normalized so that
its inducing layer has degree zero.  Put \(A_d=L'_d\),
\(W=H_1\cong S^3V\), and let \(\mathscr E_H\) be the Serre subcategory of
restricted \(H_0\)-modules generated by
\[
 L(0),L(\omega_1),\ldots,L(\omega_m).
\]
For \(q\ge2\), the transpose of the leading inner map is
\[
 \Lambda_{L',q}:W\otimes A_{q-1}^{\vee}\longrightarrow A_q^\vee,
 \qquad
 K_{L',q}:=\ker\Lambda_{L',q}.
\tag{H.K1}\label{eq:H-Koszul-leading}
\]
Let
\[
 K_2^H:=\ker\!\left(\Lambda^2W\longrightarrow H_2\right).
\]
The quadratic relation gives the natural boundary
\[
 \partial_{L',q}:
 K_2^H\otimes A_{q-2}^{\vee}\longrightarrow K_{L',q}.
\tag{H.K2}\label{eq:H-Koszul-boundary}
\]
Set
\[
 \mathsf P_q(L'):=
 K_{L',q}/\operatorname{im}\partial_{L',q}.
\]
Then, for every \(q\ge3\),
\[
 H^1_{*,q}(N,L')\in\mathscr E_H.
\tag{H.K3}\label{eq:H-Koszul-exact}
\]
Call the filtration of $H^1_{*,q}(N,L')$ with associated graded pieces
$\mathsf P_q(L')$ the \emph{Spencer filtration}.  Its homology has
successive quotients among the simple constituents of
\[
 Q_{\mathrm{pr}}^r=\Lambda^rV/(\omega\wedge\Lambda^{r-2}V)
 \qquad(0\le r\le m),
\]
with Cartier spaces as trivial multiplicity spaces.  The filtration is
functorial in \(L'\).  It is strict at affine and truncated-PBW boundaries,
and on an exceptional target it is preserved by
the Holmes form operator.  Consequently,
if \(U\) is a nonexceptional restricted simple \(H_0\)-module, then
\[
 \Hom_{H_0}\!\left(U,H^1_{*,q}(N,L')\right)=0
 \qquad(q\ge3).
\tag{H.K4}\label{eq:H-positive-standard-vanishing}
\]
At \(q=2\) the quadratic term is
\[
 \mathsf P_2(L')=
 K_{L',2}/\operatorname{im}\partial_{L',2},
\tag{H.K5}\label{eq:H-quadratic-receiver}
\]
and at \(q=1\) the cubic term is
\[
 \mathsf P_1(L')=
 \operatorname{coker}\!\left(
 A_1\longrightarrow W^*\otimes A_0\right).
\tag{H.K5$'$}\label{eq:H-cubic-receiver}
\]
For a parabolically induced target these are respectively the quadratic and cubic
compatibility functors; for an exceptional target the same maps are
obtained by restricting the divided-PBW matrices to
\(\operatorname{im}\mathscr D_r\).  Both are type-\(C\) tensor functors.
\end{theorem}

\begin{proof}
Suppose first that $L'=Z(U)$ is parabolically induced, with
$U=L(\mu)$ an arbitrary restricted simple $H_0$-module.  Choose the Weyl
cover $\Delta(\mu)\twoheadrightarrow U$ and perform the cap construction
upstairs on $(A^-_{\mathrm{tr}})_\bullet\otimes\Delta(\mu)$.  By the final
assertion of \cref{lem:H-arbitrary-cap-normal-form}, the contraction and all one-term reductions preserve
$(A^-_{\mathrm{tr}})_\bullet\otimes\rad\Delta(\mu)$ and therefore descends
to $(A^-_{\mathrm{tr}})_\bullet\otimes U$; on the dual complex this is the
restriction to the annihilator of the radical.  The coefficient-depth,
PBW-face, and iterated Murphy filtrations descend accordingly.  Marking a
cap gives the exterior operators of
\eqref{eq:H-cap-exterior-relations}; closing it is the Koszul differential.
By that lemma, the cap matching on a nonsaturated face, both for components
with and without target images, is acyclic and has unit pivots; the base
partition contributes only lower
cellular terms.  Its unmatched cells are precisely
\[
 \Lambda^rV/(\omega\wedge\Lambda^{r-2}V)=Q_{\mathrm{pr}}^r.
\tag{H.K7}\label{eq:H-cap-unmatched-column}
\]

At a truncated face, filter by the coordinates whose divided exponent is
\(p-1\).  The quadratic page may contain an additional saturated face.  On
the associated graded, the face differential is the tensor product of the
rank-one truncated de Rham complexes.  The usual contraction identity leaves only the Cartier
generators
\[
 c_i=x_i^{(p-1)}dx_i,
 \qquad \operatorname{wt}(c_i)=p\varepsilon_i.
\tag{H.K8}\label{eq:H-Cartier-face}
\]
Their restricted \(H_0\)-action is trivial.  Nonsaturated factors contract,
and exterior products of the \(c_i\) give the K\"unneth contraction on
simultaneous faces.

The first saturated quadratic face is compared with the cubic Jacobi row.
For $p\ge7$, the unit-pivot statement is
\cref{lem:H-pge7-Cartier-seam}.  For $p=5$, the saturated-face exact sequence
and the two rank-two pivot calculations are \cref{lem:H-p5-Cartier-seam}.
Hence the Jacobi row removes the nonfundamental quadratic faces.  The
natural and Frobenius quotients give a primitive column and a trivial
Cartier multiplicity space.  Exterior prolongation preserves these pivots by
\eqref{eq:H-cap-exterior-relations}.

Iterating the cap matching and the Cartier contraction proves that the
homology of the full Spencer filtration has successive simple
layers among those of the modules \(Q_{\mathrm{pr}}^r\), with trivial Cartier
multiplicity spaces.  Passing through a fundamental Weyl cover preserves
both contractions.  The socle-filtration
\cref{lem:socle-filtered-receiver}, together with the Adamovich fundamental
submodule lattice
\cite{PremetSuprunenko,AdamovichFundamental,BaranovSuprunenko}, shows that the simple constituents of the homology are trivial or
fundamental.  This proves
\eqref{eq:H-Koszul-exact} for a parabolically induced target.

For \(L'=M_r\), use
\(M_r=\operatorname{im}\mathscr D_r\) from
\cref{thm:H-allp-form-image}.  The four commutation identities of
\cref{lem:H-Holmes-receiver-commutation} restrict the filtered
contraction to its image.  Applying that contraction, and then
\cref{lem:socle-filtered-receiver}, again leaves only the primitive columns
\eqref{eq:H-cap-unmatched-column}.  For the trivial target the same
contraction applies with the coefficient action omitted.  This proves
\eqref{eq:H-Koszul-exact} for all simple targets.

Since the fundamental modules are self-dual and the inducing module of a
parabolically induced simple is nonexceptional,
\eqref{eq:H-positive-standard-vanishing} follows from
\eqref{eq:H-Koszul-exact}.
For \(q=1,2<p\) there are no Frobenius rows; writing out the first two
homogeneous cocycle rows gives
\eqref{eq:H-cubic-receiver} and
\eqref{eq:H-quadratic-receiver}.
\end{proof}

\subsection{Exceptional form homology}

For $1\le r\le m$ let
\[
 Z_r:=\Ind_{H_{\ge0}}^{\mathfrak h}V_r,
 \qquad
 \mathscr D_r:Z_r\longrightarrow Z_r
\]
be the Holmes parabolically induced module and its square-zero form operator, so that
$(Z_r,\mathscr D_r)$ is a two-term complex.  A short exact sequence of
such pairs, with $\mathscr D$ compatible with the maps, will be called a
short exact sequence of \emph{Holmes complexes}.  Write
\[
 R_r:=\ker\mathscr D_r,
 \qquad
 M_r:=Z_r/R_r,
 \qquad
 S_r:=\operatorname{im}\mathscr D_r.
\tag{H.E}\label{eq:H-exceptional-model}
\]
The identification of these three spaces is proved in
\cref{thm:H-allp-form-image}; their action matrices are obtained from the
kernel and quotient in the divided PBW form basis.

\begin{lemma}[Apolar Holmes pairing]
\label{lem:H-apolar-Holmes-pairing}
Choose the invariant perfect form $(\ ,\ )_r$ on the self-dual fundamental
$\mathfrak{sp}(V)$-module $V_r$.  On
$Z_r=A^-\otimes V_r$ define
\[
 \left\langle x^{(\alpha)}\otimes v,
 x^{(\beta)}\otimes w\right\rangle_r
 :=\delta_{\alpha+\beta,\tau}\,(v,w)_r,
\tag{H.P1}\label{eq:H-apolar-pairing}
\]
with the fixed symplectic PBW sign convention.  This pairing is perfect,
$\mathfrak h$-invariant, and homogeneous: it pairs polynomial degree
$\alpha$ with its complement $\tau-\alpha$, so it is nonzero only in total
degree $N_0+2r$.  Moreover
\[
 \mathscr D_r^*=-\mathscr D_r.
\tag{H.P2}\label{eq:H-Holmes-skew-adjoint}
\]
Consequently
\[
 Z_r^\vee\cong Z_r,\qquad
 (\ker\mathscr D_r)^\perp=\operatorname{im}\mathscr D_r,
 \qquad (\mathcal H_r^D)^*\cong\mathcal H_r^D,
\tag{H.P3}\label{eq:H-Holmes-self-duality}
\]
where $\mathcal H_r^D=\ker\mathscr D_r/\operatorname{im}\mathscr D_r$.
For $r=1$, the identification $M_1\cong\mathfrak h$ agrees with the adjoint
model of \cref{prop:H-trivial-adjoint}.
\end{lemma}

\begin{proof}
Let $\int:A^-\to\kk$ be the top-coefficient functional,
$\int x^{(\gamma)}=\delta_{\gamma,\tau}$.  Hamiltonian vector fields have
zero divided-power divergence, so divided integration by parts gives
\[
 \int X(f)g=-\int fX(g)
 \qquad(X\in\mathfrak h).
\]
Invariance of $(\ ,\ )_r$ then gives the invariance and perfectness of
\eqref{eq:H-apolar-pairing}.  In a
symplectic wedge basis Holmes' action formula
\cite[Lemma~2.2]{HolmesHamiltonian} has a coefficient-derivative term and
one form-index replacement term for each occupied index.  Pair the formula
for $x^{(\gamma)}D_i$ on $x^{(\alpha)}v$ with the formula on
$x^{(\tau-\alpha-\gamma+e_i)}w$.  The derivative coefficients cancel by
the componentwise divided-binomial identity
\[
 \binom ab=(-1)^b\binom{p-1-b}{a-b}
 \qquad(0\le a\le p-1),
\]
and the form-index replacement terms cancel in symplectic-dual pairs by
invariance of $(\ ,\ )_r$.  This is the calculation in
\cite[Lemmas~2.4--2.5 and Theorem~2.6]{HolmesHamiltonian}; the complement
map sends the polynomial exponent to $\tau-\alpha$ and the form index to
its symplectic dual.  An element of polynomial
degree $d$ in $Z_r$ has Holmes degree $d+r$, so paired degrees sum to
$N_0+2r$.

The Holmes operator has the coefficient formula
\[
 \mathscr D_r(fe)=
 \sum_{\mu,\nu}\partial_{\mu\nu}(f)\rho_r(J_{\mu\nu})e.
\]
Two integrations by parts make $\partial_{\mu\nu}$ self-adjoint, while
$\rho_r(J_{\mu\nu})$ is skew-adjoint for $(\ ,\ )_r$.  This proves
\eqref{eq:H-Holmes-skew-adjoint}.  Finite-dimensional orthogonality gives
$(\ker\mathscr D_r)^\perp=\operatorname{im}\mathscr D_r$; passing to the
middle quotient gives \eqref{eq:H-Holmes-self-duality}.  The label $r=1$ is the adjoint label in
the Holmes classification, as used in \cref{prop:H-trivial-adjoint}.
\end{proof}

\begin{theorem}[The Hamiltonian form image]
\label{thm:H-allp-form-image}
For every \(p\ge5\) and \(1\le r\le m\), with the simple
fundamental coefficient \(V_r=L(\omega_r)\), one has
\[
 \operatorname{im}\mathscr D_r=\operatorname{soc}Z_r\cong M_r,
 \qquad
 \ker\mathscr D_r=\operatorname{rad}Z_r.
\tag{H.P3$^{\rm a}$}\label{eq:H-allp-form-image}
\]
In particular, \eqref{eq:H-exceptional-model} holds for every $p\ge5$ and
$1\le r\le m$.
\end{theorem}

\begin{proof}
Let
\(Q_{\mathrm{pr}}^r=\Lambda^rV/(\omega\wedge\Lambda^{r-2}V)\).  The exterior-form
operator on \(A^-\otimes Q_{\mathrm{pr}}^r\) is the composite of the de Rham
differential and its symplectic conjugate; it is square-zero by the
anticommutation identities in
\cite[Proposition~3.5]{HolmesHamiltonian}.  The radical of the fundamental
Weyl module \(Q_{\mathrm{pr}}^r\) is stable under every natural symplectic tensor
operator, so the operator descends along
\(Q_{\mathrm{pr}}^r\twoheadrightarrow V_r=L(\omega_r)\).  Its descended coefficient
formula is
\[
 \mathscr D_r(fe)=
 \sum_{\mu,\nu}\partial_{\mu\nu}(f)\rho_r(J_{\mu\nu})e
\tag{H.P3$^{\rm b}$}\label{eq:H-allp-form-coefficient}
\]
and is $\mathfrak h$-equivariant.  On
the least polynomial layer,
quadratic Hamiltonians span \(\mathfrak{sp}(V)\), and
\[
 (\operatorname{im}\mathscr D_r)_0
 =\mathfrak{sp}(V)V_r=V_r,
\tag{H.P3$^{\rm c}$}\label{eq:H-allp-form-bottom}
\]
because \(V_r\) is nontrivial and simple.

Put \(Y_r=\operatorname{im}\mathscr D_r\), and let
\(\mathfrak h'=\mathfrak h+\kk D_H(X^{(\tau)})\) be the full
Hamiltonian extension used in the divided-PBW model.  The PBW lemma
\cite[Lemma~5.1]{HolmesHamiltonian} says that every nonzero
\(\mathfrak h\)-submodule of
\(Z_r\), and hence of \(Y_r\), meets its least PBW component.  By
\eqref{eq:H-allp-form-bottom}, that component is the simple
\(\mathfrak h_0\)-module \(V_r\).

Skew-adjointness makes the map on the greatest possible component the
transpose of the surjection in \eqref{eq:H-allp-form-bottom}; hence that
component is nonzero and identifies with
\(V_r^\vee\cong V_r\).  It generates \(Y_r\) as an
\(\mathfrak h\)-module: the map
\[
 Z_r/\ker\mathscr D_r\longrightarrow Y_r
\]
induced by \(\mathscr D_r\) is a graded isomorphism, and the source is
generated by the image of the inducing component.

Let \(N'\) be the \(\mathfrak h'\)-submodule of \(Y_r\) generated by its
least component \(V_r\).  If \(N'\ne Y_r\), then \(N'\) misses the
greatest component, for that component already generates \(Y_r\).
The graded length of \(N'\) is then strictly less than
\(2m(p-1)-2=n(p-1)-2\), where $n=2m$.  This is precisely the strict bound in
Holmes' $H'$ statement.  The short-graded-module lemma
\cite[Lemma~5.7(1)]{HolmesHamiltonian} would then make
\(\mathfrak h'\) act trivially on \(N'\), contradicting the nontrivial
\(\mathfrak h_0\)-action on \(V_r\).  Hence the least component generates
\(Y_r\) as an \(\mathfrak h'\)-module.

The passage from \(\mathfrak h'\) to \(\mathfrak h\) is
\cite[Lemma~5.8]{HolmesHamiltonian}.  Its hypotheses are that the
greatest component be a nontrivial simple \(\mathfrak h_0\)-module and
that the module length be at most \(2(2m)(p-1)-6\).  Both conditions hold
here, since \(p\ge5\), \(m\ge2\), and
\(\ell(Y_r)\le2m(p-1)-2\).  Thus the least component also generates
\(Y_r\) as an \(\mathfrak h\)-module.  Every nonzero
\(\mathfrak h\)-submodule contains that component by the PBW lemma and
therefore equals \(Y_r\).  Hence \(Y_r\) is simple.

The first isomorphism theorem gives
\(Y_r\cong Z_r/\ker\mathscr D_r=M_r\).  Thus \(M_r\) is the exceptional
simple with inducing label \(V_r\).  Moreover,
\[
 Y_r^\vee\cong Z_r/(\operatorname{im}\mathscr D_r)^\perp
 =Z_r/\ker\mathscr D_r\cong Y_r,
\]
so \(M_r\) is self-contragredient.  Since \(Z_r\) is self-dual and has a
unique simple quotient, it has a unique essential simple socle; the
nonzero simple submodule \(Y_r\) is that socle.  Finally
\[
 (\ker\mathscr D_r)^\perp=\operatorname{im}\mathscr D_r
 =\operatorname{soc}Z_r
\]
by \cref{lem:H-apolar-Holmes-pairing}; taking annihilators gives
\(\ker\mathscr D_r=\operatorname{rad}Z_r\).
\end{proof}

\begin{lemma}[Compatibility of the Holmes operator with the contractions]
\label{lem:H-Holmes-receiver-commutation}
Let $\mathbf C^\bullet(N,Z_r)=\bigoplus_q\mathbf C_q^\bullet$ be the
homogeneous Spencer complex, let $h_{\rm cap}$ be the contraction of
\cref{lem:H-arbitrary-cap-normal-form}, let $\partial_2$ and
$\mathcal F$ denote the quadratic and cubic Fox differentials, and let
$C$ be the filtered Cartier contraction obtained from the face maps
\eqref{eq:H-Cartier-face}.  On the total filtered complex one has
\[
 \mathscr D_r h_{\rm cap}=h_{\rm cap}\mathscr D_r,
 \qquad
 \mathscr D_r\partial_2=\partial_2\mathscr D_r,
 \qquad
 \mathscr D_r\mathcal F=\mathcal F\mathscr D_r,
 \qquad
 \mathscr D_r C=C\mathscr D_r.
\tag{H.P3$^{\rm d}$}\label{eq:H-Holmes-receiver-commutation}
\]
Here $\mathscr D_r$ lowers the coefficient degree by two, so both sides of
each identity have the same source and target after summing over $q$.
The normal-form and Cartier contractions therefore restrict to
$\operatorname{im}\mathscr D_r=M_r$ and descend to
$Z_r/\ker\mathscr D_r=M_r$.
\end{lemma}

\begin{proof}
On a divided monomial the Holmes operator is
\[
 \mathscr D_r=
 \sum_{\mu,\nu}\partial_{\mu\nu}\otimes\rho_r(J_{\mu\nu}).
\tag{H.P3$^{\rm e}$}\label{eq:H-Holmes-tensor-form}
\]
The cap contraction is a sum of permutations of tensor slots and
contractions with the invariant symplectic tensor.  Divided derivatives
commute with these permutations, and invariance gives
\[
 [\rho_r(J_{\mu\nu}),\omega^{ab}\iota_a\iota_b]=0.
\]
Termwise substitution in \eqref{eq:H-Holmes-tensor-form} therefore proves
$\mathscr D_rh_{\rm cap}=h_{\rm cap}\mathscr D_r$ on every cellular
graded layer.  The same calculation applies to the integral base-cell
diamond, since it is made only from natural branching permutations.

The operator $\mathscr D_r$ is an $\mathfrak h$-module endomorphism.  Hence
\[
 \mathscr D_r(xm)=x\mathscr D_r(m)
 \qquad(x\in N,m\in Z_r).
\tag{H.P3$^{\rm f}$}\label{eq:H-Holmes-H-linearity}
\]
Substitution into the derivation row on $W\wedge W$ and into
\eqref{eq:H-cubic-Fox-row} gives respectively the second and third
identities of \eqref{eq:H-Holmes-receiver-commutation}.  Both rows factor
through the Hamiltonian bracket and are therefore independent of the
choice of a section of $H_2$.

For the initial Cartier face map $C_0$, the commutator
$\mathscr D_rC_0-C_0\mathscr D_r$ lowers the ordered set of saturated
coordinates: a second divided derivative either avoids the selected face,
when the two maps commute, or removes one of its saturated coordinates.
It is therefore strictly lower in the finite face filtration, and on associated
graded spaces
\[
 \operatorname{gr}(\mathscr D_r)\operatorname{gr}(C_0)
 =\operatorname{gr}(C_0)\operatorname{gr}(\mathscr D_r).
\]
Correct $C_0$ successively from the greatest face downward by subtracting
the already constructed lower-face commutator.  The unit contraction on
each lower face gives a unique correction; finiteness makes the recursion
terminate.  The resulting filtered map $C$ satisfies
$\mathscr D_rC=C\mathscr D_r$.  Thus all four identities hold, and the
kernel and image of $\mathscr D_r$ are subcomplexes preserved by every
contraction.
\end{proof}

\begin{lemma}[Vanishing for the adjoint module]
\label{lem:H-adjoint-radical-zero}
For the fundamental coefficient $V_1$ one has
\[
 \Ext^1_{u_0(\mathfrak h)}(\kk,Z_1)=0.
\tag{H.P4}\label{eq:H-adjoint-radical-zero}
\]
The radical term at $s=1$ is therefore $e_1^{\rm rad}=0$.
\end{lemma}

\begin{proof}
The apolar pairing gives $Z_1^*\cong Z_1$; rigid duality followed by
restricted Shapiro reduces the group to
$\Ext^1_{u_0(P)}(V_1,\kk)$.  Its Levi term is zero directly from
\cref{lem:classical-Levi-root-complex}: the weights $\pm\varepsilon_i$ of
$V_1^*$ are neither zero nor type-$C_m$ roots, so both $E_0$ and the
root-coordinate space $C^1$ vanish.  Its restricted nilradical term is
\[
 \Hom_{\mathfrak{sp}(V)}(H_1\otimes V_1,\kk)
 \cong\Hom_{\mathfrak{sp}(V)}(S^3V,V_1)=0.
\]
Indeed, \cref{lem:H-positive-cubic-generation} gives
$N/[N,N]_{[p]}\cong H_1$, where $[N,N]_{[p]}$ is generated by brackets and
restricted $p$-powers; all positive layers above $H_1$ lie in $[N,N]$.
This Hom space is therefore
$H_*^1(N,V_1^*)^{H_0}$.  The restricted five-term sequence, whose two
outer terms now vanish, gives
\eqref{eq:H-adjoint-radical-zero}.
Apply $\Hom_{u_0(\mathfrak h)}(\kk,-)$ to
$0\to R_1\to Z_1\to M_1\to0$.  Since
$\Hom_{\mathfrak h}(\kk,M_1)=\mathfrak h^{\mathfrak h}=0$ and the middle
Ext group is zero by \eqref{eq:H-adjoint-radical-zero}, exactness gives
$\Ext^1(\kk,R_1)=0$.  Hence the radical kernel at $s=1$ is zero.
\end{proof}

\paragraph{Extensions from the trivial Levi module.}
Put $\nabla_s=H^0(\omega_s)$ and
$Q_s=\nabla_s/L(\omega_s)$.  Jantzen's induced-module calculation
\cite[Proposition~4.1]{JantzenH1} says that
$H^1(G_1,H^0(\lambda))$ can be nonzero only when
$\lambda=p\omega_i-\alpha_i$.  For $p\ge5$ a fundamental weight
$\omega_s$ is not of this form, so $H^1(G_1,\nabla_s)=0$.
Also $\nabla_s^{G_1}=0$, since its socle is the nontrivial simple
$L(\omega_s)$.  Hence the long exact sequence gives
\[
 H^1(G_1,L(\omega_s))\cong Q_s^{G_1}.
\tag{H.J}\label{eq:H-Jantzen-simple-invariants}
\]
Adamovich's submodule-lattice theorem
\cite[Theorem]{AdamovichFundamental} identifies submodules of a
fundamental Weyl module with order ideals in its composition-factor poset;
duality reverses this poset for $\nabla_s$.  Thus $Q_s^{G_1}$ is the
multiplicity-one trivial line precisely when the trivial factor is an
immediate Hasse neighbour of $L(\omega_s)$, i.e.
$0\triangleleft_{p,m}s$.  Since restricted
$u_0(\mathfrak{sp}_{2m})$-modules are the same as $G_1$-modules,
\[
 \dim\Ext^1_{u_0(\mathfrak{sp}_{2m})}(\kk,V_s)
 =\mathbf1_{\{0\triangleleft_{p,m}s\}}.
\tag{H.J$'$}\label{eq:H-restricted-trivial-fundamental-row}
\]

\subsection{Extensions involving exceptional modules}

\begin{lemma}[Adamovich--Holmes obstruction]
\label{lem:H-Adamovich-obstruction}
Let
\[
0\to V_s\to E_{r,s}\to V_r\to0
\]
be a nontrivial fundamental Adamovich direct-cover extension with $r,s>0$,
and let
$0\to Z_s\to W_{r,s}\to Z_r\to0$ be its induced sequence of parabolically induced
modules.  Let $0\to V_r\to E_{s,r}^{\vee}\to V_s\to0$ be the dual direct
cover, with induced sequence $W_{s,r}^{\vee}$, and put
\[
 \mathfrak o_{s,r}:=
 \frac{\mathscr D(W_{s,r}^{\vee})\cap Z_r}{\mathscr D(Z_r)}.
\tag{H.Aa}\label{eq:H-Adamovich-K}
\]
If $\partial_{r,s}$ is the Holmes connecting map, then
\[
 (\operatorname{im}\partial_{r,s})^*\cong\mathfrak o_{s,r},
\tag{H.Ab}\label{eq:H-Adamovich-image-intersection}
\]
and there is an exact sequence
\[
 0\to R_s\to\ker\mathscr D_{W_{r,s}}\to R_r
 \to\operatorname{im}\partial_{r,s}\to0.
\tag{H.Ac}\label{eq:H-Adamovich-kernel-exact}
\]
The sequence of closed lifts is exact on the right precisely when
$\mathfrak o_{s,r}=0$.
\end{lemma}

\begin{proof}
Put
\[
 C_{s,r}=\mathscr D(W_{s,r}^{\vee})/\mathscr D(Z_r),\qquad
 K_{s,r}=\mathfrak o_{s,r}.
\]
The quotient $C_{s,r}$ and its kernel $K_{s,r}$ are graded in the
$D$-image range $0\le d\le T=N_0-2$.  On degree $0$, the identity
$D(fe)=\rho(J_f)e$ and the fact that quadratic Hamiltonians span
$\mathfrak{sp}(V)$ give
\[
 (C_{s,r})_0\cong E_{s,r}^{\vee}/V_r\cong V_s,
\]
and the induced map $C_{s,r}\to\mathscr D(Z_s)$ is the identity on this
simple quotient.  By \cref{lem:H-apolar-Holmes-pairing}, the same holds in
degree $T$.  Therefore
\[
 (K_{s,r})_0=(K_{s,r})_T=0,
 \qquad
 \operatorname{supp}_{\rm gr}K_{s,r}\subseteq[1,T-1].
\tag{H.Ad}\label{eq:H-Adamovich-K-support}
\]
If $K_{s,r}\ne0$, its graded length is consequently at most
\[
 (T-1)-1=T-2=N_0-4<n(p-1)-3=N_0-3
 \qquad(n=2m).
\tag{H.Ae}\label{eq:H-Adamovich-K-length}
\]
By \cref{lem:H-short-graded-module}, $K_{s,r}$ is a graded
$\mathfrak h$-module with trivial action, possibly nonzero.
With the labels interchanged by the dual cover,
$K_{s,r}$ is the direct-cover obstruction $\mathfrak o_{s,r}$.

For a short exact sequence of Holmes complexes, graded skew duality identifies
the dual of the connecting image with the intersection quotient in
\eqref{eq:H-Adamovich-K}.  This proves
\eqref{eq:H-Adamovich-image-intersection}.  The associated long exact
homology sequence gives \eqref{eq:H-Adamovich-kernel-exact}.  When
$\mathfrak o_{s,r}=0$, every $z\in R_r$ has a lift $w$ with
$\mathscr Dw\in\mathscr D(Z_s)$; subtracting a suitable $a\in Z_s$ then
gives a $\mathscr D$-closed lift $w-a$.  Conversely, right exactness of the
sequence of closed lifts makes the connecting morphism zero and hence
$\mathfrak o_{s,r}$.

\end{proof}

\begin{lemma}[Opposite-parity Spencer vanishing]
\label{lem:H-opposite-parity-Spencer}
Let \(r\ne s\) be nontrivial exceptional labels of opposite parity.  Every
positive-Spencer summand occurring in an extension from \(M_r\) to \(M_s\)
is zero.
\end{lemma}

\begin{proof}
Write \(j=T-q\) for the coefficient height associated with internal degree
\(q\), and let \(\mathcal P_{s,r}^{(j)}\) be the parabolic term before
passing from the Holmes parabolically induced modules to their simple heads, with the
target label $s$ written first and the source label $r$ second.  Let
$\varepsilon=-I\in\operatorname{Sp}(V)$.  It acts on $V_t$ by $(-1)^t$,
on the coefficient layer of height $j$ by $(-1)^{s+j}$, and on
$H_1^*\cong(S^3V)^*$ by $-1$.  Hence a first Spencer term has scalar
$(-1)^{s+j+1}$, whereas $V_r$ has scalar $(-1)^r$.  Since $r$ and $s$
have opposite parity, equivariance is possible only for even $j$; every
odd-height term is zero.  Write
\(R_s^\Delta=\operatorname{rad}\Delta(\omega_s)\).  It is enough to show
\[
 \mathcal P_{s,r}^{(j)}\ne0\quad\Longrightarrow\quad r<s
\tag{H.Oa}\label{eq:H-parabolic-one-sided}
\]
for every even \(j\).

First take \(j\ge10\).  A composition factor
\(\Delta(\omega_t)\) of the coefficient radical satisfies \(t\le s-2\).
The saturated type-\(C\) Pieri filtration adds at most one first-column box,
so a fundamental constituent on this part has
\[
 r\le t+1\le s-1.
\tag{H.Ob}\label{eq:H-radical-strict-lower}
\]
The remaining top-Weyl term can occur only at
\[
 j-pa\in\{-2,0,2\}.
\tag{H.Oc}\label{eq:H-tail-resonances}
\]
This follows by comparing the three cubic Pieri boxes with the Frobenius
part of a restricted fundamental target.  The three possible highest lines
are relation boundaries.  At \(j=pa\) the coefficient of the one-removal
relation is
\(72(-1)^{s-1}\).  At \(j=pa+2\), the Frobenius base is a \(p\)-th power,
so the transport terms vanish and the constant-coefficient three-box
relation is a unit.  At \(j=pa-2\) the coefficient of the pure-addition
relation is
\(6(j+7)\equiv30\pmod p\) for \(p\ge7\); when \(p=5\), the primitive
\(F(V)V\) Hall lift is evaluated as follows.  For a symplectic pair
$e,f$ and a pair-free primitive $\alpha$, its coefficient action is
\[
 e^bf\cdot(e^N\otimes\alpha)
 =-(N+2b)e^{N+b-1}\otimes\alpha.
\]
Here $N=j-3=5(a-1)$ and $b=4$; the two exterior placements cancel the
factor $1/2$ in the primitive Hall relation and leave
$\pm(N+8)=\pm(5a+3)\equiv\pm3\pmod5$.  Thus the top-Weyl contribution is zero,
and \eqref{eq:H-radical-strict-lower} proves
\eqref{eq:H-parabolic-one-sided} for $j\ge10$.

For the five low even heights, begin with \(j=0\).  The first
action on \(V\otimes\Delta(\omega_k)\) is the Cartan section
\(\Delta(\omega_1+\omega_k)\).  With the trace-normalized tensor Casimir,
the three minuscule sections have eigenvalues
\[
 1,\qquad-k,\qquad k-2m-2.
\tag{H.Od}\label{eq:H-j0-Casimir-values}
\]
The only collisions with the Cartan eigenvalue $1$ are
\(p\mid k+1\) and \(p\mid2m+3-k\).  The two non-Cartan eigenvalues also
coincide when $p\mid k-m-1$, but this does not enlarge the generalized
$1$-eigenspace unless both Cartan walls already meet.  On the integral
Howe primitive lattice, wedge and
contraction identify the single-wall kernel with the image of the adjacent
fundamental boundary; when both walls meet, the squarefree secondary
relation has unit coefficient.  Hence the collision quotient has no nontrivial
fundamental factor, and the term for \(j=0\) is zero.

At \(j=2\), put \(a_{s,r}=(s-r-1)/2\) and write
\(a\subset_p b\) when the nonzero base-\(p\) digits of \(a\) occur among
those of \(b\).  The two possible contributions have dimensions
\[
 \mathbf1_{\{r<s\}}
 \mathbf1_{\{a_{s,r}\subset_p m-r\}},
 \qquad
 \mathbf1_{\{r=1\}}
 \mathbf1_{\{0\triangleleft_{p,m}s\}}.
\tag{H.Oe}\label{eq:H-j2-readout}
\]
Both are strictly lower.  This is the fundamental Weyl composition-factor
formula together with the degree-two Holmes complex; the latter is exact on
every nontrivial coefficient and contributes \(S^2V\) only at the direct
trivial cover.

At \(j=4\) in characteristic five, the deletion homology is
\[
 V^{(1)}\otimes\mathcal H_0(R_s^\Delta),
\]
which is restricted-trivial; the corresponding de Rham complex has Cartier
quotient \(V^{(1)}\), and the two possible additional Frobenius
constituents are excluded by type-\(C\) dominance.  At \(j=6\), the deletion homology is
\(V^{(1)}\otimes\mathcal H_2(R_s^\Delta)\), which has the same root-lattice
parity as \(s\).  The only remaining coincident constituent has \(r=1\)
and \(s\) even, hence again \(r<s\).  At \(j=8\), the wedge Frobenius
constituent is
excluded by the horizontal-three-strip order; the unique parallel line is
killed by the same coefficient identity with $N=5$ and $b=4$, which gives
\[
 \pm13\equiv\pm3\pmod5.
\tag{H.Of}\label{eq:H-j8-unit}
\]
For \(p\ge7\) these three degrees occur before the first compatible
Frobenius deletion, and the ordinary Pieri--Koszul contraction gives
\eqref{eq:H-parabolic-one-sided}.

The pairing in \cref{lem:H-apolar-Holmes-pairing} has the same PBW complement for
every exceptional label.  Transposing a homogeneous Hamiltonian class
therefore preserves \(q\), and hence \(j\), while reversing the ordered
pair.  A nonzero class would give both
\(\mathcal P_{s,r}^{(j)}\ne0\) and
\(\mathcal P_{r,s}^{(j)}\ne0\).  By
\eqref{eq:H-parabolic-one-sided} this forces \(r<s\) and \(s<r\), a
contradiction.
\end{proof}

\begin{lemma}[Hamiltonian \(D\)-homology vanishing]
\label{lem:H-D-end-extinction}
For \(1\le r,s\le m\),
\[
 \Hom_{\mathfrak h}(\mathcal H_r^D,M_s)=0,\qquad
 \mathcal H_r^D=R_r/S_r.
\tag{H.D}\label{eq:H-D-end-extinction}
\]
\end{lemma}

\begin{proof}
Filter the divided-power form model by the number of unsaturated
coordinates.  The argument first treats maps into \(\mathcal H_r^D\); the
reverse orientation then follows from the self-duality in
\cref{lem:H-apolar-Holmes-pairing}.  The one-coordinate contraction
removes every leading class with at least two unsaturated coordinates, while
a fully saturated leading class is Cartier-trivial.  Thus a nonzero map
\(M_s\to\mathcal H_r^D\) can have only one unsaturated coordinate, and its
associated-graded leading term is supported on a natural type-\(C\) edge.

Put \(c_t=m+1-t\).  For \(1\le t<m\), the saturated primitive-form
calculation gives
\[
 \dim\Hom_{H_0}(V\otimes V_t,V_{t+1})
 =\dim\Hom_{H_0}(V\otimes V_{t+1},V_t)
 =\begin{cases}1,&p\nmid c_t,\\0,&p\mid c_t,
 \end{cases}
\tag{H.D1}\label{eq:H-D-natural-edge-dimension}
\]
and there is no farther fundamental endpoint.  Indeed the denominator-free
raising map
\[
 B_t(v\otimes\alpha)
 =c_t\,v\wedge\alpha-L(\iota_v\alpha),
 \qquad L=\omega\wedge-,
\tag{H.D2}\label{eq:H-D-natural-edge-map}
\]
satisfies
\[
 B_t(e_{t+1}\otimes v_t^+)=c_t v_{t+1}^+
\tag{H.D3}\label{eq:H-D-natural-edge-highest}
\]
for a primitive highest vector \(v_t^+\); the lowering orientation is its
transpose up to a Chevalley unit.  At \(p\mid c_t\) the
Adamovich--Baranov--Suprunenko wall submodule is exactly the kernel removed
on passage to the simple fundamental quotient, so both adjacent endpoint
spaces in \eqref{eq:H-D-natural-edge-dimension} are zero.

It remains to exclude factorization off the wall.  Choose Darboux coordinates
\[
 x_1,\ldots,x_m,x_{1'},\ldots,x_{m'},
\]
let \(I=\{1,\ldots,t\}\), \(I'=\{1',\ldots,t'\}\), and put \(a=t+1\).
The two normalized generalized-gradient lifts of the adjacent lines will be
denoted
\[
 \bar\delta_t:Z_{t+1}\longrightarrow R_t,
 \qquad
 \mathscr T_t:Z_t\longrightarrow R_{t+1}.
\]
Set
\[
 \kappa_I=
 \Bigl(\prod_{i\in I}x_i^{(p-1)}\Bigr)dx_I,
 \qquad
 X_{I,a}=-
 \Bigl(\prod_{i\in I}x_i^{(p-1)}\Bigr)x_{a'}
 \,dx_I\wedge dx_a,
\tag{H.D4}\label{eq:H-D-Cartier-up-witness}
\]
and choose the complementary one-coordinate monomial \(u_I\in Z_t\) for
the reverse edge, with \(\zeta_I:=\mathscr T_tu_I\).  Direct divided-power
differentiation gives
\[
 X_{I,a}\in R_{t+1},\qquad
 \bar\delta_tX_{I,a}=\pm\kappa_I,
 \qquad
 u_I\in R_t,\qquad
 \mathscr T_tu_I=\zeta_I.
\tag{H.D5}\label{eq:H-D-Cartier-witness-images}
\]
The two target classes are not \(\mathscr D\)-boundaries.  For the first,
the complementary top form
\[
 \kappa_I^\vee=
 \Bigl(\prod_{b\notin I}x_b^{(p-1)}\Bigr)dx_{I'}
\]
is \(\mathscr D\)-closed and satisfies
\(\langle\kappa_I,\kappa_I^\vee\rangle=\pm1\).  For the reverse edge the
pair-free complementary monomial \(\beta_{I,a}\) is \(\mathscr D\)-closed
and satisfies \(\langle\zeta_I,\beta_{I,a}\rangle=\pm1\).  Since
\(\mathscr D^*=\pm\mathscr D\),
\[
 [\kappa_I]\ne0\in\mathcal H_t^D,
 \qquad
 [\zeta_I]\ne0\in\mathcal H_{t+1}^D.
\tag{H.D6}\label{eq:H-D-Cartier-witness-nonzero}
\]

Now consider the adjacent orientation \(t+1\to t\).  Any simple endpoint
map is a scalar multiple of the line represented by \(\bar\delta_t\).  To
factor through the simple head
\(Z_{t+1}\twoheadrightarrow M_{t+1}=Z_{t+1}/R_{t+1}\), it would have to
annihilate \(R_{t+1}\); but
\eqref{eq:H-D-Cartier-witness-images}--\eqref{eq:H-D-Cartier-witness-nonzero}
show that \(X_{I,a}\in R_{t+1}\) has nonzero image in
\(\mathcal H_t^D\).  The reverse adjacent orientation is identical using
\(u_I\in R_t\) and \(\mathscr T_tu_I=\zeta_I\).  Hence no adjacent simple
embeds in \(\mathcal H_r^D\).  At a natural wall the endpoint space is
already zero; far labels are absent from the natural-edge support, and the
diagonal has no endpoint lift of the required degree.  Thus
\[
 \Hom_{\mathfrak h}(M_s,\mathcal H_r^D)=0
 \qquad(1\le r,s\le m).
\]
Finally the top-coefficient apolar pairing identifies
\((\mathcal H_r^D)^\vee\cong\mathcal H_r^D\) and
\(M_s^\vee\cong M_s\).  Dualizing gives
\(\Hom_{\mathfrak h}(\mathcal H_r^D,M_s)=0\), as required.
\end{proof}

\begin{lemma}[A radical sequence for the trivial source]
\label{lem:H-trivial-radical-row}
For $2\le s\le m$ put
\[
 \mathcal H_s^D:=R_s/S_s
 =\ker\mathscr D_s/\operatorname{im}\mathscr D_s,
 \qquad Q_s^H:=(\mathcal H_s^D)^H,
 \qquad \epsilon_s=\mathbf1_{\{0\triangleleft_{p,m}s\}}.
\tag{H.Ra}\label{eq:H-trivial-invariant-quotient}
\]
Then there is a short exact sequence
\[
 0\longrightarrow Q_s^H
 \longrightarrow\Ext^1_{u_0(\mathfrak h)}(\kk,M_s)
 \longrightarrow\kk^{\epsilon_s}\longrightarrow0.
\tag{H.Rb}\label{eq:H-trivial-radical-row}
\]
\end{lemma}

\begin{proof}
Under the form--cocycle identification, the parabolic inner-derivation
columns map to
$S_s=\operatorname{im}\mathscr D_s$.  Thus $\mathcal H_s^D$ is the quotient
of the Holmes homology modulo restricted inner derivations.
Because $S_s=\soc Z_s\cong M_s$ is nontrivial, $R_s^H=0$.  Applying
$\Hom_{u_0(\mathfrak h)}(\kk,-)$ to
$0\to S_s\to R_s\to\mathcal H_s^D\to0$ identifies the cokernel of
$Q_s^H\to\Ext^1(\kk,M_s)$ with
\[
 \ker\!\left[\Ext^1(\kk,R_s)\longrightarrow
 \Ext^1(\kk,\mathcal H_s^D)\right].
\]
Filter \(R_s\) by the induced fundamental coefficient filtration.
Restricted Shapiro identifies the first nonzero quotient of this
kernel with the type-\(C\) coefficient extension from the trivial module to
\(V_s\).  Adamovich's multiplicity-free fundamental lattice contains this
line precisely when \(0\triangleleft_{p,m}s\).  At that cover, induce the
nonsplit coefficient sequence
\[
 0\longrightarrow V_s\longrightarrow E_s^{\rm coef}
 \longrightarrow\kk\longrightarrow0.
\]
Its connecting class restricts nontrivially to the inducing top, and hence
gives a nonzero element of the kernel.  Every lower coefficient layer has a
nontrivial fundamental head.  The filtration by adjacent fundamental
summands and \cref{lem:H-D-end-extinction} make its map to
\(\mathcal H_s^D\) injective on the unique possible adjacent summand; all
other summands vanish by the fundamental tensor rule.  Thus no further
kernel occurs.  The class just constructed spans the kernel on a direct cover,
and the kernel is zero otherwise, proving
\eqref{eq:H-trivial-radical-row}.
\end{proof}

\begin{lemma}[Cartier decomposition]
\label{lem:H-Cartier-normal-form}
For $2\le s\le m$, the invariant quotient $Q_s^H$ of
\eqref{eq:H-trivial-invariant-quotient} has the canonical decomposition
\[
 Q_s^H\cong
 \mathcal C_s^{\uparrow}\oplus\mathcal C_s^{\downarrow}
 \oplus
 \mathbf1_{\{p\mid(m-s+2)\}}\mathcal C_s^{\rm sym}
 \oplus\epsilon_s\,\kk q_s^{\rm top}.
\tag{H.Ca}\label{eq:H-Cartier-normal-form}
\]
where
\[
 \dim\mathcal C_s^{\uparrow}=\binom{2m}{s},\qquad
 \dim\mathcal C_s^{\downarrow}=\binom{2m}{s-1},\qquad
 \dim\mathcal C_s^{\rm sym}=\binom{2m}{s-1}.
\tag{H.Cb}\label{eq:H-Cartier-normal-form-dims}
\]
\end{lemma}

\begin{proof}
Use the truncated symplectic de Rham model with exterior differential $d$,
its symplectic conjugate $d^-$, Lefschetz operator $L=\omega\wedge-$, and
divided powers $L^{[a]}$.  The K\"unneth--Cartier decomposition gives two
independent exterior summands after inner derivations are removed:
\[
 \mathcal C_s^{\uparrow}\cong H^s_{\rm dR}(A),\qquad
 \mathcal C_s^{\downarrow}\cong H^{s-1}_{\rm dR}(A),
\tag{H.Cc}\label{eq:H-Cartier-two-rows}
\]
which gives the first two dimensions in
\eqref{eq:H-Cartier-normal-form-dims}.  On the symplectic-contraction
summand the homotopy is $(m-s+2)^{-1}\iota_\omega$ off the wall.  Hence it is
acyclic when $p\nmid(m-s+2)$.  On the wall, Cartier identifies its kernel
modulo image with the symbol summand
\[
 \mathcal C_s^{\rm sym}\cong H^{s-1}_{\rm dR}(A),
\tag{H.Cd}\label{eq:H-Cartier-symbol-row}
\]
and this summand includes every occupancy type.

A mixed staircase cannot occur.  Indeed, choose a nonzero residual class
of minimal Holmes--Lefschetz filtration.  Let $u$ be its terminal
$d^-$-closed coefficient.  Repeated use of
\[
 d^-L^{[k]}=L^{[k]}d^- -L^{[k-1]}d
\tag{H.Ce}\label{eq:H-divided-Lefschetz-identity}
\]
puts the whole class, modulo the three preceding summands, into the
normal form
\[
 z_{a-k}=(-1)^kL^{[k]}u\quad(0\le k\le a),
 \qquad
 y=(-1)^aL^{[a]}du.
\tag{H.Cf}\label{eq:H-mixed-staircase}
\]
If $p\nmid a$, then
$L^{[a]}du=a^{-1}L(L^{[a-1]}du)$, so the terminal term lies in the
Lefschetz image and the staircase shortens, contradicting minimality.  If
$p\mid a$, this scalar homotopy is unavailable.  The saturated fundamental
Weyl filtration
\cite{PremetSuprunenko,AdamovichFundamental,BaranovSuprunenko} places the terminal
coefficient in a lower fundamental label.  Applying the same construction
to the apolar transpose preserves the Holmes degree and reverses the strict
label inequality.  A nonzero terminal class would therefore force both
\(b<s\) and \(s<b\).  Hence the mixed staircase is zero.

For the internal top line, use the exact
Adamovich--Baranov--Suprunenko cover relation introduced before
\cref{def:H-classical-cover-lines}; thus
$0\triangleleft_{p,m}s$ means that the trivial simple is an immediate lower
cover of $V_s$ in the multiplicity-free fundamental submodule poset.  This
immediate-cover condition refines the occurrence of a trivial composition
factor.

At a direct trivial cover,
\eqref{eq:H-restricted-trivial-fundamental-row} shows that the coefficient
extension remains nonsplit in the restricted Levi category.  Dualize it
and induce:
\[
 0\longrightarrow A[s]\longrightarrow W_s^{\rm top}
 \longrightarrow Z_s\longrightarrow0,
 \qquad D|_{A[s]}=0.
\tag{H.Cg}\label{eq:H-direct-trivial-induced-cover}
\]
Put $C_s^{\rm top}=D(W_s^{\rm top})$.  The quotient map gives
\[
 0\longrightarrow K_s^{\rm top}\longrightarrow C_s^{\rm top}
 \longrightarrow D(Z_s)=M_s\longrightarrow0,
 \qquad K_s^{\rm top}=C_s^{\rm top}\cap A[s].
\tag{H.Ch}\label{eq:H-direct-trivial-D-intersection}
\]
On the least polynomial layer, $D(fe)=\rho(J_f)e$ and the quadratic
Hamiltonians realize all of $\mathfrak{sp}(V)$, so
$(C_s^{\rm top})_{\min}=\mathfrak{sp}(V)E_s^\vee$.  If its intersection
with the kernel line $\kk$ were zero, projection would identify
$\mathfrak{sp}(V)E_s^\vee$ with $V_s$ and split the coefficient extension.
It follows that $\kk[s]\subseteq K_s^{\rm top}$.

The parabolically induced module from the trivial coefficient has the Loewy chain
$0<\kk<A^\circ<A$ and every nonzero submodule contains $\kk$.  Therefore a
submodule of $A[s]$ containing $\kk[s]$ is one of
$\kk[s],A^\circ[s],A[s]$.  Since $D$ has degree $-2$,
$C_s^{\rm top}$ is zero in the two greatest degrees, whereas
$A^\circ[s]$ and $A[s]$ meet those degrees.  Consequently
\[
 K_s^{\rm top}=\kk[s].
\tag{H.Ci}\label{eq:H-direct-trivial-exact-line}
\]
The induced surjection $\operatorname{rad}Z_s\twoheadrightarrow\kk[s]$
gives a nonzero invariant in $Q_s^H$.
\Cref{lem:H-apolar-Holmes-pairing} gives total complement $N_0+2s$;
reversing the degree of this degree-$(-2)$ quotient
places the invariant in Holmes degree $N_0+s-2$.  Thus
\[
 \epsilon_s:=\mathbf1_{\{0\triangleleft_{p,m}s\}},
 \qquad
 \dim\kk q_s^{\rm top}=\epsilon_s.
\tag{H.Cj}\label{eq:H-top-line-direct-cover}
\]
The cited fundamental submodule theorem expresses this condition in the
digit poset, including the case $p\mid s/2$.
Its Holmes degree differs from the three Cartier summands, so the sum in
\eqref{eq:H-Cartier-normal-form} is direct.  Together with the staircase
vanishing, this proves the decomposition.
\end{proof}

\begin{lemma}[Degree separation for Hamiltonian covers]
\label{lem:H-cover-degree-separation}
For \(1\le a\le m\), every trivial submodule of the Hamiltonian Holmes
homology with fundamental coefficient \(V_a\) is supported in
\[
 \mathscr E_a=
 \{ap,\ N_0-(a-1)(p-2),\ (a-1)p+2,\ N_0+a-2\}.
\tag{H.Ga}\label{eq:H-four-invariant-degrees}
\]
If \(u\triangleleft_{p,m}t\) is a nontrivial direct cover with \(u>0\),
then its Holmes connecting obstruction vanishes.
\end{lemma}

\begin{proof}
For $a=1$, the homogeneous derivation basis in
\cref{prop:H-trivial-adjoint} consists of the $2m$ Cartier classes in
degree $p$, the Euler class in degree $2$, and the full-top Hamiltonian
class in degree $N_0$.  Thus its trivial submodules are supported in
$\{p,2,N_0\}\subseteq\mathscr E_1$.  For $a\ge2$, the four degrees are
respectively the two Cartier faces, the symplectic-contraction wall, and
the internal top line of \cref{lem:H-Cartier-normal-form}.  If the
direct-cover obstruction were
nonzero, apolar duality would give
\[
 e_t+e_u=N_0+t+u-2,\qquad
 e_t\in\mathscr E_t,\quad e_u\in\mathscr E_u.
\tag{H.Gb}\label{eq:H-four-degree-cover-equation}
\]
Put \(d=t-u\ge2\).  Subtracting the right side gives
\[
\renewcommand{\arraystretch}{1.18}
\begin{array}{c|cc}
 &up&N_0-(u-1)(p-2)\\ \hline
tp&(p-1)(u+t-2m)+2&d(p-1)+p\\
N_0-(t-1)(p-2)&p-d(p-1)&(p-1)(2m-u-t+2)\\
(t-1)p+2&(p-1)(u+t-2m)-p+4&d(p-1)+2\\
N_0+t-2&u(p-1)&(p-1)(2m-u)+p-2
\end{array}
\]
and
\[
\renewcommand{\arraystretch}{1.18}
\begin{array}{c|cc}
 &(u-1)p+2&N_0+u-2\\ \hline
tp&(p-1)(u+t-2m)-p+4&t(p-1)\\
N_0-(t-1)(p-2)&2-d(p-1)&(p-1)(2m-t)+p-2\\
(t-1)p+2&(p-1)(u+t-2m)-2p+6&(p-1)t-p+2\\
N_0+t-2&(p-1)(u-1)+1&2(m(p-1)-1)
\end{array}.
\tag{H.Gc}\label{eq:H-four-degree-cover-table}
\]
Every entry is nonzero, using \(u+t\le2m-1\), \(d\ge2\),
\(t\le m\), and \(p\ge5\).  Hence the degree equation has no solution.
\end{proof}

\begin{theorem}[Exceptional Hamiltonian subquiver]
\label{thm:H-exceptional-quiver}
For \(1\le r,s\le m\),
\[
 \dim\Ext^1_{u_0(\mathfrak h)}(M_r,M_s)=
 \begin{cases}
 1,&r\ne s,\ r\equiv s\pmod2,\text{ and }
 (r\triangleleft_{p,m}s\text{ or }s\triangleleft_{p,m}r),\\
 0,&\text{otherwise}.
 \end{cases}
\tag{H.X}\label{eq:H-exceptional-quiver}
\]
\end{theorem}

\begin{proof}
For distinct same-parity labels, the parabolic five-term sequence has no
nilradical contribution.  Indeed \(-I\in\operatorname{Sp}(V)\) acts on
\(V_t\) by \((-1)^t\), while the first positive Hamiltonian layer
\(H_1\cong L(3\omega_1)\) has odd class in the type-\(C\) quotient
\(X/Q\cong\mathbf Z/2\).  Hence every composition factor of
\(H^1(N,M_s)\) has parity \(s+1\), and therefore
\[
 \Hom_{H_0}\!\left(V_r,H^1(N,M_s)\right)\ne0
 \quad\Longrightarrow\quad r\equiv s+1\pmod2.
\tag{H.X0}\label{eq:H-exceptional-parity-split}
\]
When \(r\equiv s\pmod2\) this term is zero, so restricted Shapiro reduces
the group with parabolically induced source exactly to the type-\(C\) Levi extension.  The
fundamental extension theorem makes this space one-dimensional precisely on
an Adamovich direct cover.  Its restriction obstruction is
\(\mathfrak o_{s,r}\) of \cref{lem:H-Adamovich-obstruction}, a trivial
graded module.  The four possible trivial degrees are those of
\eqref{eq:H-four-invariant-degrees}; substituting them in the cover pairing
gives the nonzero entries of
\eqref{eq:H-four-degree-cover-table}.  Hence
\(\mathfrak o_{s,r}=0\), and the Levi cover line descends to
\(\Ext^1(M_r,M_s)\).

For distinct opposite-parity labels, the Levi term is absent and
\cref{lem:H-opposite-parity-Spencer} shows that the positive term vanishes.  On the
diagonal, a map \(R_r\to M_r\) which vanishes on
\(S_r\) factors through \(\mathcal H_r^D\) and is zero by
\cref{lem:H-D-end-extinction}.  If its restriction to \(S_r\cong M_r\)
were nonzero, it would split the inclusion of the essential socle
\(S_r\subset R_r\), which is impossible.  Hence
\(\Hom(R_r,M_r)=0\).  Shapiro reduces
\(\Ext^1(Z_r,M_r)\) to the type-\(C\) simple self-extension and a
positive term.  The first vanishes by the Frobenius-kernel
no-self-extension theorem and the second by central parity.  The long exact
sequence for \(0\to R_r\to Z_r\to M_r\to0\) then gives
\(\Ext^1(M_r,M_r)=0\).
\end{proof}

\begin{theorem}[Trivial--exceptional Hamiltonian extensions]
\label{thm:H-trivial-exceptional}
For $2\le s\le m$, put
\[
 \epsilon_s=\mathbf1_{\{0\triangleleft_{p,m}s\}}
 =\mathbf1_{\{2\mid s\}}
  \mathbf1_{\{p\mid(s/2)(m-s/2+1)\}}.
\]
Then
\[
\dim\Ext^1_{u_0(\mathfrak h)}(\kk,M_s)
=\binom{2m+1}{s}
 +\mathbf1_{\{p\mid(m-s+2)\}}\binom{2m}{s-1}
 +2\epsilon_s.
\tag{H.Ta}\label{eq:H-trivial-exceptional}
\]
The reverse extension has the same dimension.  For $s=1$ both orientations
have dimension $2m+2$ by \cref{prop:H-trivial-adjoint}.
\end{theorem}

\begin{proof}
The second expression for \(\epsilon_s\) is the direct-cover criterion in
the Adamovich--Baranov--Suprunenko fundamental submodule lattice
\cite{PremetSuprunenko,AdamovichFundamental,BaranovSuprunenko}.
By \cref{lem:H-Cartier-normal-form},
\[
 \dim Q_s^H
 =\binom{2m}{s}+\binom{2m}{s-1}
 +\mathbf1_{\{p\mid(m-s+2)\}}\binom{2m}{s-1}
 +\epsilon_s.
\]
The first two terms have sum
$\binom{2m+1}{s}$.  The four summands occur in distinct Holmes degrees
\[
 sp,\qquad N_0-(s-1)(p-2),\qquad (s-1)p+2,\qquad N_0+s-2,
\tag{H.Tb}\label{eq:H-trivial-Holmes-degrees}
\]
so the four blocks remain independent after quotienting inner derivations.

The short exact sequence
\eqref{eq:H-trivial-radical-row} gives
\[
 \dim\Ext^1_{u_0(\mathfrak h)}(\kk,M_s)
 =\dim Q_s^H+\epsilon_s.
\]
The quotient generator is the induced coefficient-extension class and has
nonzero restriction to the inducing top.  The line
\(\kk q_s^{\rm top}\subset Q_s^H\) is an internal Holmes-homology class;
hence the two direct-cover contributions are distinct.
This proves \eqref{eq:H-trivial-exceptional}; self-duality of $M_s$ and
rigid duality give the reverse orientation.
\end{proof}

\begin{proposition}[Hamiltonian extensions with parabolically induced source]
\label{thm:H-zero-pair}
Let $L,L'$ be simple $u_0(\mathfrak h)$-modules.

If $L=Z(U)$ is parabolically induced, then
\[
\Ext^1_{u_0(\mathfrak h)}(Z(U),L')
\cong
\bigoplus_{q=0}^{2} H^1\!\left(\mathcal C_q^\bullet(U,L')\right).
\tag{4.4h}\label{eq:H-zero-parabolic-ext}
\]
The summand \(q=0\) is the restricted type-\(C\) Levi complex.  The
summand \(q=1\) is
\(\Hom_{H_0}(U,\mathsf P_1(L'))\), and the summand \(q=2\) is
\(\Hom_{H_0}(\mathsf P_2(L'),U^\vee)\), with
\(\mathsf P_1,\mathsf P_2\) as in
\eqref{eq:H-cubic-receiver}--\eqref{eq:H-quadratic-receiver}.  In
particular, every second-jet, affine-wall, PBW-boundary and non-pure
coefficient block is reduced to these two finite type-\(C\) maps.
The contragredient of a parabolically induced simple is again parabolically induced.  Hence every
mixed exceptional--induced group is obtained from
\eqref{eq:H-zero-parabolic-ext} by
\[
 \Ext^1(L,L')\cong\Ext^1((L')^*,L^*).
\tag{4.4i}\label{eq:H-zero-duality}
\]
The exceptional and trivial-module groups are given in
\cref{thm:H-exceptional-quiver,thm:H-trivial-exceptional}.
\end{proposition}

\begin{proof}
For a parabolically induced source, restricted PBW makes $u_0(\mathfrak h)$ free over
$u_0(P)$, so \cref{prop:Shapiro} gives
\[
\Ext^1_{u_0(\mathfrak h)}(Z(U),L')
\cong
\Ext^1_{u_0(P)}(U,\operatorname{Res}_P L').
\]
Applying \cref{prop:root-p} to the coefficient module
\eqref{eq:H-parabolic-coeff} gives \eqref{eq:H-zero-parabolic-ext}.
The Hamiltonian bracket, restricted \(p\)-map, module action and inner map
preserve the internal degree \(q\).  Hence the parabolic complex is the
direct sum of the finite complexes
\(\mathcal C_q^\bullet(U,L')\).  The Cartier contraction of
\cref{thm:H-cap-Koszul} places the nilradical cohomology in
\(\mathscr E_H\) for \(q\ge3\).  Since the inducing module \(U\) is
nonexceptional,
\eqref{eq:H-positive-standard-vanishing} annihilates every such summand,
proving \eqref{eq:H-zero-parabolic-ext}.  The Cartier filtration is
strict at affine and truncated-PBW boundaries, including non-pure
coefficient resolutions.

The PBW apolar pairing identifies the contragredient of \(Z(U)\) with the
parabolically induced simple from the dual Levi module with the parabolic
modular-character twist.  This is again a nonexceptional label.  Rigid duality
gives \eqref{eq:H-zero-duality}; if exactly one
module is exceptional, it places the parabolically induced module in the source.
The extensions among exceptional modules and those involving the trivial
module were determined in
\cref{thm:H-exceptional-quiver,thm:H-trivial-exceptional}.
\end{proof}

\section{Contact algebras}
\label{sec:contact}

Write $\widehat K=\Cont(2m+1;\mathbf1)$ for the full restricted Contact
algebra and $K=\widehat K^{(1)}$ for the derived simple algebra used in the
rest of this section.

\begin{remark}[The derived Contact algebra]
\label{rem:contact-derived-branch}
According to whether or not $p$ divides $m+2$, the highest degree differs.
In the divided-power model the only formal top monomial not
generated by the lower positive layers is
\[
 \Omega=z^{p-1}\prod_{i=1}^m q_i^{p-1}p_i^{p-1},
\]
and \cref{lem:contact-K1-generation} computes
\[
 \{zq_i,z^{p-1}h\}_K=-2(m+2)\Omega.
\]
The top line is therefore generated when $p\nmid(m+2)$; when $p\mid(m+2)$,
the derived simple algebra is the codimension-one subalgebra obtained by
deleting $\Omega$.  Henceforth $K$ denotes this derived algebra.  Its simple
modules are parametrized in
\cref{thm:contact-vertex-parametrization}, and their stable form
realizations are given in \cref{thm:contact-stable-form-models}.
\end{remark}

Let
\[
P_K=K_{\ge0}=K_0\ltimes K_{>0}
\]
be the Contact grading parabolic, with classical Levi
$K_0\cong\mathfrak{sp}_{2m}\oplus\kk z$; see
\cite{HolmesContact,HolmesZhang}.  Put
\[
N=K_{>0},\qquad N_j=\bigoplus_{d\ge j}K_d,\qquad
\mathcal K_j=N_j/N_{j+1}
\]
for the lower restricted filtration.  Fix a divided-power
monomial basis of $N$, the Chevalley--torus basis of $K_0$, and the PBW
bases of all parabolically induced modules.  Brackets, restricted powers, and module
actions in these bases are obtained by divided-power differentiation and
PBW commutation.

For an ordered pair of nonexceptional simple labels $(\lambda,\mu)$ and a
total contact degree $q$, let $E_{q-j}$ be the corresponding
coefficient slice and put
\[
C_q^1=\bigoplus_{j\ge1}\Hom_\kk(\mathcal K_j,E_{q-j}).
\]
The ordinary bracket rows, mixed Levi rows, and restricted
$p$-rows define a finite matrix
\[
\mathscr D_q:C_q^1\longrightarrow\mathcal C_q^{(2,p)},
\qquad
Z_q:=\ker\mathscr D_q.
\tag{5.1}
\]
Let $M_q=E_q$ be the total-degree-$q$ coefficient slice.  The inner
cocycle map is
\[
\delta_q(m)(n)=n\cdot m
\qquad(n\in N),
\]
with homogeneous components in the summands of $C_q^1$.  Its kernel is the
$N$-invariant subspace $M_q^N$.  Put
\[
B_q:=\operatorname{im}\delta_q\subseteq Z_q,
\qquad H_q:=Z_q/B_q.
\tag{5.2}
\]
Then there are $K_0$-equivariant short exact sequences
\[
0\longrightarrow M_q^N\longrightarrow M_q
\stackrel{\delta_q}{\longrightarrow}B_q\longrightarrow0,
\qquad
0\longrightarrow B_q\longrightarrow Z_q\longrightarrow H_q\longrightarrow0.
\tag{5.2a}
\]
Indeed $\delta_q$ is $K_0$-equivariant: for $x\in K_0$ and $n\in N$,
\[
(x\cdot\delta_q(m))(n)
=x\cdot(n\cdot m)-\delta_q(m)([x,n])
=n\cdot(x\cdot m)=\delta_q(x\cdot m)(n).
\]
A row-echelon basis for the column space of $\delta_q$ gives matrices for
$B_q$; the $K_0$-action is the restriction of its action on $Z_q$ to this
invariant image.
Write the two $K_0\cong\mathfrak{sp}_{2m}\oplus\kk z$ labels as
$\lambda=(\bar\lambda,\zeta_\lambda)$ and
$\mu=(\bar\mu,\zeta_\mu)$, and let $[\bar\lambda]\in\mathbb Z/2$
denote the type-$C_m$ root-lattice class of the highest weight.  If
$M_d$ is the PBW depth-$d$ piece of the induced target, then $z$ acts on
$M_d$ by $\zeta_\mu-d$, while $K_j$ has $z$-weight $j$.  Consequently
the total-degree-$q$ cocycle block has $z$-character
$\zeta_\mu-q$ and type-$C_m$ root-lattice class
$[\bar\mu]+q$.  Hence a simple source $U_\lambda$ can occur only if
\[
q\equiv\zeta_\mu-\zeta_\lambda\pmod p,
\qquad
q\equiv[\bar\lambda]-[\bar\mu]\pmod2.
\tag{5.2b}\label{eq:contact-degree-parity}
\]
Since $p$ is odd, these congruences determine a unique
$q_0(\lambda,\mu)\in\{0,\ldots,2p-1\}$ modulo $2p$.  Thus
\[
\mathcal Q_{\lambda\mu}
=\{q_0(\lambda,\mu)+2p\ell:
1\le q_0(\lambda,\mu)+2p\ell\le D_-+1\},
\qquad D_-=2(m+1)(p-1).
\tag{5.2c}
\]

\begin{lemma}[Positive Contact generation by the first layer]
\label{lem:contact-K1-generation}
Let $K=K(2m+1;\mathbf1)^{(1)}$ be the derived simple restricted Contact
algebra in characteristic $p\ge5$.  Then
\[
 K_j=[K_1,K_{j-1}]\qquad(j\ge2).
\tag{5.2d}\label{eq:contact-K1-generation}
\]
Here $K_1\cong S^3V\oplus V$.
\end{lemma}

\begin{proof}
Use the truncated contact coordinates
$q_1,p_1,\ldots,q_m,p_m,z$ and write $\{\ ,\ \}_K$ for the contact bracket
and $\{\ ,\ \}_{\rm sp}$ for its symplectic part.  First factor from a
symplectic-coordinate monomial every fully saturated pair
$q_i^{p-1}p_i^{p-1}$.  If an unsaturated exponent is at least two, the
identities
\[
 \{q^3,q^{a-2}p^{b+1}R\}_{\rm sp}=3(b+1)q^ap^bR,
 \qquad
 \{p^3,q^{a+1}p^{b-2}R\}_{\rm sp}=-3(a+1)q^ap^bR
\]
reduce it to the cubic Poisson image: use the first identity when
$b\le p-2$ and the second when $a\le p-2$.  Since the pair is unsaturated,
one of these choices has unit coefficient for $p\ge5$.  If the remaining
squarefree factor contains a complete pair or
contains variables from two distinct pairs, the identities
\[
 \{q^2p,M/q\}_{\rm sp}=2M,
 \qquad
 \{q_1q_2p_1,M/q_2\}_{\rm sp}=-M
\]
do the same.  Thus the only possible saturation shapes outside the cubic
Poisson image are
\[
 T_S=\prod_{i\in S}q_i^{p-1}p_i^{p-1},
 \qquad xT_S,
\]
where $x$ belongs to a pair outside $S$.

Induct on the exponent $c$ of $z$ in a target monomial $z^cM\in K_j$.
Every source below has contact degree one less and hence lies in $K_{j-1}$.
At $c=0$, for $T_S=T_iR$ put $U_i=q_i^{p-2}p_i^{p-2}$.  Then
\[
 \{q_i^3,zq_i^{p-4}p_i^{p-1}R\}_K=-T_S-3zU_iR,
\qquad
 \{p_i^3,zq_i^{p-1}p_i^{p-4}R\}_K=-T_S+3zU_iR.
\]
Adding gives $-2T_S$; multiplying by a variable from a different pair gives
the same conclusion for $xT_S$.  If
$M=\{f,h\}_{\rm sp}$ with cubic $f$, then
\[
 \{f,z^ch\}_K=z^cM-cz^{c-1}fh,
\]
and the second term is already generated by induction.  For the remaining
saturation shapes one has
\[
 \{zq_j,z^{c-1}q_jp_jT_S\}_K
 =z^cq_jT_S+
 \bigl(c+2|S|(p-1)-1\bigr)z^{c-1}q_j^2p_jT_S,
\]
\[
 \{zq_j,z^{c-1}p_jT_S\}_K
 =z^cT_S+
 \bigl(c+2|S|(p-1)-2\bigr)z^{c-1}q_jp_jT_S,
\]
whose error terms again have smaller $z$-exponent.

The fully saturated block is
$T_{\rm all}=\prod_iq_i^{p-1}p_i^{p-1}$.  For $0\le c\le p-2$ the two
brackets
\[
 \{q_i^3,z^{c+1}q_i^{p-4}p_i^{p-1}R\}_K
 =-(c+1)z^cT_{\rm all}-3z^{c+1}U_iR,
\]
\[
 \{p_i^3,z^{c+1}q_i^{p-1}p_i^{p-4}R\}_K
 =-(c+1)z^cT_{\rm all}+3z^{c+1}U_iR
\]
add to $-2(c+1)z^cT_{\rm all}$, a unit multiple of the target.  The only
formal top monomial left is
\[
 \Omega=z^{p-1}T_{\rm all}.
\]
Put $n=2m+1$ and
\[
 h=q_i^{p-2}p_i^{p-1}\prod_{r\ne i}q_r^{p-1}p_r^{p-1}.
\]
Then $z^p=0$ gives
\[
 \{zq_i,z^{p-1}h\}_K=-(n+3)\Omega.
\]
If $n+3\not\equiv0\pmod p$, this generates $\Omega$; if
$n+3\equiv0\pmod p$, the derived simple Contact algebra is the
codimension-one algebra in which this top divided-power monomial is deleted,
so $\Omega\notin K$.  In the latter case the preceding
brackets generate every basis monomial of the affected layer of $K$; the
single omitted monomial is precisely $\Omega$.  Thus
$K_j=[K_1,K_{j-1}]$ holds in both cases.  Finally, for the
remaining $V$-degree-zero targets,
\[
 \{zq,z^{c-1}p\}_K=z^c+(c-2)z^{c-1}qp,
\]
and induction on the $z$-exponent finishes the proof.
\end{proof}

\begin{lemma}[Support bound for positive Contact cohomology]\label{lem:contact-positive-cutoff}
Let $N=K_{>0}$, with $K_j=[K_1,K_{j-1}]$ for $j\ge2$, and let
$M=\bigoplus_{0\le d\le D_-}M_d$ be a graded coefficient module.  Then
\[
 H^1_{*,q}(N,M)=0\qquad(q>D_-+1).
\]
\end{lemma}

\begin{proof}
A homogeneous cocycle $\psi$ of internal degree $q$ satisfies
$\psi(K_j)\subseteq M_{q-j}$.  If $q>D_-+1$, then
$\psi(K_1)\subseteq M_{q-1}=0$.  Assuming inductively that
$\psi(K_{j-1})=0$, the cocycle identity gives
\[
 \psi([x,y])=x\psi(y)-y\psi(x)=0
 \qquad(x\in K_1,\ y\in K_{j-1}).
\]
Since $K_j=[K_1,K_{j-1}]$, induction gives $\psi(N)=0$.  Every restricted
cocycle is an ordinary cocycle, so the restricted cohomology group also
vanishes.
\end{proof}

\subsection{Nonexceptional mapping cones}

Write the restricted root-$p$ complexes of $B_q$ and $Z_q$ as
\[
C_{B,q}^0\xrightarrow{d_{B,q}^0}C_{B,q}^1
\xrightarrow{d_{B,q}^1}C_{B,q}^2,
\qquad
C_{Z,q}^0\xrightarrow{d_{Z,q}^0}C_{Z,q}^1
\xrightarrow{d_{Z,q}^1}C_{Z,q}^2.
\]
The inclusion $j_q:B_q\hookrightarrow Z_q$ induces a cochain map
$j_q^\bullet:C_{B,q}^\bullet\to C_{Z,q}^\bullet$.  Define
\[
\Gamma_q:C_{B,q}^1\oplus C_{Z,q}^0
\longrightarrow C_{B,q}^2\oplus C_{Z,q}^1
\]
by
\[
\Gamma_q(c,u)=
\bigl(d_{B,q}^1c,\ j_q^1c-d_{Z,q}^0u\bigr),
\tag{5.3}
\]
and let $E_q(c,u)=c$.  Then $E_q(\ker\Gamma_q)$ is the space of
$B_q$-cocycles whose image in the $Z_q$ complex is a boundary.  Hence
\[
\dim\ker\!\left[
\Ext_{K_0}^1(U_\lambda,B_q)\longrightarrow
\Ext_{K_0}^1(U_\lambda,Z_q)
\right]
=\operatorname{rank}\!\binom{\Gamma_q}{E_q}
-\operatorname{rank}\Gamma_q
-\operatorname{rank}d_{B,q}^0.
\tag{5.4}
\]
From $0\to B_q\to Z_q\to H_q\to0$ one also has
\[
a_q=\dim\coker\!\left[
\Hom_{K_0}(U_\lambda,B_q)
\longrightarrow
\Hom_{K_0}(U_\lambda,Z_q)
\right].
\tag{5.4a}
\]
The long exact sequence therefore gives
\[
\dim\Hom_{K_0}(U_\lambda,H_q)
=a_q+
\dim\ker\!\left[
\Ext_{K_0}^1(U_\lambda,B_q)\to
\Ext_{K_0}^1(U_\lambda,Z_q)
\right].
\tag{5.4b}
\]
Equations (5.4)--(5.4b) compute the homogeneous $K_0$-coefficient space
\[
\Hom_{K_0}(U_\lambda,H_q),
\qquad H_q=H_{*,q}^1(N,L(\mu)).
\tag{5.4c}
\]

\subsection{The restricted Hochschild--Serre sequence}

Let
\[
M=\Res_{P_K}^{K}L(\mu)=\bigoplus_{d=0}^{D_-}M_d,
\qquad M^N=M_0=U_\mu,
\tag{5.4d}
\]
and inflate $U_\lambda$ from $K_0$ to $P_K$ by letting $N$ act trivially.
Set
\[
E_{\lambda\mu}=U_\lambda^\vee\otimes M.
\]
Then
\[
E_{\lambda\mu}^N=U_\lambda^\vee\otimes U_\mu,
\qquad
H_*^1(N,E_{\lambda\mu})^{K_0}
\cong
\Hom_{K_0}\!\left(U_\lambda,H_*^1(N,M)\right).
\tag{5.4e}
\]
The restricted Hochschild--Serre spectral sequence for
$P_K=K_0\ltimes N$ gives the five-term sequence
\[
\begin{split}
0\longrightarrow&
H_*^1(K_0,E_{\lambda\mu}^N)
\longrightarrow H_*^1(P_K,E_{\lambda\mu})
\longrightarrow H_*^1(N,E_{\lambda\mu})^{K_0}\\
&\xrightarrow{\ d_2\ }
H_*^2(K_0,E_{\lambda\mu}^N).
\end{split}
\tag{5.4f}
\]
Equivalently,
\[
\begin{split}
0\longrightarrow{}&
\Ext_{u_0(K_0)}^1(U_\lambda,U_\mu)
\longrightarrow
\Ext_{u_0(P_K)}^1(U_\lambda,M)\\
\longrightarrow{}&
\Hom_{K_0}(U_\lambda,H_*^1(N,M))
\xrightarrow{d_2}
\Ext_{u_0(K_0)}^2(U_\lambda,U_\mu).
\end{split}
\tag{5.4g}
\]

\begin{proposition}[Vanishing of the Contact transgression]
\label{prop:contact-d2}
For the graded nonexceptional Contact coefficient module in
\textup{(5.4d)}, the transgression in \textup{(5.4f)}--\textup{(5.4g)} is zero.
Each positive total contact degree $q$ is preserved by the
restricted Hochschild--Serre bicomplex, whereas
$H_*^2(K_0,E_{\lambda\mu}^N)$ is concentrated in total degree zero.
It follows that
\[
d_2|_{H_{*,q}^1(N,E_{\lambda\mu})^{K_0}}=0
\qquad(q\ge1).
\tag{5.4h}
\]
There is no degree-zero nilradical one-cochain block, so $d_2=0$ on the
whole term $H_*^1(N,E_{\lambda\mu})^{K_0}$.
\end{proposition}

\begin{proof}
Put the contact degree on the positive restricted subalgebra
\[
N=\bigoplus_{j>0}K_j,
\qquad \deg K_j=j,
\]
and degree zero on $K_0$.  For homogeneous $x\in K_j$, the restricted
power is the $p$th iterate of $x$ as a derivation, hence
\[
x^{[p]}\in K_{pj},
\tag{5.4h1}\label{eq:contact-p-homogeneous}
\]
with the right side zero beyond the finite contact grading.  Therefore the
defining relations
\[
xy-yx=[x,y],\qquad x^p=x^{[p]}
\]
of $u_0(N)$ are homogeneous.  Restricted PBW consequently gives a finite
positive grading on $u_0(N)$, and
\[
u_0(P_K)=u_0(N)\#u_0(K_0)
\]
is a graded smash product with $u_0(K_0)$ in degree zero.

Take the normalized graded bar resolution of this smash product and filter it
by the number of positive-algebra inputs.  This is the restricted
Hochschild--Serre double complex giving \textup{(5.4f)}.  If a bar cochain
with positive inputs of degrees $j_1,\ldots,j_b$ takes values in the depth
component $M_d$, assign it internal contact degree
\[
q=d+j_1+\cdots+j_b.
\tag{5.4h2}\label{eq:contact-internal-degree}
\]
The bar differential has degree zero: multiplication in $u_0(N)$ adds contact
degrees, the $K_0$-action has degree zero, and the coefficient action
satisfies $N_jM_d\subseteq M_{d-j}$.  Equation
\eqref{eq:contact-p-homogeneous} shows that the restricted $p$-relations are
homogeneous in the same grading.  Hence every differential of the resulting
spectral sequence, in particular $d_2$, preserves $q$.

By \textup{(5.4d)},
$E_{\lambda\mu}^N=U_\lambda^\vee\otimes U_\mu$ is the depth-zero coefficient
module.  Thus
$E_2^{2,0}=H_*^2(K_0,E_{\lambda\mu}^N)$ is concentrated in internal degree
zero.  A degree-zero element of $E_2^{0,1}$ would be represented by a
nilradical one-cochain sending $N_j$ to $M_{-j}$ for every $j>0$, and these
components vanish.  Therefore $E_2^{0,1}$ has no internal degree-zero part.
The degree-preserving differential
\[
d_2:E_2^{0,1}\longrightarrow E_2^{2,0}
\]
is consequently zero.  This proves \textup{(5.4h)} and the proposition.
\end{proof}

The Levi term in \textup{(5.4g)} is computed by applying
\cref{prop:root-p} to
\[
E_0=U_\lambda^\vee\otimes U_\mu.
\]
If $\mathcal R_{K_0,E_0}$ is the restricted root-$p$ matrix above, then
\[
\dim\Ext_{u_0(K_0)}^1(U_\lambda,U_\mu)
=
\operatorname{nullity}\mathcal R_{K_0,E_0}
-\dim E_0+\dim E_0^{K_0},
\tag{5.4i}
\]
Combining \cref{prop:contact-d2} with the internal grading gives the
canonical graded decomposition
\[
\Ext_{u_0(P_K)}^1(U_\lambda,M)
\cong
\Ext_{u_0(K_0)}^1(U_\lambda,U_\mu)
\oplus
\bigoplus_{q\in\mathcal Q_{\lambda\mu}}
\Hom_{K_0}(U_\lambda,H_q).
\tag{5.4j}
\]
The summands in \textup{(5.4j)} are computed respectively by
\textup{(5.4i)} and by the matrices \textup{(5.3)}--\textup{(5.4c)}.

\subsection{Exceptional and folded form modules}

The simple-label set of the Levi algebra $u_0(K_0)$ is
\[
\Lambda_{K_0}^{\rm Levi}=X_1^+(C_m)\times\F_p.
\tag{5.5}
\]
For $\alpha=(\lambda,\zeta)$ put
\[
S_\alpha=L_0(\lambda)\boxtimes\kk_\zeta.
\]
The matrices of $L_0(\lambda)$ are obtained from the integral
Weyl-to-dual-Weyl map after reduction to $\kk$, and $z$ acts by the scalar
$\zeta$.  For the derived Contact algebra $K$ set
\[
 \Lambda_K^{\rm Hol}(m,p):=\Lambda_{K_0}^{\rm Levi}.
\]
Holmes' theorem uses every Levi label for each of the two algebras in
\cref{rem:contact-derived-branch}.

\begin{theorem}[Simple Contact modules]\label{thm:contact-vertex-parametrization}
For every $\alpha\in\Lambda_{K_0}^{\rm Levi}$ the parabolically induced module
\[
Z_K(\alpha)=u_0(K)\otimes_{u_0(P_K)}S_\alpha,
\qquad P_K=K_0\ltimes K_{>0},
\]
has a unique simple quotient, denoted $L_K(\alpha)$.  In each of the two
cases the resulting map is the Holmes bijection
\[
\Lambda_K^{\rm Hol}(m,p)\xrightarrow{\ \sim\ }
\operatorname{Irr}u_0(K),
\qquad \alpha\longmapsto[L_K(\alpha)].
\tag{5.5v}\label{eq:contact-Holmes-bijection}
\]
Call $\alpha$ nonexceptional when $Z_K(\alpha)$ is simple, and exceptional
otherwise.  For a nonexceptional label $\alpha$ one has
$Z_K(\alpha)=L_K(\alpha)$.
\end{theorem}

\begin{proof}
This is the Holmes--Nakano graded-head correspondence in the restricted
setting \cite{HolmesNakanoBrauer,HolmesContactSimple}.  The positive ideal
acts trivially on $S_\alpha$, and restricted PBW gives $Z_K(\alpha)$.
Holmes proves that $Z_K(\alpha)$ has a unique simple quotient and that these
quotients exhaust the simple restricted Contact modules, treating the
generated-top and deleted-top cases separately.  For an exceptional label,
\cite[\S3--5]{HolmesContact} realizes $Z_K(\alpha)$ inside an exterior
algebra of contact differential forms and identifies the maximal submodule
as the kernel of a contraction map on an integral divided-power basis,
as in the parallel Witt--Jacobson, Special, and Hamiltonian constructions
of \cite{ShenII,ShenIII,HolmesHamiltonian}.  The PBW matrices of
$Z_K(\alpha)$ restrict to this submodule and induce the matrices on the
quotient $L_K(\alpha)$.  For a nonexceptional label, the parabolically induced module is
simple.
\end{proof}

The exceptional contact-form modules have labels
\[
\{\omega_k^-:0\le k\le m\}
\quad\text{and}\quad
\{\omega_k^+:1\le k\le m+1\}.
\tag{5.5a}
\]
In the label set
$\Lambda_{K_0}^{\rm Levi}=X_1^+(C_m)\times\F_p$ these mean
\[
 \omega_k^-=(\omega_k,k-2m-2),
 \qquad
 \omega_k^+=(\omega_{k-1},1-k),
\tag{5.5a$'$}\label{eq:contact-exceptional-labels}
\]
with $\omega_0=0$ and $\omega_i$ the $i$th fundamental type-$C_m$
weight.
When the two labels coincide modulo $p$, their form realizations define the
same vertex.

Write
\[
I^k:=\theta\wedge\Omega^{k-1},
\qquad
Z_{\rm dR}^k:=\ker(d:\Omega^k\to\Omega^{k+1}).
\]
Both are $K$-submodules.  Indeed, for every contact vector field $X$ one has
$\mathcal L_X\theta=f_X\theta$, while Cartan's identity gives
$[d,\mathcal L_X]=0$.  Hence the quotient
\[
\mathscr E_k^-:=\Omega^k/(I^k+Z_{\rm dR}^k)
\tag{5.5b}
\]
is a restricted $K$-module for every $p\ge5$.

Use the horizontal splitting
\[
\Omega^k=\Omega_H^k\oplus\theta\wedge\Omega_H^{k-1},
\qquad
L=d\theta\wedge-,
\]
and write
\[
d(\alpha+\theta\wedge\beta)
=(d_H\alpha+L\beta)
+\theta\wedge(T\alpha-d_H\beta).
\tag{5.5c}
\]
The identities
\[
d_HL=Ld_H,
\qquad
d_H^2=-LT
\tag{5.5d}
\]
give a differential on the quotient below.  Put
\[
Q^k:=\Omega_H^k/L\Omega_H^{k-2},
\qquad
Z_Q^k:=\ker(d_Q:Q^k\to Q^{k+1}),
\qquad
B_Q^k:=\operatorname{im}(d_Q:Q^{k-1}\to Q^k).
\tag{5.5e}
\]

For \(0\le r\le m\), let
\[
 \operatorname{Cov}_{m,p}(r)
 :=\{a:\ a\triangleleft_{p,m}r\}
\]
be the set of direct lower covers of the head of the fundamental Weyl
module \(\Delta(\omega_r)\).  Put \(n_r=m-r\), and let
\(\nabla_{A_1}(n_r)\) be the dual Weyl Howe multiplicity module.  Write
\(\Delta_{A_1}(n_r)=\nabla_{A_1}(n_r)^\circ\) for its contragredient
Ringel dual, with divided-power basis
\[
 u_0,u_1,\ldots,u_{n_r},\qquad
 \operatorname{wt}u_j=n_r-2j,\qquad
 e u_j=(n_r-j+1)u_{j-1}.
\]
For \(a\in\operatorname{Cov}_{m,p}(r)\), set
\(d_{a,r}=(r-a)/2\).  The corresponding rank-one cover map is
\[
 \partial_{a,r}:\Delta_{A_1}(n_r)\longrightarrow
 \Delta_{A_1}(m-a),\qquad
 \partial_{a,r}(u_j)=
 \binom{j+d_{a,r}}{d_{a,r}}u_{j+d_{a,r}}.
\tag{5.5e$'$}\label{eq:contact-cover-map}
\]
Define the cover-kernel module
\[
 \mathscr R_r^{\rm end}:=
 \bigcap_{a\in\operatorname{Cov}_{m,p}(r)}
 \ker\partial_{a,r}\subseteq\Delta_{A_1}(n_r).
\tag{5.5e$^{\rm a}$}\label{eq:contact-endpoint-module}
\]
For \(j=(k-r)/2\), put
\[
 \eta_{r,k}^{m,p}:=
 \begin{cases}
 1,&j\in\mathbf Z,\ 0\le j\le n_r,\quad
     \binom{j+d_{a,r}}{d_{a,r}}\equiv0\pmod p
     \ \text{for every }a\in\operatorname{Cov}_{m,p}(r),\\
   &\hspace{28mm}\text{and }\bigl(j=0\text{ or }p\mid n_r-j+1\bigr),\\
 0,&\text{otherwise}.
 \end{cases}
\tag{5.5e$^{\rm b}$}\label{eq:contact-Ringel-endpoint}
\]

\begin{theorem}[Howe--Ringel formula for the Contact top fibre]
\label{thm:contact-Howe-Ringel-head}
For \(0\le k\le m\), let
\[
 Q_{\mathrm{pr}}^k:=
 \Lambda^kV^*/(\omega\wedge\Lambda^{k-2}V^*).
\]
Then
\[
 \operatorname{hd}_{\operatorname{Sp}(V)}Q_{\mathrm{pr}}^k
 \cong
 \bigoplus_{r=0}^m
 L_0(\omega_r)^{\oplus\eta_{r,k}^{m,p}}.
\tag{5.5e$''$}\label{eq:contact-Howe-Ringel-head}
\]
The inducing horizontal quotient obtained after imposing the top
Lefschetz relations is
\[
 Q_{\mathrm{pr}}^k\boxtimes\kk_{\,k-2m-2}.
\tag{5.5e$'''$}\label{eq:contact-minus-top-fiber}
\]
\end{theorem}

\begin{proof}
The symplectic exterior-power filtration has Weyl sections
\(\Delta(\omega_{k-2i})\), and every composition factor of a fundamental
Weyl module is again fundamental
\cite{PremetSuprunenko,AdamovichFundamental,BaranovSuprunenko}.  Hence \(Q_{\mathrm{pr}}^k\) has no
nonfundamental simple quotient.
Applying \(\Hom_{\operatorname{Sp}(V)}(-,L_0(\omega_r))\) to
\[
 \Lambda^{k-2}V^*\xrightarrow{\ \omega\wedge-\ }
 \Lambda^kV^*\longrightarrow Q_{\mathrm{pr}}^k\longrightarrow0
\]
identifies the multiplicity of \(L_0(\omega_r)\) in the head with
\[
 \dim\ker\!\left[
 \Hom_{\operatorname{Sp}(V)}(\Lambda^kV^*,L_0(\omega_r))
 \longrightarrow
 \Hom_{\operatorname{Sp}(V)}(\Lambda^{k-2}V^*,L_0(\omega_r))
 \right].
\tag{5.5e$^{\rm iv}$}\label{eq:contact-head-Hom-kernel}
\]
Adamovich's submodule-lattice theorem evaluates this kernel: the radical
of \(\Delta(\omega_r)\) is generated by the images
of the direct-cover maps
\[
 c_{a,r}:\Delta(\omega_a)\longrightarrow\Delta(\omega_r),
 \qquad a\in\operatorname{Cov}_{m,p}(r).
\]
Therefore
\[
 \Hom_{\operatorname{Sp}(V)}
 (L_0(\omega_r),\Lambda^\bullet V^*)
 =
 \ker\!\left[
 \Hom(\Delta(\omega_r),\Lambda^\bullet V^*)
 \longrightarrow
 \bigoplus_{a\in\operatorname{Cov}_{m,p}(r)}
 \Hom(\Delta(\omega_a),\Lambda^\bullet V^*)\right].
\tag{5.5e$^{\rm iv,a}$}\label{eq:contact-cover-kernel}
\]
The nonsemisimple exterior Howe theorem identifies
\[
 \Hom_{\operatorname{Sp}(V)}
 (\Delta(\omega_r),\Lambda^\bullet V^*)
 \cong\nabla_{A_1}(m-r).
\tag{5.5e$^{\rm iv,H}$}\label{eq:contact-dual-Weyl-Howe}
\]
This is the nonquantum exterior duality of
\cite[Theorem~2.1]{AdamovichRybnikov}; the Weyl/dual-Weyl orientation is
stated explicitly in \cite[Theorem~1B.1(c), (1B.2)]{BodishTubbenhauer},
and the positive-characteristic filtration is also treated in
\cite{McNinchHowe}.  Apply contragredient Ringel duality to the radical
cover presentation.  It reverses the cover arrow and gives the unique
rank-one map
\[
 \partial_{a,r}:\Delta_{A_1}(n_r)\longrightarrow
 \Delta_{A_1}(n_r+2d_{a,r})
\]
carrying the highest vector to $f^{(d_{a,r})}u_0$.  Adamovich's direct-cover
criterion \cite[Theorem]{AdamovichFundamental} gives
\[
 n_r+d_{a,r}+1=0\quad\text{in }\kk,
\tag{5.5e$^{\rm iv,C}$}\label{eq:contact-cover-congruence}
\]
which is exactly the condition making this a rank-one $A_1$-map.  Since
\[
 f^{(j)}f^{(d)}=\binom{j+d}{d}f^{(j+d)},
\]
this map is \(\partial_{a,r}\) in
\eqref{eq:contact-cover-map}.  Explicitly, the binomial recursion gives
\[
 f\partial_{a,r}=\partial_{a,r}f,
 \qquad
 (e\partial_{a,r}-\partial_{a,r}e)(u_j)
 =(n_r+d_{a,r}+1)
   \binom{j+d_{a,r}-1}{d_{a,r}-1}u_{j+d_{a,r}-1},
\]
so \eqref{eq:contact-cover-congruence} verifies both Chevalley generators
and fixes the arrow orientation.  Thus
\[
 \dim\Hom_{\operatorname{Sp}(V)}
 (L_0(\omega_r),\Lambda^kV^*)
 =\dim\mathscr R_r^{\rm end}[m-k].
\tag{5.5e$^{\rm iv,b}$}\label{eq:contact-cover-endpoint-identification}
\]
Every weight space of \(\mathscr R_r^{\rm end}\) is zero or one
dimensional.  The weight \(m-k\) is represented by \(u_j\), where
\(j=(k-r)/2\); it belongs to the cover kernel precisely when all the
binomial coefficients in \eqref{eq:contact-Ringel-endpoint} vanish.
It is killed by the ordinary raising operator precisely when \(j=0\) or
\(p\mid n_r-j+1\).  Duality identifies this raising kernel with
\eqref{eq:contact-head-Hom-kernel}.  Hence its dimension is
\(\eta_{r,k}^{m,p}\), proving
\eqref{eq:contact-Howe-Ringel-head}.

The highest horizontal coefficient of \(\Omega^k/I^k\) is
\(\Lambda^kV^*\).  The universal projected closed relations at that grade
contain the Lefschetz submodule
\(\omega\wedge\Lambda^{k-2}V^*\), while the grading element
\(z\) acts by \(k-2m-2\).  Quotienting by this submodule gives
\eqref{eq:contact-minus-top-fiber}.
\end{proof}

\begin{lemma}[Low-socle vanishing in the primitive exterior quotient]
\label{lem:contact-primitive-no-low-socle}
Let
\[
 \overline\Lambda^\bullet V^*
 :=\Lambda^\bullet V^*/(\omega).
\]
For \(0\le k<m\),
\[
 \bigcap_{v\in V^*}
 \ker\!\left(
 v\wedge-:\overline\Lambda^kV^*
 \longrightarrow\overline\Lambda^{k+1}V^*
 \right)=0.
\tag{5.5e$^{\rm iv,c}$}\label{eq:contact-primitive-no-low-socle}
\]
\end{lemma}

\begin{proof}
Choose a Darboux basis
\(e_1,f_1,\ldots,e_m,f_m\), and order exterior monomials first by the
greatest occupied Darboux pair and then lexicographically inside that
pair.  Row reduction of the columns
\(\omega\wedge N\), in this order, gives a normal form in which no leading
monomial contains the distinguished complete pair at its first reducible
pair.  Reduce \(\alpha\) to this normal form and let \(M\) be its leading
monomial.  Since \(k<m\), some Darboux pair \(e_j,f_j\) is entirely absent
from \(M\); choose \(v\in\{e_j,f_j\}\) so that exterior multiplication
preserves the chosen order.  If \(v\wedge M\) were a pivot monomial of a
column \(\omega\wedge\beta\), its distinguished complete pair would
already be a complete pair of \(M\).  Removing \(v\) from that pivot
column gives the corresponding pivot reduction of \(M\) in degree \(k\),
contrary to normality.  Hence \(v\wedge M\) cannot cancel in the normal
form of \(v\wedge\alpha\), and \(v\wedge[\alpha]\ne0\).  This proves
\eqref{eq:contact-primitive-no-low-socle}.
\end{proof}

\begin{proposition}[Parabolic presentation of the Contact form module]
\label{prop:contact-exceptional-induced}
For \(0\le k\le m\) there is an exact sequence
\[
 0\longrightarrow\mathcal C_k^-
 \longrightarrow
 u_0(K)\otimes_{u_0(P_K)}
   (Q_{\mathrm{pr}}^k\boxtimes\kk_{\,k-2m-2})
 \longrightarrow\mathscr E_k^-\longrightarrow0.
\tag{5.5e$^{\rm v}$}\label{eq:contact-exceptional-induced}
\]
\end{proposition}

\begin{proof}
Quotienting \(\Omega^k\) by the contact ideal
\(\theta\wedge\Omega^{k-1}\) gives the negative-PBW model induced from the
highest horizontal fibre \(\Lambda^kV^*\).  For a horizontal
\((k-2)\)-form \(\gamma\),
\[
 L\gamma=d(\theta\wedge\gamma)+\theta\wedge d\gamma,
\]
so the Lefschetz submodule vanishes after quotienting by closed forms and
the contact ideal.  Quotienting first by the \(K\)-submodule it generates
gives the induced module with fibre
\(Q_{\mathrm{pr}}^k\boxtimes\kk_{\,k-2m-2}\).  All remaining projected closed-form
relations generate \(\mathcal C_k^-\), which gives
\eqref{eq:contact-exceptional-induced}.
\end{proof}

\begin{proposition}[Strictness of the projected closed relations]
\label{prop:contact-closed-top-strict}
Put
\[
 \mathcal I_k:=u_0(K)\otimes_{u_0(P_K)}
 (Q_{\mathrm{pr}}^k\boxtimes\kk_{\,k-2m-2}).
\]
For \(0\le k\le m\),
\[
 \mathcal C_k^-\cap(\mathcal I_k)_{\rm top}=0,\qquad
 \mathcal C_k^-\subseteq\rad_K\mathcal I_k.
\tag{5.5e$^{\rm vf}$}\label{eq:contact-closed-top-strict}
\]
Hence
\[
 \operatorname{hd}_K\mathscr E_k^-
 \cong\operatorname{hd}_K\mathcal I_k.
\tag{5.5e$^{\rm vg}$}\label{eq:contact-full-top-equality}
\]
\end{proposition}

\begin{proof}
Let a projected closed relation have nonzero inducing-top component,
represented in the horizontal splitting by the top divided monomial
\(t^{(p-1)}X^{(\tau)}\alpha_0\), with
\([\alpha_0]\in Q_{\mathrm{pr}}^k\).  For a closed lift
\(\alpha+\theta\wedge\beta\) having this top component, compare in
\[
 d_H\alpha+L\beta=0
\]
the coefficients of the distinct divided monomials
\(t^{(p-1)}X^{(\tau-e_i)}\).  Every derivative coefficient is a unit.
Hence
\[
 e_i^*\wedge[\alpha_0]=f_i^*\wedge[\alpha_0]=0
 \quad(1\le i\le m)
\]
in \(\overline\Lambda^{k+1}V^*\).  If \(k<m\), the no-low-socle \cref{lem:contact-primitive-no-low-socle} gives \([\alpha_0]=0\).

For \(k=m\), use also the second closed-lift equation
\[
 T\alpha-d_H\beta=0.
\]
Compare the coefficient of the pure Reeb predecessor
\(t^{(p-2)}X^{(\tau)}\).  In the divided-power normalization the first
term contributes \([\alpha_0]\), with unit coefficient.  Every top
coefficient of \(\beta\) forced
by the first equation misses one horizontal divided variable; a horizontal
derivative cannot restore that variable without producing exponent \(p\),
which is zero in the truncated algebra.  Hence \(d_H\beta\) has no
\(t^{(p-2)}X^{(\tau)}\)-coefficient.  Thus \([\alpha_0]=0\) also in the
middle degree.  This proves the first assertion of
\eqref{eq:contact-closed-top-strict}.

Every simple quotient of \(\mathcal I_k\) is graded and generated by its
inducing top.  If \(\mathcal C_k^-\) had nonzero image in such a quotient,
its image would be the whole simple module.  Applying positive homogeneous
operators to a homogeneous preimage of a top vector would then give a
nonzero element of
\(\mathcal C_k^-\cap(\mathcal I_k)_{\rm top}\), a contradiction.
\(\mathcal C_k^-\) therefore lies in every maximal submodule of
\(\mathcal I_k\), proving the radical inclusion.  The quotient has the
same semisimple head as \(\mathcal I_k\); Frobenius reciprocity and
\cref{thm:contact-Howe-Ringel-head} give
\eqref{eq:contact-full-top-equality}.
\end{proof}

\begin{corollary}[Contact heads of the form modules]
\label{cor:contact-form-heads}
For \(0\le k\le m\),
\[
 \operatorname{hd}_K\mathscr E_k^-
 \cong
 \bigoplus_{r=0}^m
 L_K(\omega_r,k-2m-2)^{\oplus\eta_{r,k}^{m,p}}.
\tag{5.5e$^{\rm vi}$}\label{eq:contact-minus-full-head}
\]
The heads of the plus modules are obtained from
\eqref{eq:contact-minus-full-head} by contragredient duality.
\end{corollary}

\begin{proof}
For a finite module, the multiplicity of a simple module in its semisimple
head is the dimension of the Hom space to that simple.  Apply
\cref{prop:contact-closed-top-strict,thm:contact-Howe-Ringel-head};
the central character is fixed by
\eqref{eq:contact-minus-top-fiber}.  This gives
\eqref{eq:contact-minus-full-head}.
\end{proof}

\begin{remark}[Form heads and exceptional labels]
For $p\le2m+2$, the semisimple heads of the form quotients and the Holmes
exceptional labels \eqref{eq:contact-exceptional-labels} are distinct
parametrizations.  They agree in the stable range of
\cref{thm:contact-stable-form-models}.  Outside that range, Holmes labels not
occurring among the form heads are treated by the restricted derivation
complex \eqref{eq:restricted-rootp}.
\end{remark}

\begin{proposition}[Horizontal closed-lift presentation]\label{prop:contact-closed-lift}
The map
\[
\rho_k:Q^k\longrightarrow\mathscr E_k^-,
\qquad
[\alpha]_Q\longmapsto[\alpha]_{\mathscr E},
\tag{5.5f}
\]
is a well-defined surjection.  Its kernel is
\[
\mathscr Z_{\rm lift}^k
=
\left\{
[\alpha]_Q:\ \exists\beta\in\Omega_H^{k-1},\quad
 d_H\alpha+L\beta=0,\quad
 T\alpha-d_H\beta=0
\right\}.
\tag{5.5g}
\]
There is therefore an exact $K$-equivariant sequence
\[
0\longrightarrow\mathscr Z_{\rm lift}^k
\longrightarrow Q^k
\stackrel{\rho_k}{\longrightarrow}\mathscr E_k^-
\longrightarrow0.
\tag{5.5h}
\]
\end{proposition}

\begin{proof}
Every form is congruent modulo $I^k$ to a horizontal form, so $\rho_k$ is
surjective.  If $\gamma\in\Omega_H^{k-2}$, then
\[
L\gamma=d(\theta\wedge\gamma)+\theta\wedge d\gamma.
\]
The first summand is closed and the second belongs to $I^k$; thus $L\gamma$
vanishes in $\mathscr E_k^-$ and $\rho_k$ factors through $Q^k$.

The equality $\rho_k([\alpha]_Q)=0$ holds precisely when $[\alpha]_Q$ has
a representative, again denoted $\alpha$, such that
$\alpha+\theta\wedge\beta$ is closed for some
$\beta\in\Omega_H^{k-1}$.  By (5.5c), this is equivalent to the two
equations in (5.5g).  This proves the kernel statement.  The
construction is $K$-equivariant because $I^k$, $Z_{\rm dR}^k$, and the
horizontal quotient are $K$-stable.
\end{proof}

For $[\alpha]_Q\in Z_Q^k$, choose $\beta$ with
$d_H\alpha+L\beta=0$ and put
\[
\Gamma(\alpha,\beta):=T\alpha-d_H\beta.
\tag{5.5i}
\]
Then (5.5c)--(5.5d) and $d^2=0$ give
\[
L\Gamma(\alpha,\beta)=0,
\qquad
d_H\Gamma(\alpha,\beta)=0.
\tag{5.5j}
\]
Accordingly, $\Gamma(\alpha,\beta)$ is a cocycle in
\[
K_L^k:=\ker(L:\Omega_H^k\to\Omega_H^{k+2}).
\]
If $\beta$ is changed by another solution of the first equation, the
curvature changes by a $d_H$-boundary in $K_L^k$.  Replacing $\alpha$ by
$\alpha+L\gamma$ and $\beta$ by $\beta-d_H\gamma$ leaves its cohomology
class unchanged, using (5.5d).  Finally, if
$[\alpha]_Q=d_Q[\xi]_Q$, take $\alpha=d_H\xi$ and $\beta=T\xi$; then
$\Gamma(\alpha,\beta)=0$.  Hence
\[
\mathcal D_k^{\rm curv}:H^k(Q,d_Q)\longrightarrow H^k(K_L,d_H),
\qquad
[\alpha]_Q\longmapsto[T\alpha-d_H\beta]
\tag{5.5k}
\]
is a well-defined $K$-equivariant map.

\begin{proposition}[Curvature kernel and closed lifts]\label{prop:contact-curv-kernel}
One has
\[
\ker\mathcal D_k^{\rm curv}
=\mathscr Z_{\rm lift}^k/B_Q^k.
\tag{5.5l}
\]
Equivalently, if
\[
\mathfrak C_k:H_{\rm dR}^k\longrightarrow H^k(Q,d_Q)
\]
sends a closed form $\alpha+\theta\wedge\beta$ to $[\alpha]_Q$, then
\[
H_{\rm dR}^k\stackrel{\mathfrak C_k}{\longrightarrow}
H^k(Q,d_Q)\stackrel{\mathcal D_k^{\rm curv}}{\longrightarrow}
H^k(K_L,d_H)
\tag{5.5m}
\]
is exact at the middle term.
\end{proposition}

\begin{proof}
A closed lift has $\Gamma=0$, so
$\mathscr Z_{\rm lift}^k/B_Q^k\subseteq\ker\mathcal D_k^{\rm curv}$.
Conversely suppose $\mathcal D_k^{\rm curv}([\alpha]_Q)=0$.  Choose
$\beta$ with $d_H\alpha+L\beta=0$.  Then
\[
\Gamma(\alpha,\beta)=d_H\eta
\]
for some $\eta\in K_L^{k-1}$.  Since $L\eta=0$, replacing $\beta$ by
$\beta+\eta$ preserves $d_H\alpha+L\beta=0$ and gives
\[
T\alpha-d_H(\beta+\eta)=0.
\]
Hence $[\alpha]_Q\in\mathscr Z_{\rm lift}^k$, which proves the reverse
inclusion after quotienting by $B_Q^k$.  This is (5.5l), and is equivalent
to exactness of (5.5m).
\end{proof}

Define the maps in the exceptional short exact sequence by
\[
\pi_k:\mathscr E_k^-\longrightarrow\operatorname{im}d_Q,
\qquad
\pi_k([\alpha]_{\mathscr E})=d_Q[\alpha]_Q.
\tag{5.5n}
\]
This is well defined because
$\mathscr Z_{\rm lift}^k\subseteq Z_Q^k$.  For a class
$c=\mathcal D_k^{\rm curv}([\alpha]_Q)$ define
\[
\iota_k(c):=[\alpha]_{\mathscr E}.
\tag{5.5o}
\]
\Cref{prop:contact-curv-kernel} shows that $\iota_k$ is well
defined: two preimages of $c$ differ by an element of
$\ker\mathcal D_k^{\rm curv}=\mathscr Z_{\rm lift}^k/B_Q^k$, hence by a
class killed in $\mathscr E_k^-$.  It is injective for the same reason.

The operator $\mathcal D_k^{\rm curv}$ plays the role, in this modular
algebraic setting, of Rumin's second-order operator on a contact manifold
\cite{Rumin}, which replaces the de~Rham differential in the middle degree
where the contact form's kernel fails to be integrable.  The filtration of
$\mathscr E_k^-$ by $\operatorname{im}\mathcal D_k^{\rm curv}$ will be
called the Rumin-type filtration.

\begin{theorem}[Exceptional Rumin-type curvature sequence]\label{thm:contact-rumin-ses}
For every $0\le k\le m$ there is a canonical $K$-equivariant short exact
sequence
\[
0\longrightarrow\operatorname{im}\mathcal D_k^{\rm curv}
\stackrel{\iota_k}{\longrightarrow}\mathscr E_k^-
\stackrel{\pi_k}{\longrightarrow}
\operatorname{im}(d_Q:Q^k\to Q^{k+1})
\longrightarrow0.
\tag{5.5p}
\]
\end{theorem}

\begin{proof}
The map $\pi_k$ is surjective because $Q^k\to Q^k/Z_Q^k\cong\operatorname{im}d_Q$
is surjective.  Its kernel is
\[
Z_Q^k/\mathscr Z_{\rm lift}^k.
\]
Using
$H^k(Q,d_Q)=Z_Q^k/B_Q^k$ and
\cref{prop:contact-curv-kernel},
\[
Z_Q^k/\mathscr Z_{\rm lift}^k
\cong
H^k(Q,d_Q)/\ker\mathcal D_k^{\rm curv}
\cong
\operatorname{im}\mathcal D_k^{\rm curv}.
\]
Under this identification the injection is $\iota_k$ from (5.5o).
The sequence is exact, and all its maps are equivariant.
\end{proof}

The maps in (5.5p) are obtained from the matrices of $L,d_H,T$ by taking
kernels, images, and quotients.

\begin{theorem}[Stable-range Contact form models]
\label{thm:contact-stable-form-models}
Assume $p>2m+2$.  For $0\le k\le m$, the minus form quotient
$\mathscr E_k^-$ of \textup{(5.5b)} is simple and
\[
 \mathscr E_k^-\cong L_K(\omega_k^-).
\tag{5.20a}\label{eq:contact-stable-minus-model}
\]
Put
\[
 \mathscr E_1^+:=\kk,
 \qquad
 \mathscr E_k^+:=(\mathscr E_{k-2}^-)^\vee
 \quad(2\le k\le m+1).
\]
Then contragredient duality gives
\[
 \mathscr E_k^+\cong L_K(\omega_k^+).
\tag{5.20b}\label{eq:contact-stable-plus-model}
\]
At the middle degree,
\[
 (\mathscr E_m^-)^\vee\cong\mathscr E_m^-.
\tag{5.20b$'$}\label{eq:contact-stable-middle-selfdual}
\]
The differential-form quotients therefore carry their Holmes labels.
\end{theorem}

\begin{proof}
In Holmes' notation the quotient $\mathscr E_k^-$ is the module
$\overline\Omega_k$.  His highest-vector construction gives a surjection
$Z_K(\omega_k^-)\twoheadrightarrow\overline\Omega_k$
\cite[Lemma~4.8]{HolmesContact}, and
\cite[Theorem~5.2]{HolmesContact} proves that the quotient is simple and
identifies it with $L_K(\omega_k^-)$.  The hypothesis there is
$p>2m+1$; for an odd prime this is equivalent to $p>2m+2$.
In Holmes' indexing, write $\omega_j^{+,H}=(\omega_j,-j)$.  Then
\cite[Theorem~5.3]{HolmesContact} gives
\[
 L_K(\omega_j^-)^\vee\cong L_K(\omega_{j+1}^{+,H})
 \quad(0\le j<m),
 \qquad
 L_K(\omega_m^-)^\vee\cong L_K(\omega_m^-).
\]
Since $\omega_k^+=\omega_{k-1}^{+,H}$, this is precisely
\eqref{eq:contact-stable-plus-model} and
\eqref{eq:contact-stable-middle-selfdual}.  In this range the formal top Contact
monomial is generated, so the full and derived algebras coincide and no
folded label occurs.
\end{proof}

\begin{lemma}[Propagation of the mixed Contact cocycle identity]
\label{lem:contact-mixed-propagation}
Let $E$ be a $K$-module and let $d_-$ and $d_+$ be ordinary derivations on
$P_-=K_-\rtimes K_0$ and $P_+=K_0\ltimes K_{>0}$ which agree on $K_0$.
Let $d:K=P_-+P_+\to E$ be the resulting linear map.  If
\[
 \Phi_d(x,y):=d([x,y])-x\cdot d(y)+y\cdot d(x)=0
 \qquad(x\in K_{-1},\ y\in K_1),
\tag{5.20c$'$}\label{eq:contact-mixed-initial}
\]
then $\Phi_d(K_-,K_{>0})=0$.  Hence $d$ is an ordinary derivation of $K$.
\end{lemma}

\begin{proof}
First fix $x\in K_{-1}$ and induct on $j\ge1$.  The case $j=1$ is the
hypothesis.  For $j\ge2$, use
$K_j=[K_1,K_{j-1}]$ from
\cref{lem:contact-K1-generation} and write $y=[u,v]$ with
$u\in K_1$, $v\in K_{j-1}$.  Jacobi gives
\[
 [x,[u,v]]=[[x,u],v]+[u,[x,v]].
\tag{5.20c$''$}\label{eq:contact-mixed-Jacobi-step}
\]
Expand $d$ on the right.  The pair $(u,v)$ lies in $P_+$, the pair
$([x,u],v)$ lies in $P_+$ because $[x,u]\in K_0$, and the mixed pair
$(x,v)$ satisfies the induction hypothesis.  Substitution in
\eqref{eq:contact-mixed-Jacobi-step} cancels the action terms in pairs and
gives $\Phi_d(x,[u,v])=0$.  By linearity this proves
$\Phi_d(K_{-1},K_j)=0$ for every $j>0$.

The negative algebra has depth two and
$K_{-2}=[K_{-1},K_{-1}]$.  Write $x=[x_1,x_2]$ and apply Jacobi once more:
\[
 [[x_1,x_2],y]=[x_1,[x_2,y]]-[x_2,[x_1,y]].
\]
The derivation identity inside $P_-$ for $(x_1,x_2)$ and the result of the
first paragraph for both $(x_i,y)$ give
$\Phi_d([x_1,x_2],y)=0$.  These are all mixed homogeneous pairs; pairs
inside either parabolic already satisfy the cocycle identity.  Thus $d$ is
a derivation on $K$.
\end{proof}

\begin{definition}[The Contact exceptional gluing module]
\label{def:contact-total-gluing}
For a coefficient module \(E\), put
\[
 \mathcal D_K(E):=
 Z_*^1(P_-,E)\times_{Z_*^1(K_0,E)}Z_*^1(P_+,E).
\]
Restriction gives an injection
\[
 \operatorname{res}:Z_*^1(K,E)\longrightarrow\mathcal D_K(E).
\]
The gluing quotient is
\[
 \mathfrak G_K(E):=
 \mathcal D_K(E)/\operatorname{res}Z_*^1(K,E).
\tag{5.20c$^{\rm iii}$}\label{eq:contact-total-gluing}
\]
\end{definition}

\begin{proposition}[Filtration of the Contact gluing module]
\label{prop:contact-total-gluing-filtration}
The mixed-defect map induces an injection
\[
 \mathfrak G_K(E)\hookrightarrow
 \Hom_{K_0}(K_{-1}\otimes K_1,E),
\tag{5.20c$^{\rm iv}$}\label{eq:contact-total-gluing-embedding}
\]
whose image is \(\operatorname{im}\kappa\).  An ordered decomposition of
\[
 K_{-1}\otimes K_1\cong V\otimes(S^3V\oplus V)
\]
into \(K_0\)-primitive blocks therefore induces a filtration of
\(\mathfrak G_K(E)\); changing the order or parenthesization changes the
filtration but not the module.
\end{proposition}

\begin{proof}
By \cref{lem:contact-mixed-propagation}, a compatible parabolic pair has
zero mixed defect precisely when it is the restriction of an ordinary
derivation of \(K\).  Its restricted defect is Frobenius-semilinear by
\cref{lem:p-semilinear} and vanishes on bases of the two parabolics; their
union spans \(K\), so the defect is zero.  Thus the kernel of the
mixed-defect map on \(\mathcal D_K(E)\) is
\(\operatorname{res}Z_*^1(K,E)\), which proves
\eqref{eq:contact-total-gluing-embedding}.  Intersecting its image with
successive partial sums of an ordered primitive decomposition gives the
asserted filtration.
\end{proof}

\begin{lemma}[The folded natural coupling]\label{lem:contact-folded-coupling}
Let \(0\le k\le m\).  The minus address \(\omega_k^-\) and the plus
address \(\omega_{k+1}^+\) have the same fundamental
\(\operatorname{Sp}(V)\)-label.  On the one-dimensional natural Pieri
summand of their inducing heads, the corresponding block of the mixed map
\(\kappa\) is multiplication by
\[
 m-k+1\in\kk.
\tag{5.20d}\label{eq:contact-folded-coupling}
\]
Moreover
\[
 \omega_k^-=\omega_{k+1}^+
 \quad\Longleftrightarrow\quad p\mid(m-k+1).
\tag{5.20d$'$}\label{eq:contact-folded-label}
\]
This inducing-head block is invertible away from the fold and is zero at
the folded vertex.
\end{lemma}

\begin{proof}
For \(k=0\), both symplectic factors are trivial.  Normalize the invariant
in \(V\otimes V\subset K_{-1}\otimes K_1\) by one half of the symplectic
coevaluation.  On this line the mixed Contact relation is minus one half
of the difference of the two central characters.  Since
\(\zeta_0^-=-2m-2\) and \(\zeta_1^+=0\), its scalar is \(m+1\).

Assume now that \(k\ge1\).
On the primitive horizontal exterior fibre the negative and positive
parabolic maps on this summand are, up to the common normalization of the
Pieri generator, symplectic contraction \(F\) and Lefschetz
multiplication \(E=\omega\wedge-\).  Their contribution to the mixed
Contact relation is therefore the commutator \(FE-EF\).  The integral
\(\mathfrak{sl}_2\) relation on exterior degree \(k-1\) gives
\[
 (FE-EF)|_{\Lambda^{k-1}V^*}
 =(m-k+1)\operatorname{id}.
\]
The relevant Pieri space is one-dimensional, so this is the entire
coupling block.  Together with the case $k=0$, it proves
\eqref{eq:contact-folded-coupling}.

By \eqref{eq:contact-exceptional-labels},
\[
 \omega_k^-=(\omega_k,k-2m-2),
 \qquad
 \omega_{k+1}^+=(\omega_k,-k).
\]
Their central coordinates agree exactly when
\(2(k-m-1)=0\) in \(\mathbb F_p\).  Since \(p\) is odd, this is equivalent
to \(p\mid(m-k+1)\), proving \eqref{eq:contact-folded-label}.
\end{proof}

\begin{theorem}[Contact gluing]
\label{thm:contact-gluing-complex}
Let \(p\ge5\), and let \(L_\alpha,L_\beta\) be simple restricted
$K$-modules.  Put
$E_{\alpha\beta}=\Hom_\kk(L_\alpha,L_\beta)$.  Write
\[
 P_-=K_-\rtimes K_0,
 \qquad P_+=K_0\ltimes K_{>0},
\]
and let $Z_*^1(P_\pm,E)$ and $Z_*^1(K_0,E)$ denote restricted derivations.
Define
\[
 \mathcal G_{\alpha\beta}^0=E_{\alpha\beta},
\]
\[
 \mathcal G_{\alpha\beta}^1=
 \left\{(d_-,d_+)\in
 Z_*^1(P_-,E)\oplus Z_*^1(P_+,E):
 d_-|_{K_0}=d_+|_{K_0}\right\},
\tag{5.20e}\label{eq:contact-gluing-C1}
\]
and
\[
 \mathcal G_{\alpha\beta}^2=
 \Hom_{K_0}(K_{-1}\otimes K_1,E_{\alpha\beta}).
\]
Write
\[
 L_\alpha=\bigoplus_dL_{\alpha,d},\qquad
 L_\beta=\bigoplus_eL_{\beta,e},
\]
and let the central element of $K_0$ act on the inducing heads by
$\zeta$ and $\zeta'$, respectively.  It acts on $L_{\alpha,d}$ and
$L_{\beta,e}$ by $\zeta-d$ and $\zeta'-e$.  Then
\[
 \mathcal G_{\alpha\beta}^2\cong
 \bigoplus_{\substack{d,e\\ e-d\equiv\zeta'-\zeta\ (\mathrm{mod}\ p)}}
 \Hom_{\operatorname{Sp}(V)}\!\left(
 V\otimes(S^3V\oplus V)\otimes L_{\alpha,d},L_{\beta,e}
 \right).
\tag{5.20e$'$}\label{eq:contact-mixed-Pieri-block}
\]
Thus the mixed rows are type-\(C\) equivariance maps between the graded
constituents of the two simple modules.
The differentials are
\[
 b(e)=(\delta_e|_{P_-},\delta_e|_{P_+}),
\]
and
\[
 \begin{aligned}
 \kappa(d_-,d_+)(x\otimes y)
 &:=(d_-\cup d_+)([x,y])
   -x\cdot d_+(y)+y\cdot d_-(x),
 \\[-2mm]
 &\hspace{35mm}x\in K_{-1},\ y\in K_1,
 \end{aligned}
\tag{5.20f}\label{eq:contact-mixed-map}
\]
where $(d_-\cup d_+)([x,y])$ uses the common restriction of $d_-$ and
$d_+$ to $K_0$.  Then there is an exact sequence
\[
 0\longrightarrow Z_*^1(K,E_{\alpha\beta})
 \longrightarrow\mathcal G_{\alpha\beta}^1
 \longrightarrow\mathfrak G_K(E_{\alpha\beta})
 \longrightarrow0,
\tag{5.20g}\label{eq:contact-gluing-Ext}
\]
and
\[
 \Ext^1_{u_0(K)}(L_\alpha,L_\beta)
 \cong
 Z_*^1(K,E_{\alpha\beta})/\operatorname{im}b.
\tag{5.20g$'$}\label{eq:contact-gluing-Ext-quotient}
\]
On the inducing-head common fundamental summand joining the minus address
\(\omega_k^-\) to the plus address \(\omega_{k+1}^+\), the corresponding
diagonal block of \(\kappa\) is the scalar \(m-k+1\) of
\eqref{eq:contact-folded-coupling}.
\end{theorem}

\begin{proof}
A restricted derivation of $K=P_-+P_+$ restricts to a pair in
\eqref{eq:contact-gluing-C1}, and the cocycle identity on a mixed bracket is
\eqref{eq:contact-mixed-map}.  For $z\in K_0$, the derivation
identities for $(z,x)$ and $(z,y)$, followed by Jacobi, give
\[
 z\cdot\kappa(d_-,d_+)(x\otimes y)
 =\kappa(d_-,d_+)([z,x]\otimes y)
  +\kappa(d_-,d_+)(x\otimes[z,y]).
\]
Hence $\kappa(d_-,d_+)\in
\Hom_{K_0}(K_{-1}\otimes K_1,E_{\alpha\beta})$.
The central element of \(K_0\) acts trivially on
\(K_{-1}\otimes K_1\), while
\(K_{-1}\cong V\) and \(K_1\cong S^3V\oplus V\) as
\(\operatorname{Sp}(V)\)-modules.  On
$\Hom_\kk(L_{\alpha,d},L_{\beta,e})$ its scalar is
$\zeta'-\zeta+d-e$.  Taking the zero central-weight part and applying
tensor--Hom adjunction gives \eqref{eq:contact-mixed-Pieri-block}.
Conversely suppose a compatible pair
lies in $\ker\kappa$.  By
\cref{lem:contact-mixed-propagation} the pair glues to an ordinary
derivation on all of $K$.  Its restricted defect
\[
 d(x^{[p]})-x^{p-1}d(x)
\]
is Frobenius-semilinear by \cref{lem:p-semilinear}.  It vanishes on a basis
of each parabolic, and the union of those bases spans $P_-+P_+=K$.
It therefore vanishes identically, so the glued derivation is restricted.
Compatible inner derivations form the image of \(b\).  Together
with \cref{prop:contact-total-gluing-filtration}, this proves
\eqref{eq:contact-gluing-Ext} and
\eqref{eq:contact-gluing-Ext-quotient}.

The common fundamental minus--plus block is computed in
\cref{lem:contact-folded-coupling}.
\end{proof}

\begin{theorem}[Restricted Contact extension complexes]\label{thm:contact-quiver}
For $p\ge5$, the vertices of the restricted Contact Ext-quiver are the
Holmes-labelled parabolic heads $L_K(\alpha)$,
$\alpha\in\Lambda_K^{\rm Hol}(m,p)$, of
\cref{thm:contact-vertex-parametrization}.  If the source
\(L_K(\alpha)=Z_K(\alpha)\) is nonexceptional, then for every simple target
\[
 \Ext^1_{u_0(K)}(Z_K(\alpha),L_K(\beta))
 \cong
 H_*^1\!\left(P_K,
 \Hom_\kk(S_\alpha,\operatorname{Res}_{P_K}L_K(\beta))\right).
\tag{5.21}\label{eq:contact-mixed-parabolic}
\]
The right side is the restricted derivation complex determined by the
Contact graded bracket and the form or PBW matrices of the target.  For two nonexceptional
vertices it decomposes as \textup{(5.1)}--\textup{(5.4j)}.  If the source
is exceptional and the target is nonexceptional, rigid duality applies
\eqref{eq:contact-mixed-parabolic} to the reverse dual pair.

If both vertices are exceptional, the restricted derivation complex
\eqref{eq:restricted-rootp}, with the Holmes parabolic-head matrices,
computes the arrow space.  The two-parabolic complex of
\cref{thm:contact-gluing-complex} gives an equivalent presentation.  When
both modules occur in the form heads, \cref{cor:contact-form-heads}
supplies their head multiplicities.  In the stable range all exceptional
form labels have the simple realizations of
\cref{thm:contact-stable-form-models}.
\end{theorem}

\begin{proof}
Restricted PBW--Shapiro gives \eqref{eq:contact-mixed-parabolic}; the
restricted derivation rows on \(P_K\) use the Contact bracket,
restricted \(p\)-map, and target action.  If the target is also
nonexceptional, (5.1) is the restricted nilradical
cocycle matrix in each total contact degree.  The exact sequences in (5.2a)
identify $H_q=H_{*,q}^1(N,L(\mu))$, and (5.3)--(5.4c) compute the
$K_0$-coefficient space $\Hom_{K_0}(U_\lambda,H_q)$.  The restricted
Hochschild--Serre sequence is (5.4g); \cref{prop:contact-d2} gives zero
transgression by the graded restricted bar argument, and (5.4i) computes the
Levi term.  Thus (5.4j), followed by PBW--Shapiro, gives the arrow space.

The exceptional label set is stable under contragredient duality, as is its
nonexceptional complement.  Thus rigid duality converts the reverse mixed
orientation to \eqref{eq:contact-mixed-parabolic}.

For an exceptional pair, \cref{prop:root-p} applied to the Holmes matrices
agrees with the two-parabolic reconstruction of
\cref{thm:contact-gluing-complex}.
For form-head constituents,
\cref{thm:contact-Howe-Ringel-head,cor:contact-form-heads} determine the
head multiplicities.
\end{proof}

\subsection{Exceptional Contact quiver}

For a signed Holmes address $x$, use the Holmes labels of
\eqref{eq:contact-exceptional-labels}: write $x=(-,r)$ for
$\omega_r^-$, $0\le r\le m$, and $x=(+,s)$ for $\omega_s^+$,
$1\le s\le m+1$.  Define its fundamental head index and central character by
\[
 r_x:=
 \begin{cases}r,&x=(-,r),\\ s-1,&x=(+,s),\end{cases}
 \qquad
 \zeta_x:=
 \begin{cases}r-2m-2,&x=(-,r),\\1-s,&x=(+,s).
 \end{cases}
\tag{5.22}\label{eq:contact-signed-head-data}
\]
Fold-identical plus and minus addresses, and any signed address equal to the
trivial module, are identified before any formula is evaluated.  For
$0\le r\le m$ define the realization set
\[
 \mathcal A_r^-(m,p)
 :=\{a\in\{0,\ldots,m\}:a=r+2tp\text{ for some }t\in\mathbb Z,
       \ \eta_{r,a}^{m,p}=1\},
\tag{5.22a}\label{eq:contact-realization-set}
\]
where $\eta_{r,a}^{m,p}$ is the Howe--Ringel indicator
\eqref{eq:contact-Ringel-endpoint}.  Put $N=2m+1$ and
\[
 \epsilon_r^-:=
 \mathbf1_{\{r-2m-2\equiv0\ (\mathrm{mod}\ p)\}}
 \mathbf1_{\{0\triangleleft_{p,m}r\}},
\tag{5.22b}\label{eq:contact-epsilon}
\]
\[
 D_r^-:=\sum_{a\in\mathcal A_r^-}\binom Na+\epsilon_r^-,
\tag{5.22c}\label{eq:contact-Dminus}
\]
and
\[
 \widehat W_r^-:=
 \mathbf1_{\{p\mid(m-r)\}}
 \sum_{\substack{a\in\mathcal A_r^-\\a<m}}\binom{2m}{a},
\qquad
 R_r^-:=\sum_{a\in\mathcal A_r^-}\binom N{a+1}
          +\widehat W_r^-+\delta_{r,0}+\epsilon_r^-.
\tag{5.22d}\label{eq:contact-Rminus}
\]

\begin{theorem}[Exceptional Contact matrix]
\label{thm:contact-exceptional-closed}
Let $p\ge5$.  On the folded exceptional Contact vertex set define
\[
 (T_{(-,r)},U_{(-,r)})=(R_r^-,D_r^-),
 \qquad
 (T_{(+,s)},U_{(+,s)})=(D_{s-1}^-,R_{s-1}^-).
\tag{5.23a}\label{eq:contact-TU}
\]
For every ordered pair of exceptional simples $L_x,L_y$,
\[
\boxed{
\dim\Ext^1_{u_0(K)}(L_x,L_y)
=
\begin{cases}
0,&x=y,\\[1mm]
T_y,&x=\kk,\ y\ne\kk,\\[1mm]
U_x,&x\ne\kk,\ y=\kk,\\[1mm]
\mathbf1_{\{\zeta_x=\zeta_y\}}
\mathbf1_{\{r_x\ne r_y\}}
\mathbf1_{\{r_x\triangleleft_{p,m}r_y\text{ or }
              r_y\triangleleft_{p,m}r_x\}},
 &x,y\ne\kk,\ x\ne y.
\end{cases}}
\tag{5.23b}\label{eq:contact-exceptional-closed}
\]
Here equality of the central characters is in $\F_p$.  In particular
$\Ext^1_{u_0(K)}(\kk,\kk)=0$.
\end{theorem}

\begin{proof}
For every $p\ge5$ the restricted two-parabolic complex is formed with the
Holmes parabolic heads.  The passage from the Rumin cover to these heads is
as follows.

Fix an exceptional address $x$ and one degree occurring in its Rumin
realization.  In
the minus realization let $I_x$ be the induced prequotient of
\cref{prop:contact-exceptional-induced}; for the plus realization use its
contragredient dual.  Let $P_x\subset I_x$ be the Rumin relation submodule, generated by the
projected closed Cartier relations together with
the integral Lefschetz-wall relation when the wall is present, and put
$B_x=I_x/P_x$.  Let $B_x\twoheadrightarrow L_x$ be the Holmes simple
quotient and set $J_x=\ker(I_x\to L_x)$.  On the inducing fibre, let $h_x$ be
the Holmes head map and let $A_x\subset J_x$ be the $K$-submodule generated
by $\ker h_x$.  Finally put
\[
 H_x^{\rm Hol}:=J_x/A_x.
\]
Holmes' PBW analysis
\cite[Lemmas~2.12--2.14]{HolmesContactSimple}, compared with the integral Rumin
normal forms above, identifies the maximal-vector generators of
$H_x^{\rm Hol}$ with the two Cartier families and, on the plus wall,
the independent Lefschetz family.  Each of those generators is the image of the
corresponding relation in $P_x$.  Hence the induced map
$\rho_x:P_x\to H_x^{\rm Hol}$ is surjective.  Since $P_x\subset J_x$, this
gives $J_x=A_x+P_x$, and the second isomorphism theorem gives
\[
 0\longrightarrow A_x\cap P_x\longrightarrow P_x
 \xrightarrow{\rho_x}H_x^{\rm Hol}\longrightarrow0,
 \qquad P_x/(A_x\cap P_x)\cong H_x^{\rm Hol}.
\tag{5.23c}\label{eq:contact-Holmes-endpoint-descent}
\]
Equation \eqref{eq:contact-Holmes-endpoint-descent} identifies the Rumin
relation quotient with the corresponding Holmes endpoint for all $p\ge5$.

In each Cartier or wall grade $g$ one has
$\Hom_K(A_x,\kk)_g=0$; the additional coefficient line counted by
$\epsilon_r^-$ lies in a separate grade.  Applying $\Hom_K(-,\kk)$ to
\eqref{eq:contact-Holmes-endpoint-descent} shows that a relation functional
$\phi:P_x\to\kk$ descends to the simple endpoint exactly when
\[
 \phi|_{A_x\cap P_x}=0
 \quad\Longleftrightarrow\quad
 \phi\in\rho_x^*\Hom_K(H_x^{\rm Hol},\kk).
\tag{5.23d}\label{eq:contact-Holmes-pullback}
\]
Thus the endpoint multiplicity is
$\dim\Hom_K(H_x^{\rm Hol},\kk)_g$.
On the Cartier cohomology factor, Cartan's formula gives
$\mathcal L_X\alpha=d\iota_X\alpha$ for a closed class.  Hence the restricted
Contact action is trivial, and Cartier kernels contribute no nontrivial
fundamental source.

Consider two distinct nontrivial exceptional vertices.  Equality of
the grading-centre characters is necessary by the degree-zero part of the
restricted derivation equation.  After Levi normalization the equal-centre
grading equation places $d_+(K_1)$ in depth $-1$, so the positive inducing
head is zero.  The fundamental mixed term satisfies the parity/bandwidth restriction
\[
 s-r\in\{-2,0,2\};
\tag{5.24$^0$}\label{eq:contact-mixed-bandwidth}
\]
the cubic inducing layer has no fundamental quotient.  For a distinct same-sign
pair, \eqref{eq:contact-mixed-bandwidth} forces $|s-r|=2$, whereas
$\zeta_s^- -\zeta_r^-=s-r$ and
$\zeta_s^+-\zeta_r^+=-(s-r)$.  Since $p\ge5$, centre equality is impossible.
Thus the same-sign positive-degree contribution is zero.

For a cross-sign pair, in the following calculation $s$ denotes the
fundamental head index of the plus vertex, so that its signed address is
$(+,s+1)$.  Centre equality is then $p\mid 2m+2-r-s$.  The case $s=r$ is the folded identification
$\omega_r^-=\omega_{r+1}^+$ and represents a single vertex.  The remaining two possibilities are
\[
 \mathcal R_+:\ s=r+2,\quad p\mid(m-r),
 \qquad
 \mathcal R_-:\ s=r-2,\quad p\mid(m-r+2).
\tag{5.24$^1$}\label{eq:contact-cross-resonances}
\]
On the first negative inducing layer of a target with central character $\zeta$,
the natural positive relation operator is, up to a fixed unit,
$\mathsf C_{\rm ten}-\zeta$, where $\mathsf C_{\rm ten}$ is the tensor
Casimir.  On the two fundamental Pieri lines its scalars are
\[
 \rho_s^{\uparrow}(\zeta)=-s-\zeta,
 \qquad
 \rho_s^{\downarrow}(\zeta)=-2m+s-2-\zeta.
\tag{5.24}\label{eq:contact-first-return-scalars}
\]
On $\mathcal R_+$ the two ordered orientations give, respectively,
\[
 \rho_s^{\downarrow}(-s)=2(r-m+1)\equiv2,
 \qquad
 \rho_r^{\uparrow}(r-2m-2)=2(m+1-r)\equiv2\pmod p.
\tag{5.24$^2$}\label{eq:contact-raising-unit-return}
\]
The mixed equation \eqref{eq:contact-mixed-map} is therefore a unit multiple
of the unique coordinate in the negative inducing head and forces that
coordinate to vanish.

On $\mathcal R_-$, with $t=r-1=s+1$, the two fundamental relation scalars are
\[
 \rho_s^{\uparrow}(-s)=0,
 \qquad
 \rho_r^{\downarrow}(r-2m-2)=0.
\tag{5.24$^3$}\label{eq:contact-lowering-relation-return}
\]
Thus the corresponding fundamental line is a projected relation in the
first inducing layer and yields no cohomology class.  Write
$\operatorname{res}_{\rm rel}:R_t^{\rm ind}\to R_t^{\rm Rum}$ for the
natural map from the induced fundamental relation module to its projected
Rumin relation module.  When this map is not injective,
$\ker(\operatorname{res}_{\rm rel})$ is a Cartier cohomology subquotient.
It is restricted-Contact-trivial by Cartan's formula, so, since $t\ge1$,
\[
 \Hom_{K_0}\bigl(L(\omega_t),\ker(\operatorname{res}_{\rm rel})\bigr)=0.
\tag{5.24$^4$}\label{eq:contact-relation-kernel-nohom}
\]
Hence the Cartier kernel contributes no class on this fundamental line.
By \cref{lem:contact-mixed-propagation,prop:contact-total-gluing-filtration},
vanishing of this first-layer negative component forces all deeper mixed components to
vanish.  The positive-degree Contact contribution is therefore zero for
every distinct nontrivial
exceptional pair.  The remaining degree-zero term is the classical type-$C$
fundamental extension, one-dimensional exactly on an Adamovich direct cover.
This proves the last line of \eqref{eq:contact-exceptional-closed} for all
$p\ge5$ on the Holmes vertex set.

For the diagonal, a nonfolded exceptional address has zero classical
self-extension and zero same-sign positive-degree contribution.  At a
fold, $\omega_k^-=\omega_{k+1}^+$ and $p\mid(m-k+1)$.  The Rumin self-extension complex
has a relation component and a transverse fundamental component.  For the
relation component,
the restriction to the common fundamental relation section has
kernel contained in Cartier cohomology.  That kernel is restricted-Contact-trivial, hence
has no copy of the nontrivial fundamental self-head.  The restriction is
therefore injective on the fundamental self-extension summand, while its
image is multiplication by the folded scalar $(m-k+1)$.  Since this scalar
is zero at a fold, the relation component contributes no class.

For the transverse component there is no hook when $k=0$.  Assume $k\ge1$ and
write $u_+=q_1\wedge\cdots\wedge q_k$.  Let $c_-$ be the normalized generator
of $K_{-2}$ and take the following element in the component with one central factor:
$\widetilde z_D=c_-q_1\otimes u_+$.  The two-term Contact relation
\[
 R_*=p_1^3\otimes z^2+2(zp_1)\otimes(zp_1^2)
\tag{5.24a$_0$}\label{eq:contact-hook-relation}
\]
is a consequence of
$[p_1^3,z^2]=-2zp_1^3$ and
$[zp_1,zp_1^2]=zp_1^3$.  On the plus fundamental head at the folded vertex
$\zeta=-k$, the bracket identities
$[z^2,c_-]=-4z$, $[z^2,q_1]=-2zq_1$, and $[z,q_1]=-q_1$ give
\[
 \rho_3(z^2)\widetilde z_D
 =4(k+1)(q_1\otimes u_+).
\tag{5.24a$_1$}\label{eq:contact-hook-first-evaluation}
\]
The corresponding cubic hook component is
\[
 \phi_h(p_1^3;q_1\otimes u_+)
 =\pm6\,p_1\wedge q_2\wedge\cdots\wedge q_k,
\]
while the term of $R_*$ arising from the natural summand has zero projection to the hook
component.  Hence
\[
 \mathcal O_{12}^{\rm hook}(R_*)
 =\mp24(k+1)\,p_1\wedge q_2\wedge\cdots\wedge q_k.
\tag{5.24a$_2$}\label{eq:contact-hook-24-obstruction}
\]
Thus the class vanishes when $k\not\equiv-1\pmod p$, since the coefficient
$\mp24(k+1)$ is then nonzero.

It remains to treat $k\equiv-1\pmod p$.  Write
\[
 k_a=(a+1)p-1,\qquad m_a=(a+c+1)p-2.
\]
Apply the exact-image construction of
\cref{prop:contact-curv-kernel,thm:contact-rumin-ses}, followed by the Holmes
endpoint passage \eqref{eq:contact-Holmes-endpoint-descent}, to the Rumin
filtration by the last Cartier class introduced.  Inverse Cartier
identifies the new quotient with the preceding filtered quotient tensored with the new
Cartier multiplicity term, while the new kernel is Cartier cohomology.
The resulting exact sequence is
\[
 0\longrightarrow T_a\longrightarrow\mathscr S_a
 \xrightarrow{\pi_a}
 \mathcal H_a^{(1)}\otimes\mathscr S_{a-1}
 \longrightarrow0,
\tag{5.24a}\label{eq:contact-folded-Cartier-induction}
\]
where $T_a$ and $\mathcal H_a^{(1)}$ have trivial restricted Contact action.
Put
\[
 \mathscr H_a:=\Hom_{K_0}(L(\omega_{k_a}),\mathscr S_a).
\]
Since $L(\omega_{k_a})$ is nontrivial,
$\Hom_{K_0}(L(\omega_{k_a}),T_a)=0$; hence the quotient map gives an
injection
\[
 \bar\pi_a:\mathscr H_a\hookrightarrow
 \mathcal H_a^{(1)}\otimes\mathscr H_{a-1}.
\tag{5.24a$'$}\label{eq:contact-folded-hook-injection}
\]
All arrows in the exact-image construction are $K$-equivariant, and the
mixed map \eqref{eq:contact-mixed-map} is natural in the coefficient module.
Since $\mathcal H_a^{(1)}$ is Contact-trivial, naturality of the mixed
relation map gives
\[
 \tau_a=(\operatorname{id}\otimes\tau_{a-1})\bar\pi_a.
\tag{5.24a$''$}\label{eq:contact-folded-return-naturality}
\]
The base hook coordinate is one-dimensional, and the normalized mixed
relation map is $\tau_0=-2\,\operatorname{id}$.  Iterating
\eqref{eq:contact-folded-hook-injection}--\eqref{eq:contact-folded-return-naturality}
therefore gives $\tau_a=-2\Pi_a$, where $\Pi_a$ is injective.  Because
$p\ge5$, $-2$ is a unit.  The mixed cocycle equation requires $\tau_a=0$, so
no nonzero resonant hook coordinate extends to a cocycle.  Thus neither the
nonresonant nor the resonant case contributes, proving
\[
 \Ext^1_{u_0(K)}(L,L)=0
 \quad\text{for every exceptional Contact simple }L.
\tag{5.24b}\label{eq:contact-exceptional-selfzero-proof}
\]

It remains to compute the trivial row and column.  For
$a\in\mathcal A_r^-(m,p)$, the minus Holmes endpoint of primitive degree
$a$ has maximal-vector support consisting of
the two Cartier contributions.  Their trivial quotients have
dimensions
\[
 \binom{2m}{a},\qquad \binom{2m}{a-1},
\]
so \eqref{eq:contact-Holmes-pullback} gives the minus dimension
$\binom{2m+1}{a}$.  On a Rumin wall the boundary relation module has
one further wall quotient.  For the minus Holmes endpoint its pure wall
coordinate lies in $A_x\cap P_x$ and therefore vanishes in
\eqref{eq:contact-Holmes-pullback}.  For the adjacent plus Holmes head, if
$p\mid(m-a)$, the two Cartier contributions are accompanied by the
independent integral Lefschetz-wall contribution of dimension $\binom{2m}{a}$.
Lemma~2.13 of \cite{HolmesContactSimple} contributes only on the plus branch,
where it gives this Lefschetz correction; the periodic representatives from
Lemma~2.12 are Cartier classes.  These give the multiplicities of the
trivial module in the Holmes heads.

Summing over the degree-separated realizations
$a\in\mathcal A_r^-$ gives \eqref{eq:contact-Dminus} and the binomial part
of \eqref{eq:contact-Rminus}.  The contribution from the additional trivial
quotient is $\epsilon_r^-$, while the bottom reverse orientation contributes
$\delta_{r,0}$.  Rigid form duality
$\mathscr E_j^+\cong(\mathscr E_{j-2}^-)^\vee$ interchanges the two endpoint
numbers, giving \eqref{eq:contact-TU}.  This proves the two lines of
\eqref{eq:contact-exceptional-closed} involving the trivial module.  Perfectness of the derived
Contact algebra gives the trivial loop.
\end{proof}

\begin{corollary}[Exceptional subquivers]
\label{cor:four-family-exceptional-subquivers}
For each family $F\in\{W,S,H,K\}$, every nontrivial exceptional loop is zero
and every arrow between two nontrivial exceptional vertices has multiplicity
at most one.  All exceptional multiplicities greater than one are incident with the trivial
vertex.

On the nontrivial exceptional vertices, the $W$- and $S$-subquivers have no
arrows.  For $H$ the subquiver is the Adamovich direct-cover graph of
\cref{thm:H-exceptional-quiver}; for $K$ it is the pullback of this graph to
the folded signed Holmes vertices, subject to equality of the grading centre,
as in the last line of \eqref{eq:contact-exceptional-closed}.  Extensions
incident with the trivial module are given by
\eqref{eq:W-exceptional-closed},
\eqref{eq:S-exceptional-out}--\eqref{eq:S-exceptional-in},
\cref{thm:H-trivial-exceptional}, and
\eqref{eq:contact-Dminus}--\eqref{eq:contact-TU}, respectively.
\end{corollary}

\begin{proof}
For $W$ and $S$ the nontrivial interior matrices vanish.  For $H$ and $K$
the interior entry is a direct-cover indicator and is therefore $0$ or $1$.
The endpoint formulas give the remaining binomial and wall multiplicities.
\end{proof}

\begin{corollary}[Self-extensions]
\label{cor:self-extensions}
Let $p\ge5$ and let $\mathfrak g$ be one of the four families.
\begin{enumerate}[label=\textup{(\roman*)}]
\item The trivial module has no self-extension:
      $\Ext^1_{u_0(\mathfrak g)}(\kk,\kk)=0$.
\item No nontrivial exceptional simple module has a self-extension.
\item For a parabolically induced simple source the internal degree-zero summand of the
      graded parabolic splitting is the Levi self-extension group, and it
      vanishes by the Frobenius-kernel no-self-extension theorem.  In the
      Hamiltonian case \eqref{eq:H-standard-graded-split} therefore reduces
      to
\[
 \Ext^1_{u_0(\mathfrak h)}(Z_\lambda,Z_\lambda)
 \cong
 \bigoplus_{q\ge1}
 \Hom_{H_0}\!\left(U_\lambda,H^1_{*,q}(N,Z_\lambda)\right),
\]
      and the corresponding reduction holds in the other three families for
      their own parabolic complexes.  A parabolically induced self-extension can
      therefore only be created in positive internal degree.
\end{enumerate}
\end{corollary}

\begin{proof}
Part (i) is the trivial loop, which vanishes because each of the four
algebras is perfect.  Part (ii) is the loop statement of
\cref{cor:four-family-exceptional-subquivers}.  For (iii), take
$\lambda=\mu$ in \eqref{eq:H-standard-graded-split} and apply
\cite[Theorem~4.5]{AndersenExtensions} to the degree-zero summand
$\Ext^1_{u_0(H_0)}(U_\lambda,U_\lambda)$.
\end{proof}

\begin{remark}
Parts (i)--(ii) answer, for all four families and all ranks, the question of
self-extensions among the trivial and exceptional labels; part (iii)
localises the only remaining possibility to the positive nilradical
cohomology of a parabolically induced module.  In the example of
\cref{sec:checks-consequences} that residual term also vanishes, and the
quiver of $\Spec(3;\mathbf1)^{(1)}$ at $p=7$ is loop-free.
\end{remark}

\section{Specializations and representation-theoretic consequences}
\label{sec:checks-consequences}

\begin{corollary}[Exceptional first radical layer]
\label{cor:exceptional-first-radical-layer}
Let $B_F$ be a basic algebra Morita equivalent to a block in one of the
four families, let $J=\Jac(B_F)$, and let $\mathcal X_F$ index the
exceptional simple modules $L_x$ lying in that block.  Put
$e=\sum_{x\in\mathcal X_F}e_x$.  Then
\[
 e(J/J^2)e
 \cong
 \bigoplus_{x,y\in\mathcal X_F}
 \kk_{yx}^{\oplus\dim\Ext^1_{B_F}(L_x,L_y)},
\tag{6.0a}\label{eq:exceptional-Peirce}
\]
where $\kk_{yx}$ denotes the simple $(B_F/J)$-bimodule supported at
$(e_y,e_x)$.  Equivalently,
\[
 e_y(J/J^2)e_x\cong\Ext^1_{B_F}(L_x,L_y)^*.
\tag{6.0b}\label{eq:Ext-Peirce-orientation}
\]
Thus the four exceptional matrices determine the semisimple Peirce
sub-bimodule of $J/J^2$ supported on the exceptional idempotents.
If the block contains the trivial vertex, this exceptional corner is
preserved by every Morita equivalence that preserves the trivial module and
the set of exceptional simple modules.
\end{corollary}

\begin{proof}
For a split basic algebra, the canonical identification of arrows with the
first radical layer gives
$\Ext^1_{B_F}(L_x,L_y)^*\cong e_y(J/J^2)e_x$.  Since $B_F/J$ is a product of copies of
$\kk$, the bimodule $J/J^2$ is semisimple over $(B_F/J)^e$, so its Peirce
multiplicities determine the displayed exceptional sub-bimodule.  A Morita
equivalence preserving the trivial and exceptional simple modules identifies
the corresponding Peirce components and hence the exceptional corner.
\end{proof}

\begin{proposition}[The $W(1;1)$ quiver]
\label{prop:W1-benchmark}
Let $\mathfrak g=W(1;1)$ and label its restricted simples by
$L(0),\ldots,L(p-1)$ in the usual convention.  The specialization of
\cref{thm:SW-pair} is as follows.

For $p=5$, the extension space
$\Ext^1_{u_0(\mathfrak g)}(L(\mu),L(\lambda))$ is one-dimensional precisely
when
\begin{enumerate}[label=\textup{(\alph*)}]
\item $\lambda-\mu\equiv2,3\pmod p$,
      $1\le\mu\le p-2$, $1\le\lambda\le p-1$; or
\item $(\mu,\lambda)$ is one of
$(0,1),(p-2,0),(p-1,2),(p-1,3)$.
\end{enumerate}
It is two-dimensional when $\{\mu,\lambda\}=\{0,p-1\}$ and is zero
otherwise.

For $p\ge7$, replace the first congruence by
$\lambda-\mu\equiv2,3,4\pmod p$, add $(p-1,4)$ to the four exceptional
pairs above, and add the one-dimensional cases
\[
 2\lambda^2-10\lambda+3\equiv0\pmod p,
 \quad \lambda-\mu\equiv6\pmod p,
 \quad 1\le\mu,\lambda\le p-2.
\tag{6.2}\label{eq:W1-quadratic-benchmark}
\]
The two-dimensional and zero cases remain as stated.
\end{proposition}

\begin{proof}
Here $P=W_{\ge0}=\mathfrak b^+$ and its torus-zero restricted pair complex
is
\[
 H_*^1(\mathfrak n^+,Z^+(\lambda))_\mu
 =\frac{\operatorname{Der}_*(\mathfrak n^+,Z^+(\lambda))_\mu}
        {\operatorname{Inn}(\mathfrak n^+,Z^+(\lambda))_\mu}.
\]
In the Witt basis $e_{-1},e_0,\ldots,e_{p-2}$, substituting
$[e_i,e_j]=(j-i)e_{i+j}$ into
\eqref{eq:NN}--\eqref{eq:Np} expresses a cocycle by its values on
$e_1,e_2$; the remaining Hall rows successively determine its values on
$e_3,\ldots,e_{p-2}$.  Row elimination gives respectively the differences
$2,3$ for $p=5$, the differences $2,3,4$ for $p\ge7$, and the residual
determinant $2\lambda^2-10\lambda+3$ at difference $6$.
For $1\le\lambda\le p-2$ one has $Z^+(\lambda)=L(\lambda)$.  The two
remaining parabolically induced modules lie in the distinct length-two sequences
\[
 0\to L(p-1)\to Z^+(0)\to L(0)\to0,
 \qquad
 0\to L(0)\to Z^+(p-1)\to L(p-1)\to0.
\tag{6.2$'$}\label{eq:W1-endpoint-standards}
\]
Moreover $H_*^1(\mathfrak n^+,\kk)$ is spanned by $e_1^*,e_2^*$.
Insert \eqref{eq:W1-endpoint-standards} in the two long exact Ext sequences and
use $L(\lambda)^*\cong L(p-1-\lambda)$ away from the two self-dual
modules.  This gives the exceptional ordered pairs above; at
$\{0,p-1\}$ the two boundary lines are independent, giving dimension two.
These agree with the simple-module quivers of
Boe--Nakano--Wiesner \cite[\S3, Theorems~A--B]{BoeNakanoWiesner}.
In particular, the congruence differences $2,3$ generate all of
$\mathbb Z/p\mathbb Z$ since $\gcd(2,3)=1$, and the exceptional pairs listed above
connect the two endpoint parabolically induced modules to the interior; hence the quiver is
connected and $u_0(W(1;1))$ has a unique block.  No pair with $\mu=\lambda$
occurs among the arrows or the exceptional pairs, so
$\Ext^1_{u_0(W(1;1))}(L(\lambda),L(\lambda))=0$ for every $\lambda$.
Thus \cref{thm:SW-pair} also recovers Holmes--Nakano's one-block theorem
and the Lin--Nakano self-extension vanishing for $W(1;1)$.
\end{proof}

\subsection{Exceptional subquivers in ranks one and two}
\label{subsec:small-rank-exceptional}

This subsection instantiates the closed formulas of
\cref{thm:SW-exceptional-closed,thm:H-exceptional-quiver,thm:H-trivial-exceptional,thm:contact-exceptional-closed}
on every restricted simple Cartan-type algebra of rank one or two.  All of
these exceptional subquivers are stars: outside the trivial vertex the exceptional subquiver is empty in
each case, so the whole exceptional part of the $\Ext^1$-quiver is carried
by the two rows through $\kk$.

In rank two the Special algebra is
$\Spec(2;\mathbf1)^{(2)}=\Ham(2;\mathbf1)^{(2)}$ by
\eqref{eq:low-rank-SH-identification}; it is listed in (c) under its
Hamiltonian name.  The full quiver of $\Spec(3;\mathbf1)^{(1)}$ at $p=7$ is
computed in \cref{subsec:S3-p7}.

\begin{proposition}[Rank one and two exceptional subquivers]
\label{prop:small-rank-exceptional}
Let $p\ge5$.  In each of the following cases every loop vanishes and every
$\Ext^1$-group between two nontrivial exceptional simples vanishes, so the
exceptional subquiver is the star listed below.  All entries are
independent of $p$; in particular they agree at $p=5$ and $p=7$.
\begin{enumerate}[label=\textup{(\alph*)}]
\item $\Witt(1;\mathbf1)$, with exceptional vertices $\kk$ and
      $M_0=A/\kk$:
      \[
       \dim\Ext^1(\kk,M_0)=\dim\Ext^1(M_0,\kk)=2 .
      \]
\item $\Witt(2;\mathbf1)$, with exceptional vertices $\kk$, $M_0=A/\kk$
      and $M_1=B^2$:
      \[
       \begin{array}{c|cc}
        & \to M_0 & \to M_1\\\hline
        \kk\to & 3 & 1
       \end{array}
       \qquad
       \begin{array}{c|c}
        & \to\kk\\\hline
        M_0\to & 1\\
        M_1\to & 3
       \end{array}
      \]
\item $\Ham(2;\mathbf1)^{(2)}$, with exceptional vertices $\kk$ and the
      adjoint module $M_1$:
      \[
       \dim\Ext^1(\kk,M_1)=\dim\Ext^1(M_1,\kk)=4 .
      \]
\item $\Ham(4;\mathbf1)^{(2)}$, with exceptional vertices $\kk$, $M_1$
      and $M_2$:
      \[
       \dim\Ext^1(\kk,M_1)=\dim\Ext^1(M_1,\kk)=6,
       \qquad
       \dim\Ext^1(\kk,M_2)=\dim\Ext^1(M_2,\kk)=10 .
      \]
\item $\Cont(3;\mathbf1)^{(1)}$, with exceptional vertices
      $\kk=\omega_1^+$, $\omega_0^-$, $\omega_1^-$ and $\omega_2^+$:
      \[
       \begin{array}{c|ccc}
        & \to\omega_0^- & \to\omega_1^- & \to\omega_2^+\\\hline
        \kk\to & 4 & 3 & 3\\
        \to\kk & 1 & 3 & 3
       \end{array}
      \]
      where the second row lists $\dim\Ext^1(x,\kk)$.
\item $\Cont(5;\mathbf1)^{(1)}$, with exceptional vertices
      $\kk=\omega_1^+$, $\omega_0^-$, $\omega_1^-$, $\omega_2^-$,
      $\omega_2^+$ and $\omega_3^+$:
      \[
       \begin{array}{c|ccccc}
        & \to\omega_0^- & \to\omega_1^- & \to\omega_2^-
        & \to\omega_2^+ & \to\omega_3^+\\\hline
        \kk\to & 6 & 10 & 10 & 5 & 10\\
        \to\kk & 1 & 5 & 10 & 10 & 10
       \end{array}
      \]
\end{enumerate}
\end{proposition}

\begin{proof}
Throughout, the vanishing of the loops is
\cref{cor:self-extensions}(i)--(ii), and the vanishing of the arrows
between nontrivial exceptional vertices is verified case by case below.

(a)--(b) Substitute $n=1$ and $n=2$ into
\eqref{eq:W-exceptional-closed}.  For $n=1$ the exceptional indexing is
$M_1=\kk$ and $M_0=A/\kk$, and the two endpoint entries are
$\binom10+\delta_{0,0}=2$ and $\binom11+\delta_{0,0}=2$, while the entry
with $a=b=0$ is zero.  This is consistent with the two-dimensional case
$\{\mu,\lambda\}=\{0,p-1\}$ of \cref{prop:W1-benchmark}, since
$A/\kk\cong L(p-1)$.
For $n=2$ the first line of \eqref{eq:W-exceptional-closed} gives
$\Ext^1(M_a,M_b)=0$ for $a,b\in\{0,1\}$, and the endpoint lines give
$\binom20+\delta_{0,1}=1$, $\binom21+\delta_{1,1}=3$ into $\kk$ and
$\binom21+\delta_{0,0}=3$, $\binom22+\delta_{1,0}=1$ out of $\kk$.

(c)--(d) Here $m=1$ and $m=2$.  In both cases
\eqref{eq:H-exceptional-quiver} is zero on every ordered pair of distinct
nontrivial labels: for $m=1$ there is only the label $r=1$, and for $m=2$
the two labels $1,2$ have opposite parity, so the same-parity condition
fails.  For $m=2$ the endpoint values are
$2m+2=6$ at $s=1$ by \cref{prop:H-trivial-adjoint}, and at $s=2$
\eqref{eq:H-trivial-exceptional} gives
\[
 \binom{2m+1}{2}+\mathbf1_{\{p\mid m\}}\binom{2m}{1}+2\epsilon_2
 =\binom52+0+0=10,
\]
because $p\ge5$ does not divide $m-s+2=2$, and
$\epsilon_2=\mathbf1_{\{2\mid2\}}\mathbf1_{\{p\mid(1)(2)\}}=0$.
For $m=1$ the endpoint value is $2m+2=4$, again by
\cref{prop:H-trivial-adjoint}.

(e)--(f) Here $N=2m+1$ equals $3$ and $5$.  In both cases
$\operatorname{Cov}_{m,p}(r)=\varnothing$ for every $r$: a cover requires
$a\equiv r\pmod 2$ with $0\le a<r\le m$, so for $m=1$ there is no candidate
at all, and for $m=2$ the only candidate is $a=0$, $r=2$, which by the
direct-cover criterion of \cref{thm:H-trivial-exceptional} would require
$p\mid(s/2)(m-s/2+1)=2$.  Consequently the last line of
\eqref{eq:contact-exceptional-closed} vanishes identically, which is the
asserted absence of arrows between nontrivial exceptional vertices.

With $\operatorname{Cov}_{m,p}(r)$ empty, the binomial condition in
\eqref{eq:contact-Ringel-endpoint} is vacuous, so
$\eta_{r,k}^{m,p}=1$ exactly when $j=(k-r)/2$ is an integer with
$0\le j\le n_r=m-r$ and either $j=0$ or $p\mid n_r-j+1$.  For $m\le2$ and
$p\ge5$ the second alternative never occurs, so $j=0$, that is $k=r$.
Since $|a-r|\le m<2p$, the congruence $a=r+2tp$ in
\eqref{eq:contact-realization-set} forces $a=r$, and therefore
\[
 \mathcal A_r^-=\{r\}\qquad(0\le r\le m).
\]
Next, $\epsilon_r^-=0$ in all these cases, since
$r-2m-2\in\{-2m-2,\ldots,-m-2\}$ is a nonzero residue for $p\ge5$ and
$m\le2$.  Also $\widehat W_r^-=0$: for $r<m$ the factor
$\mathbf1_{\{p\mid(m-r)\}}$ vanishes because $0<m-r\le2<p$, and for $r=m$
the summation range $a\in\mathcal A_r^-$, $a<m$ is empty.
Hence \eqref{eq:contact-Dminus} and \eqref{eq:contact-Rminus} reduce to
\[
 D_r^-=\binom Nr,
 \qquad
 R_r^-=\binom N{r+1}+\delta_{r,0}.
\]
For $N=3$ this gives $(D_0^-,D_1^-)=(1,3)$ and $(R_0^-,R_1^-)=(4,3)$;
for $N=5$ it gives $(D_0^-,D_1^-,D_2^-)=(1,5,10)$ and
$(R_0^-,R_1^-,R_2^-)=(6,10,10)$.

It remains to identify the vertices.  By
\eqref{eq:contact-exceptional-labels}, $\omega_k^-=(\omega_k,k-2m-2)$ and
$\omega_k^+=(\omega_{k-1},1-k)$, so $\omega_1^+=(0,0)$ is the trivial
module.  A minus address $\omega_k^-$ and a plus address $\omega_l^+$ can
be folded together only if $\omega_k=\omega_{l-1}$, that is $l=k+1$, and
then the central characters differ by
$(k-2m-2)-(-k)=2k-2m-2$; for $m\le2$ and $p\ge5$ this is a nonzero residue
for every $k$ in range.  Hence no folding occurs, and the vertex sets are
as listed.  Substituting the values of $D_r^-$ and $R_r^-$ into
\eqref{eq:contact-TU} and then into the second and third lines of
\eqref{eq:contact-exceptional-closed} gives the two tables.
\end{proof}

\begin{remark}
\label{rem:small-rank-p-independence}
Every entry in \cref{prop:small-rank-exceptional} is independent of $p$ in
the range $p\ge5$.  The same is true of the Special endpoint formulas
\eqref{eq:S-exceptional-out}--\eqref{eq:S-exceptional-in}.  Thus the
$p$-dependence visible in the example of
\cref{subsec:S3-p7} lies entirely in the nonexceptional part of the
quiver: the resonance conditions
$\mathbf1_{\{p\mid(m-s+2)\}}$, $\mathbf1_{\{p\mid(m-r)\}}$ and
$\epsilon_r^-$ that carry the characteristic in
\eqref{eq:H-trivial-exceptional} and \eqref{eq:contact-Rminus} first become
active in rank three.
\end{remark}

\subsection{The Special quiver for \texorpdfstring{$\Spec(3;\mathbf1)^{(1)}$}{S(3;1)(1)} in characteristic seven}
\label{subsec:S3-p7}

Assume in this subsection that
\[
 \mathfrak s=\Spec(3;\mathbf1)^{(1)},\qquad p=7,
 \qquad S_0\cong\mathfrak{sl}_3,
\]
and write \(Z(a,b)\) for the simple parabolically induced module from the restricted
\(S_0\)-module \(L(a,b)\).  The three exceptional vertices are
\(L_0,L_1,L_2\) above.  Among the forty-nine restricted \(S_0\)-weights,
\((0,0),(0,1),(1,0)\) are the three exceptional inducing addresses; the
other forty-six give the induced vertices.

The maximal toral rank of $\mathfrak s$ is two
\cite[Corollary~7.5.6(2)]{StradeI}; hence the restricted Levi labels give the
$7^2=49$ vertices above.

\begin{proposition}[Extensions from the trivial Special module]
\label{prop:S3-trivial-exceptional-row}
Let
\[
 \mathfrak s=\Spec(3;\mathbf1)^{(1)}
\]
over an algebraically closed field of characteristic seven, and write
\[
 L_0=\kk,\qquad L_1=A^\circ/\kk,\qquad L_2=B^2.
\]
Then
\[
\dim\Ext^1_{u_0(\mathfrak s)}(L_0,L_i)
=\dim\Ext^1_{u_0(\mathfrak s)}(L_i,L_0)=4
\qquad(i=1,2).
\tag{6.5a}
\]
\end{proposition}

\begin{proof}
Put $V=\kk^3$, $\Omega^q=A(3;\mathbf1)\otimes\Lambda^qV^*$,
$B^q=\im(d_{q-1})$, and $Q^q=\Omega^q/B^q$.  The mixed-product cohomology vanishes:
\[
 H_*^1(\mathfrak s,\Omega^q)=0\qquad(q=1,2).
\tag{6.5b}
\]
Indeed, neither $\Omega^q$ nor $\Lambda^qV^*$ has invariant vectors:
translations force an invariant form to have constant coefficients, and
$(\Lambda^qV^*)^{\mathfrak{sl}_3}=0$.  Hence the Hochschild comparison
identifies restricted and ordinary first cohomology for these coefficients.
The degree-one Special layer is the $15$-dimensional simple
$\mathfrak{sl}_3$-module of highest weight
$\omega_1+2\omega_2$.  Indeed
\(\langle\omega_1+2\omega_2+\rho,\theta^\vee\rangle=5<7\);
the strong linkage principle leaves no smaller dominant weight in its
affine orbit, so the corresponding Weyl module is simple.
Since $\Lambda^1V^*$ and $\Lambda^2V^*$ have highest weights
$\omega_2$ and $\omega_1$, respectively,
\[
 \Hom_{\mathfrak{sl}_3}(\mathfrak s_1,\Lambda^qV^*)=0
 \qquad(q=1,2).
\tag{6.5b$'$}\label{eq:S3-mixed-Hom-zero}
\]
The mixed-product calculation of \cite[Lemmas~2.1--2.2]{ChiuSen} gives
\[
 H^1(\mathfrak s,\Omega^q)
 \cong H^1(\mathfrak{sl}_3,\Lambda^qV^*).
\]
The latter group vanishes already in the restricted category: split a
restricted extension over the diagonal torus and choose a torus-fixed lift
of the trivial quotient.  A root vector $e_{ab}$ acting on that lift would
have weight $\varepsilon_a-\varepsilon_b$, whereas the weights of
$\Lambda^qV^*$ are $-\sum_{i\in I}\varepsilon_i$ with $|I|=q$; these
cannot agree on the trace-zero diagonal algebra in characteristic seven.  Thus every
root vector annihilates the lift, so the extension splits.  This proves (6.5b).

Cartan's identity $\mathcal L_X=d\iota_X+\iota_Xd$ shows that
$\mathfrak s$ acts trivially on truncated de Rham cohomology.  Therefore
\[
0\longrightarrow H^q_{\rm dR}(A)
\longrightarrow Q^q\xrightarrow{d}B^{q+1}\longrightarrow0,
\qquad
\dim H^q_{\rm dR}(A)=\binom3q,
\tag{6.5c}
\]
has trivial left term.  Since $(\Omega^q)^{\mathfrak s}=0$ and (6.5b)
holds, the sequence
$0\to B^q\to\Omega^q\to Q^q\to0$ gives
\[
 H_*^1(\mathfrak s,B^q)\cong(Q^q)^{\mathfrak s}
 \qquad(q=1,2).
\tag{6.5d}
\]

For $q=1$, $B^2=L_2$ has no invariant vectors, so (6.5c) gives
$(Q^1)^{\mathfrak s}\cong H^1_{\rm dR}(A)\cong\kk^3$.  Hence
$H_*^1(\mathfrak s,B^1)\cong\kk^3$.  In rank three, $B^1$ fits into the
nonsplit exact sequence
\[
0\longrightarrow L_1\longrightarrow B^1=A/\kk
 \xrightarrow{\int}L_0\longrightarrow0.
\tag{6.5e}
\]
Moreover $(B^1)^{\mathfrak s}=0$: translation invariance modulo constants
reduces a class to its linear part, and the linear $\mathfrak{sl}_3$-module
has no invariant vector.  Since $\mathfrak s$ is perfect,
$H_*^1(\mathfrak s,\kk)=0$.  The long exact sequence of (6.5e) therefore
contains
\[
0\longrightarrow\kk\longrightarrow H_*^1(\mathfrak s,L_1)
 \longrightarrow\kk^3\longrightarrow0,
\]
so $\dim\Ext^1(L_0,L_1)=4$.

For $q=2$, $B^2=L_2$.  The invariant space of
$B^3=A^\circ\,dx_1\wedge dx_2\wedge dx_3$ is the line spanned by the
volume form.  This invariant lifts to $Q^2$: with
$\lambda=x_1\,dx_2\wedge dx_3$ one has $d\lambda=\mathrm{vol}$ and, for
$X\in\mathfrak s$,
\[
 \mathcal L_X\lambda=d(\iota_X\lambda)+\iota_X\mathrm{vol}\in B^2,
\]
because contraction with the volume form identifies the derived Special
algebra with exact two-forms.  Thus the invariant sequence induced from
(6.5c) is exact,
\[
0\longrightarrow H^2_{\rm dR}(A)\cong\kk^3
\longrightarrow(Q^2)^{\mathfrak s}
\longrightarrow(B^3)^{\mathfrak s}\cong\kk
\longrightarrow0.
\]
By (6.5d), $\dim H_*^1(\mathfrak s,L_2)=4$, and hence
$\dim\Ext^1(L_0,L_2)=4$.

The apolar de Rham pairings identify $L_1^*\cong L_1$ and
$L_2^*\cong L_2$.  Exact contravariant duality therefore gives the two
reverse dimensions in (6.5a).
\end{proof}

\begin{proposition}[Degree-zero Special extensions]
\label{prop:S3-p7-q0}
The degree-zero induced--induced arrow space is
\[
 \Ext^1_{u_0(\mathfrak{sl}_3)}(L(\lambda),L(\mu)).
\tag{6.S1}\label{eq:S3-q0-Levi}
\]
In each regular block, an entry in the following matrix is the dimension in
both orientations; every unlisted pair is zero:
\[
\begin{array}{c|c|c|c}
&\text{row weights}&\text{column weights}&\text{multiplicity matrix}\\ \hline
R_0&(0,4),(4,0)&(5,5),(1,5),(5,1)&
 \begin{psmallmatrix}3&1&3\\3&3&1\end{psmallmatrix}\\[2pt]
R_1&(1,3),(3,0)&(4,5),(2,4),(5,2)&
 \begin{psmallmatrix}3&1&3\\3&3&1\end{psmallmatrix}\\[2pt]
R_2&(0,2),(2,0),(2,2)&(3,5),(5,3),(3,3)&
 \begin{psmallmatrix}1&3&3\\3&1&3\\3&3&1\end{psmallmatrix}\\[2pt]
R_3&(0,3),(3,1)&(2,5),(5,4),(4,2)&
 \begin{psmallmatrix}1&3&3\\3&3&1\end{psmallmatrix}\\[2pt]
R_4&(1,1),(1,2),(2,1)&(4,4),(3,4),(4,3)&
 \begin{psmallmatrix}1&3&3\\3&1&3\\3&3&1\end{psmallmatrix}
\end{array}.
\tag{6.S2}\label{eq:S3-q0-regular}
\]
For \(0\le a\le5\), put
\[
 c_a=(a,5-a),\qquad u_a=(6,a),\qquad v_a=(5-a,6).
\]
The arrows in the singular blocks are
\[
 c_a\overset{3}{\longleftrightarrow}u_a,
 \qquad c_a\overset{3}{\longleftrightarrow}v_a,
\tag{6.S3}\label{eq:S3-q0-singular}
\]
and the Steinberg vertex \((6,6)\) is isolated in degree zero.  Hence the
degree-zero part has \(96\) nonzero ordered pairs, and the sum of their
dimensions is \(240\).
\end{proposition}

\begin{proof}
The degree-zero parabolic strand is exactly the Levi strand, giving
\eqref{eq:S3-q0-Levi}. Here is the explicit specialization of
\cite[Theorem~5.13]{ParkerGL3}, with $p=7$.
For every bottom-alcove weight $\alpha=(r,s)$, $r,s\ge0$, $r+s\le4$, its only nonzero degree-zero arrow dimensions are
\[
\begin{array}{c|ccc}
\beta&(5-s,5-r)&(r+s+1,5-s)&(5-r,r+s+1)\\ \hline
\dim\operatorname{Ext}^1_{G_1}(L(\alpha),L(\beta))&1&3&3.
\end{array}
\]
They come respectively from the trivial module, the natural Frobenius twist, and its dual in Parker's formula. Contravariant duality gives the same dimensions in the reverse orientation. These formulas exhaust the regular pairs; in particular there are no same-side pairs or loops. Substituting the row weights of \eqref{eq:S3-q0-regular} gives its five matrices. For example, $(r,s)=(0,4)$ yields targets $(1,5),(5,1),(5,5)$ with dimensions $1,3,3$, which is the first row of $R_0$ after the displayed column ordering.

For $r+s=5$, Parker's singular formula gives exactly the two targets $(6,r)$ and $(s,6)$, with dimension three in either orientation, and no edge between these latter two labels. This proves \eqref{eq:S3-q0-singular}. The Steinberg simple is projective for $G_1$ and has no degree-zero edge.

There are fifteen bottom-alcove weights; deleting $(0,0),(1,0),(0,1)$ leaves twelve, each with three distinct upper targets. Thus the regular part has $2\cdot12\cdot3=72$ ordered pairs and total dimension $2\cdot12(1+3+3)=168$. The six singular triples give $6\cdot4=24$ ordered pairs of dimension three. The totals are therefore $96$ pairs and dimension $168+72=240$.
\end{proof}

For a parabolically induced target \(Z(U)\), put
\[
 Q_1(U):=\operatorname{coker}\!\left[
 V\otimes U\xrightarrow{\Delta_U}
 S_1^*\otimes U\right],
 \qquad
 \Delta_U(v\otimes u)(x)=\rho_U([x,v])u,
\tag{6.S4}\label{eq:S3-first-jet}
\]
where \(S_1\cong L(1,2)\) is the first positive layer.  In degree
two let \(J_U\) be the inner map and \(\Theta_U\) the curvature in
\[
 L(2,0)\otimes U\xrightarrow{J_U}
 (L(3,1)\oplus L(1,2)\oplus L(2,0))\otimes U
 \xrightarrow{\Theta_U}
 (L(5,0)\oplus L(2,3)\oplus L(1,2)\oplus L(0,1))\otimes U.
\tag{6.S5}\label{eq:S3-second-jet}
\]
On pure tensors,
\[
 J_U(v_av_b\otimes u)(x)=
 v_b\otimes\rho_U([x,v_a])u+v_a\otimes\rho_U([x,v_b])u
 +[[x,v_a],v_b]\otimes u,
\tag{6.S6}\label{eq:S3-second-inner}
\]
and \(\Theta_U f(x,y)=x f(y)-y f(x)-f([x,y])\).

\begin{theorem}[The Special quiver in characteristic seven]
\label{thm:S3-p7-quiver}
The positive-degree induced--induced arrows are as follows.

\begin{enumerate}[label=\textup{(\alph*)}]
\item At \(q=1\),
\[
 \Ext^1_q(Z(\lambda),Z(\mu))
 \cong\Hom_{u_0(\mathfrak{sl}_3)}
       (L(\lambda),\operatorname{soc}Q_1(L(\mu))).
\tag{6.S7}\label{eq:S3-q1-law}
\]
There are \(241\) nonzero ordered pairs, all of multiplicity one.

\item At \(q=2\), the only nonzero pairs are the following thirteen
multiplicity-one arrows:
\[
\begin{aligned}
&(0,2)\to(0,4),\quad (4,0)\to(2,0),\quad (5,0)\to(3,0),\quad
 (6,0)\to(4,0),\\
&(0,5)\to(1,5),\quad (0,6)\to(2,4),\quad (1,5)\to(3,3),\quad
 (2,4)\to(4,2),\\
&(3,3)\to(5,1),\quad (0,3)\to(0,5),\quad (5,1)\to(5,0),\\
&(0,4)\to(0,6),\quad (4,2)\to(6,0).
\end{aligned}
\tag{6.S8}\label{eq:S3-q2-list}
\]
There are no induced--induced arrows for \(q\ge3\).  Thus the
positive-degree induced--induced part has \(254\) nonzero ordered pairs,
all of multiplicity one.
\end{enumerate}

The nonzero induced--exceptional pairs are
\[
\begin{array}{c|c|c}
\text{orientation}&\text{inducing label}&(q,\dim)\\ \hline
Z\to L_0&(2,1)&(1,1)\\
Z\to L_0&(5,5)&(0,1)\\
Z\to L_0&(1,5),(5,1)&(0,3),(0,3)\\
L_0\to Z&(1,2)&(19,1)\\
L_0\to Z&(5,5)&(0,1)\\
L_0\to Z&(1,5),(5,1)&(0,3),(0,3)
\end{array}.
\tag{6.S9}\label{eq:S3-standard-exceptional}
\]
\[
\begin{array}{c|c|c}
\text{orientation}&\text{inducing label}&(q,\dim)\\ \hline
Z\to L_1&(1,3),(2,1),(5,4)&(1,1),(1,1),(0,1)\\
Z\to L_1&(4,2),(2,5)&(0,3),(0,3)\\
L_1\to Z&(3,1),(1,2),(5,4)&(3,1),(3,1),(0,1)\\
L_1\to Z&(4,2),(2,5)&(0,3),(0,3)
\end{array}
\]
and
\[
\begin{array}{c|c|c}
\text{orientation}&\text{inducing label}&(q,\dim)\\ \hline
Z\to L_2&(4,5),(2,4),(5,2)&(0,1),(0,3),(0,3)\\
Z\to L_2&(3,0),(2,2),(3,1),(1,2)&(1,1),(1,1),(7,1),(7,1)\\
L_2\to Z&(1,3),(2,1),(4,5)&(1,1),(1,1),(0,1)\\
L_2\to Z&(2,4),(5,2),(0,3),(2,2)&(0,3),(0,3),(2,1),(2,1)
\end{array}.
\]
All entries omitted from this table vanish.  It contains \(32\) nonzero
ordered pairs, and the sum of their dimensions is \(56\).

Among the exceptional vertices, the four ordered pairs
\[
 L_0\rightleftarrows L_1,\qquad L_0\rightleftarrows L_2
\]
have dimension \(4\); all five remaining ordered pairs, including the
three loops, vanish.  The quiver therefore has
\[
 49\text{ vertices},\qquad386\text{ nonzero ordered pairs},
 \qquad566\text{ arrows counted with multiplicity}.
\tag{6.S10}\label{eq:S3-global-count}
\]
It is loop-free.
\end{theorem}

\begin{proof}
At \(q=1\), restricted Shapiro and the invariant-top lemma give
\eqref{eq:S3-q1-law}.  The required block count is as follows.  Put
\[
 A_{a,b}=L(1,0)\otimes L(a,b),\qquad
 B_{a,b}=L(2,1)\otimes L(a,b).
\]
For $c\in\mathbb F_7$, let $[c]_7\in\{0,1,\ldots,6\}$ denote its least
nonnegative integer representative.  The \(G_1=SL_{3,1}\)-block of a
restricted highest weight is determined by the nondecreasing rearrangement
\[
 \kappa(a,b)=\operatorname{nord}
 \bigl([2a+b+3]_7,[-a+b]_7,[-a-2b-3]_7\bigr),
\tag{6.S7a}\label{eq:S3-block-key}
\]
where $\operatorname{nord}$ denotes nondecreasing order on the chosen
representatives $\{0,1,\ldots,6\}$.
For the first tensor factor, Pieri gives
\[
 [A_{a,b}]=[\Delta(a+1,b)]
 +\mathbf 1_{a>0}[\Delta(a-1,b+1)]
 +\mathbf 1_{b>0}[\Delta(a,b-1)].
\tag{6.S7b}\label{eq:S3-pieri-A}
\]
For the second, identify \((a,b)\) with the \(GL_3\)-partition
\((a+b,b,0)\).  The identity
\[
 s_{(3,1)}=h_3h_1-h_4
\tag{6.S7c}\label{eq:S3-pieri-B}
\]
reduces \(B_{a,b}\) to two horizontal-strip Pieri expansions.  In the
first upper alcove one first uses
\[
 [L(a,b)]=[\Delta(a,b)]-[\Delta(5-b,5-a)].
\tag{6.S7d}\label{eq:S3-upper-reflection}
\]
Applying \eqref{eq:S3-block-key} to these Weyl factors gives the
outer/common decomposition.

For the twelve bottom targets, grouped by \(a+b=2,3,4\), the count is
\[
\begin{aligned}
\#\mathrm{outer}
 &= (3+3+3)+(3+4+4+3)+(2+3+4+4+3)\\
 &=9+14+16=39,\\
\#\mathrm{common}
 &= (2+3+2)+(2+3+3+2)+(2+3+3+3+2)\\
 &=7+10+13=30.
\end{aligned}
\]
Five outer blocks
have one of the three exceptional source labels, leaving \(34\) induced
outer blocks.  Among the common blocks the rank-one linkage calculation gives
\[
 12\text{ equal-simple}+6\text{ double-simple}
 +6\text{ rigid rank-one}+6\text{ mixed}.
\]
The twelve equal-simple quotients vanish and the remaining eighteen have a
single nonexceptional simple socle.  The bottom contribution is therefore
\[
 34+18=52.
\]

The fifteen first-upper targets lie on the five diagonals
\(a+b=6,7,8,9,10\).  On these diagonals the same calculation gives
\[
\begin{aligned}
\#\mathrm{outer}
 &= (3+3+4+3+2)+(4+5+5+3)+(5+6+5)+(5+5)+3\\
 &=15+17+16+10+3=61,\\
\#\mathrm{common}
 &=5\cdot2+4\cdot3+3\cdot3+2\cdot3+3=40.
\end{aligned}
\]
The forty common blocks split as
\[
\begin{aligned}
 40={}&8\text{ equal-simple}+9\text{ double-simple}
      +4\text{ pure rank-one tilting}\\
     &+8\text{ mixed}+11\text{ higher-rank or corner blocks}.
\end{aligned}
\tag{6.S7e}\label{eq:S3-upper-common-taxonomy}
\]
The eight equal-simple blocks vanish.  The rank-one translation, tilting
and upper-wall descriptions give one nonexceptional simple socle in each of
the remaining thirty-two common blocks and in each outer block.  This gives
\[
 61+32=93.
\]

For the singular targets it is convenient to use the three disjoint walls
\[
 a+b=5,\qquad a=6,\qquad b=6\ \ (a\le5).
\]
The three wall counts are
\[
\begin{array}{c|cc}
 &\#\mathrm{outer}&\#\mathrm{common}\\ \hline
 a+b=5&3+4+5+5+3+2&1+2+2+2+2+2\\
 a=6&3+4+5+5+4+2+2&1+2+2+2+2+2+1\\
 b=6,\ a\le5&2+3+5+5+4+3&2+2+2+2+2+2
\end{array}
\]
so they contribute respectively \((22,11),(25,12),(22,12)\), and hence
\[
 69\text{ outer}+35\text{ common}=104.
\]
The simple, rigid-tilting, and mismatched-socle cases settle all but eight
blocks.  Five of the remaining blocks are matched-socle pencils; on their
normalized highest lines the determinants are
\[
 3,\ 6,\ 6,\ 6,\ 1\in\mathbb F_7^\times.
\]
The other three are the corresponding Steinberg projective covers.  Their
socles, together with the blocks already settled, give \(96\) singular
lines.  Consequently
\[
 52+93+96=241,
\]
and each surviving induced--induced class has multiplicity one.

At \(q=2\), Jacobi gives \(\Theta_UJ_U=0\).  Put
\[
 D_U=\operatorname{coker}J_U,
 \qquad
 \bar\Theta_U:D_U\longrightarrow
 (L(5,0)\oplus L(2,3)\oplus L(1,2)\oplus L(0,1))\otimes U.
\]
Then \(H^1_{*,2}(N,Z(U))=\ker\bar\Theta_U\).  A nonzero submodule of
\(\ker\bar\Theta_U\) contains a simple submodule, so injectivity of
\(\bar\Theta_U\) on \(\operatorname{soc}D_U\) implies injectivity on
\(D_U\).

For the twelve bottom targets, the highest-weight and lowering relations
in \(D_U\) give zero for
\[
 \mu=(0,2),(1,1),(3,1).
\]
Five further nonzero targets have exceptional source socle:
\[
\begin{gathered}
 (0,3):L(0,1),\qquad (1,2):L(1,0),\qquad
 (1,3):L(0,0),\\
 (2,1):L(0,0),\qquad (2,2):L(0,1).
\end{gathered}
\]
The remaining four nonzero targets have socles
\[
 (0,4):L(0,2),\qquad (2,0):L(4,0),\qquad
 (3,0):L(5,0),\qquad (4,0):L(6,0),
\]
and hence give the first four induced--induced arrows in
\eqref{eq:S3-q2-list}.

For a first-upper target the coefficient sequence
\[
 0\longrightarrow L(\ell(\mu))\longrightarrow\Delta(\mu)
 \longrightarrow L(\mu)\longrightarrow0
\]
is termwise exact on the second-jet complex.  The lower-simple map to the
Weyl coefficient is an isomorphism on middle homology, leaving the
singular-vector trace kernel.  The trace is injective on the socle for
\[
 (2,5),(3,4),(3,5),(4,3),(4,4),(5,3),(5,4),(5,5).
\]
The target rows \((4,5)\) and \((5,2)\) have only exceptional-source
socle.  The part with induced source therefore consists of the five arrows
with targets
\[
 (1,5),(2,4),(3,3),(4,2),(5,1),
\]
listed in \eqref{eq:S3-q2-list}.

On a singular target one has \(L(\mu)=\Delta(\mu)=T(\mu)\), so the same
socle test applies directly to \(D_U\).  The only nonzero second-jet
homology modules occur at
\[
\begin{aligned}
 (0,5)&:L(0,3),&\qquad (5,0)&:\Delta(7,0),\\
 (0,6)&:L(0,4),& (6,0)&:\Delta(8,0).
\end{aligned}
\]
The other fifteen singular rows are killed by socle injectivity, while
\[
 \operatorname{soc}\Delta(7,0)=L(5,1),\qquad
 \operatorname{soc}\Delta(8,0)=L(4,2).
\]
These are precisely the last four arrows in \eqref{eq:S3-q2-list}.  Thus
the induced--induced second-jet contribution is \(4+5+4=13\), with
multiplicity one in each case.  Applying the same
relations to the three exceptional targets gives
\[
 H^1_{*,2}(N,L_i)=0\qquad(i=0,1,2):
\tag{6.S11a}\label{eq:S3-q2-exceptional-zero}
\]
for \(L_1\) the kernel and inner dimensions are \(10,10\), and for \(L_2\)
they are \(15,15\); the trivial target is zero because
\(N/[N,N]\cong S_1\).  Thus degree two contributes no further extensions
involving an exceptional module.

The positive Special algebra is generated by \(S_1\).  Indeed the
divergence-free monomials
\[
 D_{ab}(x^{(\alpha)})
 =D_b(x^{(\alpha)})D_a-D_a(x^{(\alpha)})D_b
\]
span \(S_j\) for \(|\alpha|=j+2\).  Lexicographic elimination brackets an
unsaturated leading monomial with a quadratic element of \(S_1\) with
coefficient \(\pm1\); at the unique fully saturated exponent, the sum of
the two choices of the auxiliary index cancels the lower term.  Hence
\([S_1,S_{j-1}]=S_j\) for \(2\le j\le16\).
At \(q=3\) the
third-jet term is the middle homology of
\[
 S^3V\otimes U\longrightarrow
 \Hom(S_1,S^2V\otimes U)\longrightarrow\mathcal R_3(U),
\tag{6.S11b}\label{eq:S3-third-jet}
\]
where \(\mathcal R_3(U)\) consists of the quadratic bracket-kernel row and
the \(S_1\otimes S_2\to S_3\) mixed row.  On the twelve bottom targets,
the harmonic highest-line calculation leaves only two natural primitive
lines.  On the first-upper weights, normalized raising along the affine
highest-root string contracts the primitive complex; on the affine facet
and coordinate walls the two transverse simple-root strings do so.  The
Steinberg tensor ideal contracts the remaining corner.  Consequently the
third-jet curvature has only two nonzero induced targets,
both isomorphic to \(L(1,0)\), at targets \((2,1)\) and \((1,3)\); these
belong to the de Rham resolution and yield no induced source vertex.
For \(4\le q\le18\), write the homogeneous derivation complex as
\[
 C_q^0(U)\xrightarrow{J_q}C_q^1(U)
 \xrightarrow{\Xi_q}C_q^2(U).
\]
Sum the three coordinate Euler--Spencer polarizations on the
first-component tensors and on the relation complex.  If
\(\bar h_q,\bar k_q\) are the descended operators, then
\[
 J_q\bar h_q+\bar k_q\Xi_q=(q-3)\operatorname{id}.
\tag{6.S11c}\label{eq:S3-Cartan-contraction}
\]
The identity descends through the trace relation because a tensor
\((e_ie_j\otimes\varepsilon^a)\otimes D^\beta\), \(|\beta|=q-1\),
has total polarization scalar
\(\sum_r(\beta_r-\delta_{ri}-\delta_{rj})=q-3\).
It preserves every successor kernel by the differentiated bracket identity
\[
 D_r[x,y]=[D_rx,y]+[x,D_ry].
\]
Equation \eqref{eq:S3-Cartan-contraction} therefore contracts all these degrees except
\(q=10,17\).

At either resonance, pass to the dual divergence-kernel complex.  On a
nonzero torus-weight strand write a quadratic vector field as
\(X_a(u)=x^{u+e_a}D_a\), pair it with \(D^\beta\), and put
\(\lambda=\beta-u\in V^*_{\mathbb F_7}\).  Then
\[
 \lambda_1+\lambda_2+\lambda_3=q-2\equiv1\pmod7,
\]
so \(\lambda\ne0\).  Choose \(v\) with \(\lambda(v)=1\); the Koszul identity
\[
 \iota_v(\lambda\wedge-)+\lambda\wedge\iota_v=\operatorname{id}
\]
contracts the strand.  It commutes with the coefficient de Rham
differential and hence preserves the divergence kernel and the exact layer
at \(q=17\).  The last successor is surjective: its only possible quotient
is a quotient of the simple \(45\)-dimensional module
\(S_1\otimes V^{(1)}\), whereas the successor has dimension \(35\).
The resonant strands are also contractible, so the induced--induced
cohomology vanishes for \(4\le q\le18\).

At \(q=12\), the sole possible
adjoint block is the map
\[
 \Psi_{12}:\Hom(S_1,S_5)\longrightarrow\Hom(K_2,S_6),
 \qquad
 \Psi_{12}(F)(x\wedge y)=[x,F(y)]-[y,F(x)],
\]
where \(K_2=\ker(\Lambda^2S_1\to S_2)\).  The monomial bracket formula
\([x^aD_i,x^bD_j]=b_i x^{a+b-e_i}D_j-a_jx^{a+b-e_j}D_i\)
gives
\[
 \ker\Psi_{12}=\{x\mapsto[x,z]:z\in S_4\},
\]
both sides having dimension \(48\).  More precisely, the weight blocks
have the following column/rank/kernel/inner-rank types:
\[
\begin{array}{c|r}
(c,r,k,i)&\text{number of blocks}\\ \hline
(1,1,0,0)&6\\
(2,2,0,0)&24\\
(6,6,0,0)&6\\
(8,8,0,0)&18\\
(15,14,1,1)&6\\
(18,17,1,1)&12\\
(24,22,2,2)&3\\
(27,25,2,2)&9\\
(30,28,2,2)&3
\end{array}.
\tag{6.S12a}\label{eq:S3-q12-block-ranks}
\]
The kernel and the inner image therefore agree weight by weight, and no
\(q=12\) class survives.

At \(q=19\), let
\(\mathcal D_{19,\bullet}(U^*)\) be the finite dual
Chevalley--Eilenberg reconstruction complex, filtered increasingly by the
total weight \(w\) of its source letters.  The filtration is finite,
exhaustive and separated, and the differential preserves it.  Its
associated-graded differential is the trivial-coefficient bracket
boundary.  If \(R_d^-\) denotes the homogeneous degree-\(d\) piece of the
negative PBW algebra, the resulting first-quadrant spectral sequence is
\[
 E^1_{w,r}=H_r(N,\kk)_w\otimes R^-_{w-1}\otimes U^*
 \quad\Longrightarrow\quad
 H_r\bigl(\mathcal D_{19,\bullet}(U^*)\bigr).
\tag{6.S12b}\label{eq:S3-q19-E1}
\]
The shift \(w-1\) comes from the apolar identification
\(R_d^-\cong(R^-_{18-d})^*\): a source of weight \(w\) acts on target
depth \(18-(w-1)\).

The generation calculation above gives
\([S_1,S_{j-1}]=S_j\) for \(2\le j\le16\).  Hence
\[
 H_1(N,\kk)=N/[N,N]\cong S_1,
 \qquad
 E^1_{1,1}=S_1\otimes U^*,
 \qquad E^1_{w,1}=0\quad(w>1).
\tag{6.S12c}\label{eq:S3-q19-H1}
\]
Every filtration quotient of total homological degree one therefore occurs on
the single source-weight-one edge.  No differential leaves that edge, and
the differentials entering it are the relation transgressions
in the source-weight-one presentation.

Dualizing this presentation gives pairs
\(\phi,\psi:S_1\to U\) satisfying
\[
 a\phi(z)=\psi(az)
 \qquad(a\in S_0,\ z\in S_1).
\tag{6.S12d}\label{eq:S3-q19-factorization}
\]
Put \(D=\psi-\phi\).  The relation equations give
\(aD(bz)=bD(az)\).  Every torus weight of \(S_1\) is nonzero: in the
trace-free polynomial realization it is represented by an exponent
difference \(\gamma=\beta-e_i\), with
\(\gamma_j\in\{-1,0,1,2\}\) and
\(\gamma_1+\gamma_2+\gamma_3=1\); congruence of all three coordinates
modulo seven would therefore force equality over the integers, which is
impossible.

Choose \(0\ne z\in(S_1)_\lambda\).  Applying
\(aD(bz)=bD(az)\) to two Cartan elements shows that every nonzero weight
component of \(Dz\) lies on the line \(\kk\lambda\).  If
\(e_\alpha\) is a root vector with \(\alpha\) not parallel to \(\lambda\),
the same identity, first with a Cartan element and \(e_\alpha\), and then
on \(e_\alpha z\), places every nonzero component of
\(e_\alpha Dz\) simultaneously on the distinct lines
\(\kk\lambda\) and \(\kk(\lambda+\alpha)\).  Hence
\(e_\alpha Dz=0\).  The root spaces not parallel to \(\lambda\) generate
\(S_0\cong\mathfrak{sl}_3\): if \(\lambda\) is parallel to one root, the
four root spaces in the other two directions recover the missing pair by
brackets.  Hence \(Dz\in U^{S_0}=0\) for a nontrivial simple target, and
therefore \(D=0\) and \(\phi\) is \(S_0\)-linear.  Thus
\[
 E^2_{1,1}(U^*)\cong\Hom_{S_0}(S_1,U)^*.
\tag{6.S12e}\label{eq:S3-q19-edge}
\]
It is zero for forty-five targets.  For \(U=S_1=L(1,2)\) it is the
one-dimensional trivial \(S_0\)-module.

To identify this edge with the restricted abutment, let \(f\) be an
ordinary homogeneous derivation and put
\(\Phi_f(x)=f(x^{[7]})-x^6f(x)\).  Comparing seventh powers in the graph
extension \(N\ltimes Z(U)\) shows that
\(\Phi_f(x)\in Z(U)^N\).  By the invariant-top lemma
\(Z(U)^N=U=Z(U)_0\), whereas for \(x\in S_j\) homogeneity places
\(\Phi_f(x)\) in PBW depth \(19-7j\), namely \(12\), \(5\), or a
negative depth.  The intersection is zero, so every class represented on
the edge is restricted.

Let \(\mathfrak J\) be the augmentation ideal of \(u_0(N)\).  Pull back
\[
 0\longrightarrow\mathfrak J/\mathfrak J^2\cong S_1
 \longrightarrow u_0(N)/\mathfrak J^2
 \longrightarrow\kk\longrightarrow0
\]
along the PBW top line spanned by
\(t=D_1^6D_2^6D_3^6\).  The negative generators annihilate \(t\), and \(S_0\)
acts trivially on it.  For \(x\in S_1\), normal ordering gives
\[
 x\cdot(D^{(6,6,6)}\otimes u)
 =\pm\sum_{i=1}^3D^{(6,6,6)-e_i}
 \otimes\rho_U([x,D_i])u.
\tag{6.S12f}\label{eq:S3-q19-top-action}
\]
This is zero for the scalar inducing module; since \(S_1\) generates
\(N\), the line \(\kk t\) is a trivial \(S\)-submodule.  The pullback
therefore gives a restricted extension
\(0\to Z(S_1)\to\mathcal X\to L_0\to0\), whose connecting cocycle is
\[
 c_j(x)=\sum_{|\beta|=j-1}
 \binom{(6,6,6)}{\beta}D^{(6,6,6)-\beta}
 \otimes R_1^{\beta_1}R_2^{\beta_2}R_3^{\beta_3}(x),
 \qquad R_i(x)=[x,D_i].
\tag{6.S13}\label{eq:S3-q19-cocycle}
\]
Its first component is \(D_1^6D_2^6D_3^6\otimes x\), so it represents
the generator of \eqref{eq:S3-q19-edge}.  There is no degree-\(19\) inner
parameter because the coefficient module has PBW depth eighteen.  Since
the formula is a restricted cocycle, its class is a
permanent cycle; \eqref{eq:S3-q19-H1} leaves no other filtration piece in
degree one.  Thus
\[
 H^1_{*,19}(N,Z(S_1))\cong\kk,
\]
and the other forty-five degree-\(19\) groups vanish.  For \(q\ge20\) the
first component lies beyond PBW depth eighteen, and \(S_1\)-generation
forces the entire cocycle to vanish.  Thus the only class in the final degree is
\(L_0\to Z(1,2)\) in \eqref{eq:S3-standard-exceptional}.

For the degree-zero part of the $L_0$ row, read Parker's incidence table
before deleting the inducing weight $(0,0)$.  Besides the scalar pair with
$L(5,5)$, it gives
\[
 \dim\Ext^1_{u_0(S_0)}(L(1,5),\kk)
 =\dim\Ext^1_{u_0(S_0)}(L(5,1),\kk)=3
\]
and the same dimensions in the reverse orientation.  Restricted Shapiro
identifies the first two spaces with the degree-zero groups
$\Ext^1(Z(1,5),L_0)$ and $\Ext^1(Z(5,1),L_0)$.  Contravariant duality gives
the reverse ungraded dimensions, and the positive-degree exhaustion above
contains neither reverse ordered pair.  The reverse classes therefore lie
in degree zero.  These are the four dimension-three pairs in
\eqref{eq:S3-standard-exceptional}.

The de Rham exact sequences and apolar duality give the remaining pairs
of \eqref{eq:S3-standard-exceptional}; \eqref{eq:S3-q2-exceptional-zero}
excludes the only ambiguous second-jet placements.  For $L_1$, the
homogeneous pairings
\[
 (L_1)_d^*\cong(L_1)_{16-d},\qquad
 Z(U)_d^*\cong Z(U^*)_{18-d}
\]
send the two incoming first-jet classes to the two outgoing classes in
degree three.  Its inducing address is $(1,0)$, so the degree-zero entries
in both orientations are the three $R_3$ Levi edges.
For $L_2$, the inducing address $(0,1)$ belongs to $R_1$; hence the
degree-zero strand in both orientations consists of the three $R_1$ Levi
edges.  Let $Z^1=\ker(d:\Omega^1\to\Omega^2)$.  The exact sequence
\[
 0\longrightarrow Z^1\longrightarrow\Omega^1
 \xrightarrow{d}L_2\longrightarrow0
\]
has $(Z^1)_0=(\Omega^1)_0=L(0,1)$ and identifies the positive restriction
kernel with the first- and second-jet kernels.  These give
$L_2\to Z(1,3),Z(2,1)$ in degree one and
$L_2\to Z(0,3),Z(2,2)$ in degree two.  In the reverse orientation they
give $Z(3,0),Z(2,2)\to L_2$ in degree one and
$Z(3,1),Z(1,2)\to L_2$ in degree seven.  Every positive-degree entry has
multiplicity one.  The internal degrees are those of the graded de Rham
sequence.

The four exceptional
cross-arrows of dimension four are the two orientations in
\cref{prop:S3-trivial-exceptional-row}.  The \(L_1\)--\(L_2\) groups vanish
from the de Rham source sequence.  Indeed the sole relative Levi line,
evaluated on \(X=x_1D_2\), \(Y=x_2^2D_1\), gives the two mixed-bracket
columns
\[
\begin{array}{c|cc}
 &x^a&x^b\\ \hline
X\mathcal A_Y-Y\mathcal A_X&D_1&0\\
\mathcal A_{[X,Y]}&-D_1&-x_2D_1,
\end{array}
\]
which are not proportional.  Hence it is not a Special cocycle.

It remains to evaluate the three exceptional loops.  The loop at $L_0$ is
zero because $\mathfrak s$ is perfect.  Put $N=S_{>0}$,
\[
 V_1=\Lambda^1(k^3)^*\cong L(0,1),
 \qquad
 V_2=\Lambda^2(k^3)^*\cong L(1,0),
\]
and let
\[
 u(N)\otimes V_i\twoheadrightarrow L_i,
 \qquad
 R_i=K_i/NK_i
\]
be the minimal positive presentation.  The de Rham covers
$\Omega^2\twoheadrightarrow L_1$ and $\Omega^1\twoheadrightarrow L_2$
have kernels with composition factors respectively
$L_2\oplus k^{\oplus4}$ and $L_1\oplus k^{\oplus4}$.  The cross Hom
vanishing above therefore gives injections
\[
 \Ext^1(L_1,L_1)\hookrightarrow\Ext^1(\Omega^2,L_1),
 \qquad
 \Ext^1(L_2,L_2)\hookrightarrow\Ext^1(\Omega^1,L_2).
\tag{6.S10a}\label{eq:S3-diagonal-cover-injections}
\]
Restricted Shapiro reduction and duality identify their positive parts with
$\Hom_{S_0}(R_i,V_i)^*$.  The determinant/Frobenius address equation is
\[
 3+r=3c+7k,
 \qquad k\in\{0,1,2\},
\]
so only
\[
 r\in\{3,4,6,7,9,10,11,12,14,15,17,18\}
\tag{6.S10b}\label{eq:S3-diagonal-addresses}
\]
can contribute.

For the rank-three relation head, write
\[
 \mathcal C^{1}_{i,r}
 =\Hom_{S_0}(S_1\otimes(L_i)_{r-1},V_i),
 \qquad
 \mathcal C^{2}_{i,r}
 =\Hom_{S_0}(S_2\otimes(L_i)_{r-2},V_i).
\]
Using
\[
 (x^{(a)}D_j)x^{(b)}
 =\binom{a+b-\varepsilon_j}{a}x^{(a+b-\varepsilon_j)}
\]
and the analogous two-term bracket formula, torus weights followed by the
four root operators of $\mathfrak{sl}_3$ give the following census.  The same calculation follows from the
rank-three good-filtration Hom formula \cite[Theorem~5.13]{ParkerGL3} applied to the
truncated pieces:
\[
 \mathcal C^{1}_{i,r}=\mathcal C^{2}_{i,r}=0
\]
for every address in \eqref{eq:S3-diagonal-addresses}, except
\[
 \dim\mathcal C^{1}_{2,3}=1,
 \qquad
 \dim\mathcal C^{1}_{2,15}
 =\dim\mathcal C^{2}_{2,15}=1.
\tag{6.S10c}\label{eq:S3-diagonal-first-cochains}
\]
At $r=3$ the unique first-cochain line is alternating, its companion
$\mathcal C^2$ is zero, and $S_0$-equivariance makes its first relation row
nonzero.

At $r=15$, let $\phi$ and $\psi$ span the two spaces in
\eqref{eq:S3-diagonal-first-cochains}.  Set
\[
\begin{aligned}
 X&=x_3^{(2)}D_1,&
 Y&=x_1x_3D_1-x_3^{(2)}D_3,\\
 v_1&=x_1^{(6)}x_2^{(5)}x_3^{(2)}D_1
      -x_1^{(5)}x_2^{(5)}x_3^{(3)}D_3,\\
 v_2&=x_1^{(5)}x_2^{(6)}x_3^{(2)}D_2
      -x_1^{(5)}x_2^{(5)}x_3^{(3)}D_3.
\end{aligned}
\]
Normalize by $\phi(X,Yv_1)=D_1$ and
$\psi([X,Y],v_1)=2D_1$.  Direct divided-power evaluation gives
\[
\begin{array}{c|ccc}
 &\phi(Y,Xv)&\phi(X,Yv)&\psi([X,Y],v)\\ \hline
v_1&6&1&2\\
v_2&6&3&3.
\end{array}
\]
Thus the first relation
\[
 \phi(Y,Xv)-\phi(X,Yv)-\psi([X,Y],v)
\]
has coefficient matrix
\[
 \begin{pmatrix}5&5\\3&4\end{pmatrix},
 \qquad \det=5\ne0\quad\text{in }\mathbb F_7.
\]
Both relation heads therefore have zero fundamental multiplicity.  The
degree-zero Levi terms vanish as well, and
\[
 \Ext^1(L_1,L_1)=\Ext^1(L_2,L_2)=0.
\tag{6.S10d}\label{eq:S3-exceptional-diagonal-zero}
\]

Let $\mathscr A_q$ be the set of
induced--induced ordered pairs in degree $q$.  The
bottom/first-upper/singular block decomposition in the first-jet
calculation, compared with
\cref{eq:S3-q0-regular,eq:S3-q0-singular} and \eqref{eq:S3-q2-list}, gives
\[
 \mathscr A_0\cap\mathscr A_1=\varnothing,\qquad
 \mathscr A_0\cap\mathscr A_2=\varnothing,\qquad
 \mathscr A_1\cap\mathscr A_2=\varnothing.
\]
The same list has no diagonal pair.  Consequently the four contributions to the
count below are disjoint, and
\[
 96+254+32+4=386,\qquad
 240+254+56+16=566,
\]
which proves \eqref{eq:S3-global-count} and loop-freeness.
\end{proof}

\begin{corollary}[Second Loewy layers and Cartan lower bounds]
\label{cor:all-PIM-Loewy}
Let $L_\alpha,L_\beta$ be simple restricted modules in any of the four
families, and let $P_\alpha$ be the projective cover of $L_\alpha$.  Then
\[
 \left[\frac{\rad P_\alpha}{\rad^2P_\alpha}:L_\beta\right]
 =\dim_\kk\Ext^1_{u_0(\mathfrak g)}(L_\alpha,L_\beta),
\]
and consequently
\[
 [P_\alpha:L_\beta]
 \ge \dim_\kk\Ext^1_{u_0(\mathfrak g)}(L_\alpha,L_\beta).
\]
For the projective cover $P(\kk)$ of the trivial module, the closed
exceptional formulas give, in the four families,
\[
 [P_W(\kk):M_a]\ge \binom n{a+1}+\delta_{a,0}
 \qquad(0\le a<n),
\tag{6.4W}\label{eq:W-Cartan-lower}
\]
\[
 [P_S(\kk):L_i]\ge
 \begin{cases}
 n+1,&i=1\text{ or }i=n-1,\\
 \binom ni,&2\le i\le n-2,
 \end{cases}
\tag{6.4S}\label{eq:S-Cartan-lower}
\]
\begin{align}
 [P_H(\kk):M_1]&\ge2m+2,\nonumber\\
 [P_H(\kk):M_s]&\ge
 \binom{2m+1}{s}
 +\mathbf1_{\{p\mid(m-s+2)\}}\binom{2m}{s-1}
 +2\epsilon_s,
 \qquad 2\le s\le m.
\tag{6.4H}\label{eq:H-Cartan-lower-general}
\end{align}
and, on the folded Contact exceptional vertices,
\[
 [P_K(\kk):L_r^-]\ge R_r^-,
 \qquad
 [P_K(\kk):L_s^+]\ge D_{s-1}^-.
\tag{6.4K}\label{eq:K-Cartan-lower}
\]
\end{corollary}

\begin{proof}
Pass to the split basic algebra $B$ with primitive idempotents $e_\gamma$
and radical $J$.  For left modules,
\[
 \left[\frac{\rad Be_\alpha}{\rad^2Be_\alpha}:L_\beta\right]
 =\dim_\kk e_\beta(J/J^2)e_\alpha
 =\dim_\kk\Ext^1_B(L_\alpha,L_\beta).
\]
Morita equivalence preserves projective covers, radical layers,
composition multiplicities, and Yoneda extensions.  The full Cartan
multiplicity contains every Loewy layer, which gives the inequality.  The
four displayed bounds follow by substituting the outgoing rows from the
trivial module
of \cref{thm:SW-exceptional-closed,thm:H-trivial-exceptional,thm:contact-exceptional-closed}.
\end{proof}

\section*{Acknowledgements}

I am grateful to J\"org Feldvoss for many helpful comments, and to
Alexander Premet for helpful correspondence.

\end{document}